\documentclass{article}

\usepackage{amsmath,amsthm,amssymb}
\usepackage{enumerate}
\usepackage{graphicx}
\usepackage[noadjust]{cite}
\usepackage{comment}
\usepackage{oands}
\usepackage{tikz}
\usepackage{changepage}
\usepackage{bbm}
\usepackage{mathtools}
\usepackage[margin=1.2in]{geometry}
\usepackage{hyperref}
\usepackage{authblk}
\usepackage[pagewise,mathlines]{lineno}
\usepackage{appendix}
\usepackage{multicol}
\usepackage{enumitem}
\usepackage[normalem]{ulem}

\newcommand{\notion}[1]{{\emph{#1}}}

\theoremstyle{plain}
\newtheorem{thm}{Theorem}[section]
\newtheorem{cor}[thm]{Corollary}
\newtheorem{lem}[thm]{Lemma}
\newtheorem{definition}[thm]{Definition}
\newtheorem{prop}[thm]{Proposition}

\def\@rst #1 #2other{#1}
\newcommand\MR[1]{\relax\ifhmode\unskip\spacefactor3000 \space\fi
  \MRhref{\expandafter\@rst #1 other}{#1}}
\newcommand{\MRhref}[2]{\href{http://www.ams.org/mathscinet-getitem?mr=#1}{MR#2}}

\newcommand{\arxiv}[1]{\href{http://arxiv.org/abs/#1}{#1}}

\hypersetup{
    colorlinks=false,
    linktocpage,
    }

\theoremstyle{definition}
\newtheorem{defn}[thm]{Definition}
\newtheorem{remark}[thm]{Remark}

\numberwithin{equation}{section}

\newcommand{\dsb}{\begin{adjustwidth}{2.5em}{0pt}
\begin{footnotesize}}
\newcommand{\dse}{\end{footnotesize}
\end{adjustwidth}}

\newcommand{\ssb}{\begin{adjustwidth}{2.5em}{0pt}}
\newcommand{\sse}{\end{adjustwidth}}

\newcommand{\aryb}{\begin{eqnarray*}}
\newcommand{\arye}{\end{eqnarray*}}
\def\alb#1\ale{\begin{align*}#1\end{align*}}
\def\allb#1\alle{\begin{align}#1\end{align}}
\newcommand{\eqb}{\begin{equation}}
\newcommand{\eqe}{\end{equation}}
\newcommand{\eqbn}{\begin{equation*}}
\newcommand{\eqen}{\end{equation*}}
\newcommand{\defeq}{:=}

\newcommand{\initial}{\cC_{\op{initial}}}
\newcommand{\tri}{{\op {tri}}}
\newcommand{\diam}{{\op {diam}}}
\newcommand{\area}{{\op {area}}}
\newcommand{\reg}{{\op {reg}}}
\newcommand{\GHPUL}{{\op {GHPUL}}}
\newcommand{\BB}{\mathbbm}
\newcommand{\ol}{\overline}

\newcommand{\op}{\operatorname}

\newcommand{\frk}{\mathfrak}
\newcommand{\eqD}{\overset{d}{=}}
\newcommand{\ep}{\epsilon}
\newcommand{\rd}{\mathrm{d}}
\newcommand{\rta}{\rightarrow}

\newcommand{\wt}{\widetilde}
\newcommand{\wh}{\widehat} 
\newcommand{\mcl}{\mathcal}

\newcommand{\bdy}{\partial} 
\newcommand{\rng}{\mathring}

\newcommand{\el}{\ell}
\newcommand{\elll}{\ell_{\op L}}
\newcommand{\ellr}{\ell_{\op R}}
\newcommand{\LL}{{\mathfrak l}}
\newcommand{\mrb}{\mathrm{b}}
\newcommand{\mry}{\mathrm{r}} 
\newcommand{\mrm}{\mathrm{m}}   
\newcommand{\ans}{\mathrm{AnFr}}   
\newcommand{\ansd}{\mathrm{anfr}}   
      
\newcommand{\type}{\mathrm{type}}
\def\bpv{\mathbf{piv}} 

\newcommand{\Cardy}{ \op {Cardy} }
  
\newcommand{\dcon}{{\mathbbm a}}
\newcommand{\mcon}{{\mathbbm b}}
\newcommand{\bcon}{{\mathbbm c}}
\newcommand{\tcon}{{\mathbbm s}}
\newcommand{\mas}{{\mathrm{a}}}
 
\newcommand{\Map}{{\mathsf{M}}}
\newcommand{\PBT}{{\mathrm{PBT}}}

\let\originalleft\left
\let\originalright\right
\renewcommand{\left}{\mathopen{}\mathclose\bgroup\originalleft}
\renewcommand{\right}{\aftergroup\egroup\originalright}

\newcommand{\BD}{\mathrm{BD}}
\def\cZ{\mathcal{Z}}

\def\cV{\mathcal{V}}

\def\cR{\mathcal{R}}

\def\cN{\mathcal{N}}

\def\cL{\mathcal{L}}
\def\cK{\mathcal{K}}

\def\cI{\mathcal{I}}

\def\cF{\mathcal{F}}
\def\cE{\mathcal{E}}

\def\cC{\mathcal{C}}
\def\cB{\mathcal{B}}

\def\BM{\BB M}
\newcommand{\C}{\mathbb{C}}
\newcommand{\R}{\mathbb{R}}
\newcommand{\N}{\mathbb{N}}

\newcommand{\Z}{\mathbb{Z}}
\newcommand{\D}{\mathbb{D}}
\renewcommand{\P}{\mathbb{P}}

\DeclareMathOperator{\SLE}{SLE}
\DeclareMathOperator{\CLE}{CLE}
 
\DeclareMathOperator{\Peel}{Peel}  
 
\newcommand{\DP}{\mathfrak{P}}
\newcommand{\DPm}{\mathfrak{P}^{\mathrm{m}}}

\newcommand{\GHPU}{{\mathrm{GHPU}}}
\newcommand{\bk}{\BB k}
\newcommand{\Be}{\BB e}
\newcommand{\fh}{\mathfrak h}

\newcommand{\Cut}{\mathrm{Cut}}

\title{Joint scaling limit of site percolation on random triangulations in the metric and peanosphere sense}
\date{  }
\author{Ewain Gwynne, Nina Holden, and Xin Sun}

\begin{document}

\maketitle

\begin{abstract}
Recent works have shown that random triangulations decorated by critical ($p=1/2$) Bernoulli site percolation converge in the scaling limit to a $\sqrt{8/3}$-Liouville quantum gravity (LQG) surface (equivalently, a Brownian surface) decorated by SLE$_6$ in two different ways:
\begin{itemize}
\item The triangulation, viewed as a curve-decorated metric measure space equipped with its graph distance, the counting measure on vertices, and a single percolation interface converges with respect to a version of the Gromov--Hausdorff topology.  
\item There is a bijective encoding of the site-percolated triangulation by means of a two-dimensional random walk, and this walk converges to the correlated two-dimensional Brownian motion which encodes SLE$_6$-decorated $\sqrt{8/3}$-LQG via the mating-of-trees theorem of Duplantier-Miller-Sheffield (2014); this is sometimes called \emph{peanosphere convergence}. 
\end{itemize} 
We prove that one in fact has \emph{joint} convergence in both of these two senses simultaneously. 
We also improve the metric convergence result by showing that the map decorated by the full collection of percolation interfaces (rather than just a single interface) converges to $\sqrt{8/3}$-LQG decorated by CLE$_6$ in the metric space sense. 

This is the first work to prove simultaneous convergence of any random planar map model in the metric and peanosphere senses.
Moreover, this work is an important step in an ongoing program to prove that random triangulations embedded into $\mathbb C$ via the so-called \emph{Cardy embedding} converge to $\sqrt{8/3}$-LQG.
\end{abstract}

\section{Introduction}
\label{sec-intro}

\subsection{Overview}
\label{sec-overview}

A planar map is a planar graph (multiple edges and self-loops allowed)
embedded into the two-dimensional sphere, viewed modulo orientation-preserving homeomorphisms of the sphere.   
Starting in the 1980's, physicists have interpreted random planar maps as the discrete analogs of random fractal surfaces called \emph{Liouville quantum gravity (LQG)} surfaces with parameter $\gamma \in (0,2)$ (see \cite{shef-kpz,nakayama-lqg} and the references therein). 
Heuristically speaking, if $D\subset \BB C$ and $h$ is some variant of the Gaussian free field (GFF) on $D$, then for $\gamma \in (0,2)$ the  $\gamma$-LQG surface corresponding to $(D,h)$ 
is the random two-dimensional Riemannian manifold with Riemannian metric tensor $e^{\gamma h} \, (dx^2 +dy^2)$. 
The parameter $\gamma$ depends on the particular type of random planar map model under consideration. 
Uniform random planar maps --- including uniform maps with local constraints like triangulations and quadrangulations ---  correspond to $\gamma=\sqrt{8/3}$, which is sometimes called ``pure gravity". This case will be our main interest in this paper.
Other values of $\gamma$ correspond to random planar maps weighted by the partition function of a statistical mechanics model, such as the Ising model ($\gamma=\sqrt 3$) or the uniform spanning tree ($\gamma=\sqrt 2$). 

The above definition of an LQG surface does not make literal sense since $h$ is only a distribution, not a function, so cannot be exponentiated.
However, one can rigorously define various aspects of LQG surfaces using regularization procedures.
For example, one can define the \emph{$\gamma$-LQG area measure} $\mu_h $ on $D$ as a limit of regularized versions of ``$e^{\gamma h(z)} \,d^2z$", where $d^2z$ denotes Lebesgue measure~\cite{shef-kpz,kahane,rhodes-vargas-review,dkrv-lqg-sphere}.
In the special case when $\gamma=\sqrt{8/3}$, one can also define $\sqrt{8/3}$-LQG as a random metric space~\cite{lqg-tbm1,lqg-tbm2,lqg-tbm3,miermont-brownian-map,legall-uniqueness} (it is a major open problem to construct a metric on $\gamma$-LQG for $\gamma\not=\sqrt{8/3}$).

Mathematically, the statement that ``random planar maps are the discrete analog of LQG" means that the former should converge in the scaling limit to the latter as, say, the total number of edges of the map tends to $\infty$. 
For a number of natural random planar maps decorated by statistical mechanics models, various curves associated with the statistical mechanics model also converge in some sense to Schramm-Loewner evolution (SLE$_\kappa$) curves~\cite{schramm0} for $\kappa \in \{\gamma^2,16/\gamma^2\}$ which are independent from the limiting LQG surface.

There are three main ways to formulate the convergence of random planar maps (decorated by statistical mechanics models) to LQG surfaces (decorated by SLE curves). 
The main goal of the present paper is to show that random triangulations decorated by critical ($p=1/2$) Bernoulli site percolation converge \emph{jointly} to $\sqrt{8/3}$-LQG decorated by SLE$_6$ in two of these senses: metric convergence and peanosphere convergence.
This is a major step toward proving convergence in the third sense (embedding convergence), which will be accomplished in~\cite{hs-cardy-embedding}.  
Let us now briefly review the main types of convergence for random planar maps. 
\medskip

\noindent \textbf{Metric convergence.} One can view a planar map as a random metric measure space --- equipped with the counting measure on vertices and the graph distance --- and show convergence to an LQG surface --- equipped with its LQG area measure and LQG metric --- with respect to the Gromov--Hausdorff--Prokhorov topology, the natural analog of the Gromov--Hausdorff topology for metric measure spaces (see, e.g.,~\cite{adh-ghp}).
Presently, this type of convergence is well understood for uniform random planar maps, but not for random planar maps in the $\gamma$-LQG universality class for $\gamma\not=\sqrt{8/3}$. 

The main reason why distances in uniform random planar maps are tractable is the \emph{Schaeffer bijection}~\cite{schaeffer-bijection} and its generalization due to Bouttier-Di Francesco-Guitter~\cite{bdg-bijection}, which encode various types of uniform random planar maps by means of labeled trees, where the labels correspond to graph distances in the map. 
Using the Schaeffer bijection, it was shown independently by Le Gall~\cite{legall-uniqueness} and Miermont~\cite{miermont-brownian-map} that uniform quadrangulations converge in the scaling limit to a random metric measure space called the \emph{Brownian map}, a continuum metric measure space constructed via a continuum analog of the Schaeffer bijection. 
Subsequent works have extended this result to planar maps with different local constraints~\cite{miermont-brownian-map,legall-uniqueness,ab-simple,bjm-uniform,abraham-bipartite,beltran-legall-pendant} and planar maps with different topologies (such as the whole plane or disk, instead of the sphere)~\cite{bet-mier-disk,curien-legall-plane,gwynne-miller-uihpq,bmr-uihpq,gwynne-miller-simple-quad,aasw-type2}.
Particularly relevant to the present work is the paper~\cite{aasw-type2}, which shows that uniform triangulations of the disk of type II (i.e., multiple edges, but not self-loops, allowed) converge in the scaling limit to the \emph{Brownian disk}, the disk analog of the Brownian map which was constructed in~\cite{bet-mier-disk}.

In a series of works~\cite{lqg-tbm1,lqg-tbm2,lqg-tbm3}, Miller and Sheffield showed that one can define a metric on a $\sqrt{8/3}$-LQG surface (i.e., a metric on $D\subset\BB C$ induced by a GFF-type distribution on $D$). Moreover, it is shown in~\cite[Corollary 1.4]{lqg-tbm2} that certain special $\sqrt{8/3}$-LQG surfaces (corresponding to particular choices of $(D,h)$) are equivalent, as metric measure spaces, to Brownian surfaces such as the Brownian map and Brownian disk. 
We will review the background on Brownian and $\sqrt{8/3}$-LQG surfaces necessary to understand the present paper in Section~\ref{sec-sle-markov}.

For certain planar maps decorated by a curve, one can prove convergence to SLE-decorated LQG with respect to the \emph{Gromov--Hausdorff--Prokhorov--uniform (GHPU)} topology, the natural generalization of the Gromov--Hausdorff topology for curve-decorated metric measure spaces which was introduced in~\cite{gwynne-miller-uihpq}. 
For example, it was shown in~\cite{gwynne-miller-saw,gwynne-miller-perc} that a uniform random planar map decorated by a self-avoiding walk or a percolation interface converges in the scaling limit to a $\sqrt{8/3}$-LQG surface decorated by SLE$_{8/3}$ or SLE$_6$, respectively.

{We remark that the $\gamma$-LQG metric for general $\gamma \in (0,2)$ was very recently constructed in~\cite{gm-uniqueness}, building on~\cite{dddf-lfpp,local-metrics,lqg-metric-estimates,gm-confluence}; see also~\cite{gm-coord-change}. For $\gamma \not=\sqrt{8/3}$, it is conjectured, but not yet proven, that appropriate weighted random planar map models converge to $\gamma$-LQG surfaces equipped with this metric in the Gromov--Hausdorff sense.}
\medskip

\noindent\textbf{Peanosphere convergence.} Certain random planar maps decorated by statistical mechanics models can be encoded by random walks on $\BB Z^2$ (with increment distributions depending on the model). Some encodings of this type are called \emph{mating-of-trees bijections}   since  the bijection can be interpreted as gluing together, or ``mating", the discrete random trees associated with the two coordinates of the walk to construct the map. 
The simplest such bijection is the Mullin bijection~\cite{mullin-maps} (see~\cite{shef-burger,bernardi-maps} for more explicit expositions), which encodes a planar map decorated by a spanning tree by a  nearest-neighbor random walk on $\BB Z^2$.
Other mating-of-trees bijections are obtained in~\cite{shef-burger,kmsw-bipolar,gkmw-burger,bernardi-dfs-bijection,bhs-site-perc,lsw-schnyder-wood}. 
We will review the mating-of-trees bijection for site percolation on a uniform triangulation from~\cite{bernardi-dfs-bijection,bhs-site-perc} (which is the only such bijection used in the present paper) in Section~\ref{sec-walk}. 

In the continuum setting, Duplantier, Miller, and Sheffield~\cite{wedges} showed that a $\gamma$-LQG surface decorated by a space-filling variant of SLE$_\kappa$ for $\kappa = 16/\gamma^2$ can be encoded by a correlated two-dimensional Brownian motion, with the correlation of the two coordinates given by $-\cos(\pi\gamma^2/4)$, via an exact continuum analog of a mating-of-trees bijection. 
This result is sometimes called the \emph{peanosphere construction} since it implies that SLE-decorated LQG is homeomorphic to a random curve-decorated topological measure space called the \emph{peanosphere} which is constructed from two correlated Brownian motions.
For each of the mating-of-trees bijections discussed above, it can be shown that the walk on $\BB Z^2$ which encodes the decorated random planar map converges in the scaling limit to the correlated two-dimensional Brownian motion which encodes the SLE-decorated LQG. 
We interpret this as a scaling limit result for random planar maps in a certain topology --- namely, the one where two decorated ``surfaces" are close if their encoding functions are close.
Convergence with respect to this topology is called \emph{peanosphere convergence}.

Several extensions of peanosphere convergence are possible, giving convergence of a wide range of different functionals of the decorated random planar map to their continuum analogs. See~\cite{ghs-bipolar,gms-burger-cone,gms-burger-local,gms-burger-finite,lsw-schnyder-wood,bhs-site-perc}.
\medskip

\noindent \textbf{Embedding convergence.} 
There are several natural ways of embedding a planar map into $\BB C$, such as circle packing, Riemann uniformization, and Tutte embedding.
It is expected that for any reasonable choice of embedding with conformal properties, 
the embedded planar maps should converge to LQG, e.g., in the sense that the counting measure on vertices (appropriately rescaled) should converge to the LQG measure. Moreover, certain random curves on the embedded planar map should converge to SLE curves. 
So far, this type of convergence has only been proven for a special one-parameter family of random planar maps called \emph{mated-CRT maps} which are defined for all $\gamma \in (0,2)$~\cite{gms-tutte}.
\medskip

A priori, there is no direct relationship between the above modes of convergence. 
Each encodes different information about the planar map and none implies any of the others. 
One expects that random planar maps in the $\gamma$-LQG universality class should converge to $\gamma$-LQG in each of the above three senses. In fact, this convergence should occur \emph{jointly}, in the sense that the joint law of the triple consisting of three copies of the random planar map should converge to the joint law of the triple consisting of three copies of the $\gamma$-LQG surface with respect to the product of the above three topologies.  

In this paper, we will prove the joint convergence of critical site percolation on a random triangulation to SLE$_6$ on $\sqrt{8/3}$-LQG in the metric and peanosphere sense (Theorem~\ref{thm-metric-peano}). This is the first such joint scaling limit result for any random planar map model. 
We will also extend the result of~\cite{gwynne-miller-perc}, which gives the GHPU convergence of the map decorated by a single percolation interface toward $\sqrt{8/3}$-LQG decorated by chordal SLE$_6$, to a convergence result for the full collection of interfaces toward $\sqrt{8/3}$-LQG decorated by a conformal loop ensemble with $\kappa=6$~\cite{shef-cle} (Theorem~\ref{thm:loop-vague}). 

This paper is an important step in an ongoing program of the second and third authors to prove that site percolation on a random triangulation converges to SLE$_6$-decorated $\sqrt{8/3}$-LQG under a certain embedding --- the so-called \emph{Cardy embedding} --- which is named after Cardy's formula for percolation~\cite{cardy-formula,smirnov-cardy}.
In fact, combined with the results of the present paper the argument will give joint convergence in all three of the above senses. 
Other papers involved in the proof of the Cardy embedding convergence include~\cite{gwynne-miller-char,gwynne-miller-perc,bhs-site-perc,hlls-cut-pts,hls-sle6,ghss-ldp,hs-cardy-embedding,aasw-type2}.
See Section~\ref{sec:app} and Remark~\ref{rmk:Cardy} for further discussion of the Cardy embedding. 

\subsection{Main result}
\label{sec:main-result}

For a planar map $\Map$, we write $\mcl V(\Map)$, $\mcl E(\Map)$, and $\mcl F(\Map)$ for the set of vertices, edges, and faces, respectively, of $\Map$. 
A map is \emph{rooted} if one of its edges, called the \emph{root edge}, is 
distinguished and oriented.  The face to the right of the root edge is called the \emph{root face}. Given an integer $\el\ge 2$, a planar map $\Map$ is called \emph{a triangulation with boundary} if every   face in $\cF(\Map)$  has degree 3 except that the root face has degree $\el$. We call $\el$ the \emph{boundary length} of $\Map$. We write $\partial \Map$ for the graph consisting of  edges and vertices on the root face of $\Map$. A vertex on  $\Map$ is called a \emph{boundary vertex} if it is on $\bdy\Map$. Otherwise, it is called an \emph{inner vertex}. We similarly define \emph{boundary edges} and \emph{inner edges}. 

A graph is called \emph{2-connected} if removing any vertex does not disconnect the  graph. If a  triangulation with boundary $\Map$ is 2-connected, we call it a \emph{(loopless) triangulation with simple boundary} since there are no self-loops  in $\Map$ and  $\bdy \Map$ is a simple cycle. 
For an integer $\el\ge 2$, let $\frk T(\el)$ be the  set of such maps with boundary length $\el$.  
By convention, we view a map with a single edge as an element in $\frk T(2)$ which we  call the \emph{degenerate triangulation}. To highlight the role of the root, we will write  each element in $\cup_{\ell\ge 2}\frk T(\ell)$ in the form of $(\Map,\Be)$, where $\Be$ is the directed rooted edge of $\Map$. 

Given $\Map\in \cup_{\el\ge 2} \frk T(\el)$, a \emph{site percolation} on $\Map$ is a coloring of $\cV(\Map)$ in two colors, say, red and blue. The Bernoulli-$\frac12$ site percolation on $M$ is the random site percolation $\omega$ on $M$ such that each inner vertex is independently colored in red or blue with equal probability. The coloring of the boundary vertices is called the \emph{boundary condition} of $\omega$, which can follow any distribution independent of $\omega|_{\cV(\Map)\setminus \bdy\cV(\Map)}$.
We say $\omega$ has \emph{monochromatic} red (resp.\ blue) boundary condition if all boundary vertices are red (resp.\ blue).

\begin{definition}\label{def:bol}
	For an integer $\el\ge 2$, the \emph{(critical) Boltzmann triangulation with simple boundary of length} $\el$ is a probability measure on $\frk T (\el)$ where each element is assigned probability
	\begin{equation}\label{eq:Bol-def}
	\left(\frac{2}{27}\right)^n \frac{(\el-2)!\el!}{(2\el-4)!}\left(\frac{4}{9}\right)^{\el-1}, \qquad\textrm{where $n$ is the number of inner vertices.}
	\end{equation}
	Suppose a random triple $(\Map,\Be,\omega)$ is such that the marginal law of $(\Map,\Be)$ is the critical Boltzmann triangulation given its boundary length and conditioning on $(\Map,\Be)$ and $\omega|_{\cV(\bdy \Map)}$, the conditional  law of $\omega$   is the Bernoulli-$\frac12$ $\cV(\Map)$ site percolation. Then we call (the law) of $(\Map,\Be,\omega)$ a \emph{critical site-percolated Boltzmann triangulation}.
\end{definition}

The primary object of interest in our paper is the scaling limit of the site-percolated Boltzmann triangulation. To be more precise, fix $\LL >0$ and a sequence  $\{\el^n\}_{n\in\BB N}$  in $[2,\infty)\cap \N$ such that $\left(\frac2{3n}\right)^{1/2}\el^n \rta \LL$ as $n\rta\infty$. For $n\in\BB N$, we consider the triple $(\Map^n, \BB e^n, \omega^n)$ where $\Map^n$ is the Boltzmann triangulation of boundary length $\el^n$ rooted at $\BB e^n$, and conditional on $\Map^n$, {we have that} $\omega^n$ is a Bernoulli-$\frac12$ site percolation on $\Map^n$ with monochromatic red boundary condition.

{The site-percolated Boltzmann triangulation was first studied in~\cite{angel-peeling}, where it was proved that the critical threshold is $\frac12$. 
In that paper a key tool called the \emph{peeling process} was introduced. The peeling process}  also plays a fundamental role in, e.g.,~\cite{angel-uihpq-perc,angel-curien-uihpq-perc,angel-ray-classification,richier-perc,gwynne-miller-saw,caraceni-curien-saw,gwynne-miller-simple-quad,gwynne-miller-perc}. 
We will review the peeling process associated with the percolation interface in Section~\ref{sec-perc-interface}. Roughly speaking, this process explores the edges along the percolation interface in order, keeping red vertices to the left and blue vertices to the right. 

In~\cite{bhs-site-perc}, an iterative peeling process was used to define a space-filling exploration\footnote{Throughout this paper, we denote space-filling curves with a prime and non-space-filling curves (such as percolation interfaces or ordinary chordal SLE$_6$) without a prime. Note that this differs  from the convention of~\cite{ig1,ig2,ig3,ig4}, where a prime is used for any non-simple curve.}
$\acute{\lambda}^n$ of $\cE(\Map^n)$, i.e., a total ordering of $\cE(\Map^n)$. Moreover $\acute\lambda^n$ defines a random walk $\acute{\mcl Z}^n = (\acute{\mcl L}^n , \acute{\mcl R}^n)$ of duration
$\#\mcl E(\Map^n)$, with steps in $\{(1,0), (0,1), (-1,-1)\}$, describing the evolution of the lengths of the two arcs between the starting point and the target point on  the boundary of the unexplored region.  
The construction of $\acute{\lambda}^n$ and $\acute{\cZ}^n$ will be reviewed in Section~\ref{sec-walk}.\footnote{In fact the peeling process perspective is alluded to but not highlighted in \cite{bhs-site-perc}. We give a self-contained treatment in Section~\ref{sec-walk} with this perspective.}

We equip $\Map^n$ with the graph distance and the counting measure on vertices, rescaled appropriately (the precise scaling is specified at the end of Section~\ref{sec-ghpu}). Then $\acute{\frk M}^n:=(\Map^n,\acute{\lambda^n})$ can be thought of as a compact metric measure space decorated with two curves,  $\bdy \Map^n$ and $\acute{\lambda}^n$. The natural topology on  the space of compact curve-decorated metric measure spaces is the \emph{Gromov--Hausdorff--Prokhorov--uniform (GHPU) topology}, whereby two such spaces are close if they can be isometrically embedded into a common space in such a way that the spaces are close in the Hausdorff distance, the measures are close in the Prokhorov distance, and the curves are close in the uniform distance. This topology was introduced in~\cite{gwynne-miller-uihpq}, and will be reviewed in Section~\ref{sec-ghpu}.

The continuum analog of site percolation on a Boltzmann triangulation with boundary is a space-filling SLE$_6$ curve on a Brownian disk (equivalently, by~\cite[Corollary 1.4]{lqg-tbm2}, a $\sqrt{8/3}$-LQG disk) with boundary length $\LL$ (and random area). This object can be viewed as a metric measure space decorated by two curves (the SLE$_6$ and the boundary of the disk). 
We denote this curve-decorated metric measure space by $\frk H'$. 
We will review the definitions of the above objects in more detail in Section~\ref{sec-mating}. In particular, we will explain how the mating-of-trees theorem of~\cite{wedges} allows us to associate with $\frk H'$ a pair of correlated Brownian excursions $Z' = (L',R')$, with correlation $ 1/2$, 
in a manner directly analogous to the definition of the left/right boundary length process $\acute{\cZ}^n$ above.
The following is an informal statement of our first main result. 
\begin{thm} \label{thm-metric-peano}
	Under appropriate scaling, the joint law of $(\acute{\frk M}^n, \acute Z^n)$ converges to the joint law of $(\frk H',Z')$ where the first coordinate is given the GHPU topology and the second coordinate is given the uniform topology.
\end{thm}

A precise statement of Theorem~\ref{thm-metric-peano}, including the proper scaling, the convergence topology and the description of the limiting object, will be given in Section~\ref{sec-continuum} as Theorem~\ref{thm:main-precise}.
An important input to the proof of Theorem~\ref{thm-metric-peano} is~\cite{gwynne-miller-perc},
which give the joint convergence of a random triangulation decorated by a single percolation interface together with its associated left/right boundary length process to a Brownian disk decorated by a chordal SLE$_6$ together with its associated left/right boundary length process.\footnote{Actually,~\cite{gwynne-miller-perc} proves the analogous statement for face percolation on a quadrangulation instead of site percolation on a triangulation. But, as explained in~\cite[Section 8]{gwynne-miller-perc}, the proof carries over verbatim to site percolation on a triangulation once one has the convergence of triangulations with simple boundary to the Brownian disk, which is proven in~\cite{aasw-type2}. }
Roughly speaking, the idea of the proof is to build the space-filling exploration from nested percolation interfaces, build the space-filling SLE$_6$ analogously from nested chordal SLE$_6$'s, and then apply the result of~\cite{gwynne-miller-perc} countably many times. 
The relationship between chordal and space-filling curves is explained in the discrete (resp.\ continuum) setting in Section~\ref{sec-sle6-def-discrete} (resp.\ Section~\ref{sec-mating}). A similar iteration strategy is used in \cite{camia-newman-sle6} to extract the convergence to $\CLE_6$ from the convergence to $\SLE_6$ for Bernoulli-$\frac12$ site percolation on the regular triangular lattice. However, the argument in that paper heavily relies on the fact that the regular triangular lattice is nicely embedded in the plane, where very strong percolation estimates are known.

\subsection{Applications of the main result}
\label{sec:app}

Suppose $\omega$ is a site percolation on  $\Map\in \bigcup_{\el\ge 2} \frk T(\el)$ with monochromatic boundary condition.
Removing all edges on $\Map$ whose endpoints have different colors, we call each connected component in the remaining graph a \emph{percolation cluster}, or simply a cluster, of $\omega$. 
By definition, vertices in each cluster share the same color. 
Moreover, each pair of neighboring vertices that are on different clusters must have different colors. 
We call the cluster containing $\bdy \Map$ the \emph{boundary cluster}. If $\cC$ is a non-boundary cluster of $\omega$, one can canonically define a loop on $\Map^n$ surrounding $\cC$ as a path of vertices in the dual map. See Footnote \ref{fn:edge-path} and  Figure~\ref{fig-discrete-loops}.
The collections of such loops is called the \emph{loop ensemble} of $\omega$, denoted by $\Gamma(\Map,\omega)$.
In \cite{camia-newman-sle6}, it is proved that given a Jordan domain $D$, the loop ensemble for Bernoulli-$\frac12$ site percolation on the regular triangular lattice on $D$ converges to the so-called $\CLE_6$
on $D$, as defined in~\cite{shef-cle}, as the mesh size tends to zero. In Section~\ref{sec-cle-def} we review the definition of the $\CLE_6$ on the Brownian disk, which we denote by $(H,d,\mu,\xi,\Gamma)$. We also define a natural topology  called the \emph{Gromov--Hausdorff--Prokhorov--uniform--loop (GHPUL)} topology on metric measure space decorated with  a boundary curve and a collection of loops.   
As a byproduct of our proof of Theorem~\ref{thm-metric-peano}, we prove the following.

\begin{figure}[ht!]
	\begin{center}
		\includegraphics[scale=.8]{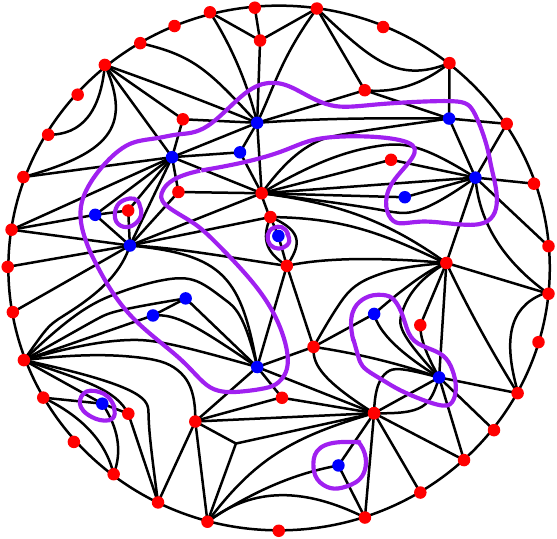} 
		\caption{\label{fig-discrete-loops} 
			A site-percolated triangulation with simple boundary with monochromatic red boundary data along with the associated collection of loops. Theorem~\ref{thm:loop-vague} asserts that this object converges in the metric space sense to CLE$_6$ on the Brownian disk.
		}
	\end{center}
\end{figure}

\begin{thm}\label{thm:loop-vague}
	In the setting of Theorem~\ref{thm-metric-peano}, let $\Gamma^n=\Gamma(\Map^n,\omega^n)$.
	Then under the appropriate scaling $(\Map^n,\Gamma^n)$ converges in law to $(H,d,\mu,\xi,\Gamma)$ as a loop ensemble decorated metric measure space with a boundary curve.
	Moreover, this convergence occurs jointly with the convergence of Theorem~\ref{thm-metric-peano}.
\end{thm}

We state and prove the precise version of Theorem~\ref{thm:loop-vague} in Section~\ref{subsec:CLE-conv}.  
In \cite{bhs-site-perc}, many convergence results related to $(\Map^n,\Be^n,\omega^n)$ above are  proved under a certain embedding from $\Map^n$ to $\D$ depending on the randomness of the percolation configuration and a coupling of percolated maps for different $n$. Theorems~\ref{thm-metric-peano} and~\ref{thm:loop-vague} allow us to transfer all the convergence result in \cite{bhs-site-perc} to a more intrinsic setting. 
The reason for this is that Theorem~\ref{thm-metric-peano} reduces  convergence results for curves on $\Map^n$  in the uniform topology  and measures on $\Map^n$ in the Prokhorov  topology  to the convergence result of certain time sets of the random walk $\acute Z^n$.
As examples, we prove convergence results related to pivotal points and crossing events in Section~\ref{subsec:pivotal} and~\ref{subsec:crossing},   respectively.  These results are important inputs for the work \cite{hs-cardy-embedding}, which shows the convergence of $\Map^n$ to $\sqrt{8/3}$-LQG under the Cardy embedding. See Remarks~\ref{rmk:dynamical} and~\ref{rmk:Cardy} for more detail.

Another important application of Theorem~\ref{thm:loop-vague} is related to the following question. In the setting of Theorem~\ref{thm-metric-peano}, condition on $(\Map^n , \Be^n)$ and let $\wt\omega^n$ be another critical site percolation configuration on $\Map^n$ which is conditionally independent from $\omega^n$. 
Let $\wt \Gamma^n$ be the loop ensemble associated with $\wt\omega^n$. 
In light of Theorem~\ref{thm:loop-vague}, it is natural to conjecture that $(\Map^n,\Gamma^n,\wt \Gamma^n)$   converges in law to $(H,\Gamma,\wt\Gamma )$  where $\Gamma $ and $\wt \Gamma$ are conditionally independent CLE$_6$'s on $H$. In a similar vein, one expects that the boundary length processes $\acute Z^n$ and {$\wt{Z}^n$} associated with $(\Map^n,\Be^n ,\omega^n)$ and $(\Map^n,\Be^n,\wt\omega^n)$ converge jointly to the boundary length processes $Z'$ and $\wt Z'$ associated with two independent space-filling SLE$_6$'s on the same Brownian disk. 

These questions appear to be very challenging since it is difficult to tease apart the randomness arising from $H$ and from $\eta'$ in the definition of $Z$. 
Any subsequential scaling limit of $(\acute{Z}^n, \wt Z^n)$ is a pair of coupled Brownian motions. These coupled Brownian motions give rise to a pair of coupled SLE$_6$-decorated Brownian disks due to the mating-of-trees theorem~\cite[Theorem 2.1]{sphere-constructions}.
Theorem~\ref{thm-metric-peano} implies the corresponding Brownian disks must be equal almost surely. 
In \cite{hs-cardy-embedding}, it will be shown that in fact the corresponding space-filling SLE$_6$'s are conditionally independent given the Brownian disk, which proves the joint convergence $(\acute{Z}^n, \wt{Z}^n) \rta (Z',\wt Z')$. As demonstrated in \cite{hs-cardy-embedding}, this joint convergence is essentially equivalent to the convergence of the Cardy embedding of $(\Map^n,\Be^n,\omega^n)$ to SLE$_6$-decorated $\sqrt{8/3}$-LQG.

\subsection*{Outline}

The rest of this paper is structured as follows. 

Section~\ref{sec-discrete} is the combinatorial foundation of the paper, where we review the space-filling exploration of site-percolated loopless triangulations with simple boundary and the associated random walk encoding from~\cite{bernardi-dfs-bijection,bhs-site-perc}. In fact, we will reformulate this encoding in terms of an iterative peeling process, which makes its connection to peeling clearer.

In Section~\ref{sec-continuum}, we review the GHPU topology used in Theorem~\ref{thm-metric-peano} and provide the necessary background on $\SLE_6$, $\sqrt{8/3}$-LQG, and the Brownian disk. In particular, we explain how space-filling SLE$_6$ can be constructed by iterating ordinary chordal SLE$_6$ curves inside the ``bubbles" which the curve cuts out, in a manner analogous to the definition of the space-filling exploration from the previous section. We will also give a precise statement of Theorem~\ref{thm-metric-peano}. 

In Section~\ref{sec-scaling}, we prove that the joint law of $\Map^n$ (viewed as a metric measure space with a distinguished boundary curve) and the random walk $\acute{\cZ}^n$ converges in the scaling limit to a Brownian disk and a correlated Brownian excursion. In other words, we prove all of Theorem~\ref{thm-metric-peano} except for the uniform convergence of the space-filling exploration toward space-filling SLE$_6$.
In Section~\ref{sec-sle6-conv}, we conclude the proof of Theorem~\ref{thm-metric-peano}. 
The proofs in Sections~\ref{sec-scaling} and~\ref{sec-sle6-conv} are both based on the main result in~\cite{gwynne-miller-perc}, which gives the GHPU convergence of a single percolation interface and its left/right boundary length process, together with the description of the space-filling curves in terms of nested chordal curves.
In Section~\ref{sec-CLE}, we prove Theorem~\ref{thm:loop-vague} as well as some consequences thereof which will be used in~\cite{hs-cardy-embedding}.

\subsubsection*{Basic notation}
\label{sec-basic-notation}

We write $\BB N$ for the set of positive integers and $\BB N_0 = \BB N\cup \{0\}$. For a finite set $A$ we let $\#A$ denote its cardinality.
For $a,b \in \BB R$ with $a<b$, we define the discrete interval $[a,b]_{\BB Z} := [a, b]\cap \BB Z$.  
For each $A\subset \Z$, we define the \emph{connected components} of $A$ to be the connected components with respect to the graph structure on $\Z$ where $i,j\in\Z$ are adjacent if and only if $|i-j|=1$.
We write $X\overset{d}{=}Y$ if two random variables $X,Y$ have the same law.

\subsubsection*{Acknowledgements}
{We thank an anonymous referee for helpful comments on the draft.} E.G.\ was partially supported by a Herchel Smith fellowship and a Trinity College junior research fellowship.
N.H.\ was partly supported by a doctoral research fellowship from the Norwegian Research Council and partly supported by Dr.\ Max R\"ossler, the Walter Haefner Foundation, and the ETH Z\"urich Foundation.  
X.S.\ was supported by Simons Foundation as a Junior Fellow at Simons Society of Fellows  and by NSF grants DMS-1811092 and by Minerva fund at Department of Mathematics at Columbia University.

\section{Discrete preliminaries}
\label{sec-discrete}
This elementary section reviews the classical peeling process  as well as the bijection between  percolated triangulations and certain random walks discovered in \cite{bernardi-dfs-bijection,bhs-site-perc}. Our presentation is self-contained and reveals the close relation between the peeling process and the bijection. 

\subsection{Peeling process along a percolation interface}
\label{sec-perc-interface}

We let $\DP$ be the set of triples $(\Map,\BB e,\omega)$ satisfying the following. First, $\Map\in \bigcup_{\el\ge 2}\frk T(\el)$ is a loopless triangulation with simple boundary equipped with an oriented rooted edge $\Be \in\bdy\Map$. Moreover,  $\omega$ is a site percolation  on $\Map$ with  \emph{dichromatic boundary condition}, namely, 
if  $\Be $ is oriented from a red vertex to a blue vertex, and if $\partial \Map$ is counterclockwise oriented, then there exists a unique edge $\wh{\Be }$ oriented
from a blue vertex to a red vertex.  Given $(\Map,\BB e,\omega)\in \DP$,  {the boundary} $\partial \Map$ consists of a red (resp.\ blue)  clockwise (resp.\ counterclockwise) arc between $\Be $ and $\wh{\Be }$, which we call the left (resp.\ right) boundary of $\Map$. The left (resp.\ right) boundary length  is the number of edges on $\partial \Map$ whose two endpoints are both red (resp.\ blue). Given $\elll,\ellr \in \N_0$, let $\DP(\elll,\ellr)$ be the subset of $\DP$ where the left and right boundary length of each element is given by $\elll$ and $\ellr$, respectively. {See Figure \ref{fig-tri-perc} for an illustration.} For $n\in\N_0$ let $\DP(n;\elll,\ellr)$ be the subset of $\DP(\elll,\ellr)$ where each map has $n$ inner vertices.  If $\Map$ is degenerate, (i.e., $\Map$ consists of a single edge), let $\Be=\wh\Be$ be the unique edge in $\cE(\Map)$, and let $\omega$ be the coloring on $\cV(\Map)$  such that $\Be $ is oriented from a red vertex to a blue vertex. By convention we consider $(\Map,\Be, \omega)$ to be the unique element in $\DP(0;0,0)$.  
We further let
\begin{align}
&\DP^{\mathrm{r}}(n;\el)=\DP(n;\el,0), &&\DP^{\mathrm{b}}(n;\el)=\DP(n;0,\el),\qquad \DPm(n;\el)=\DP^{\mry}(n;\el)\cup \DP^{\mathrm b}(n;\el),\nonumber\\
&\DP^{\bullet}(\el)=\bigcup_{n\in \N_0} \DP^{\bullet}(n;\el), &&\DP^{\bullet}=\bigcup_{\el\in \N_0} \DP^{\bullet}(\el) \qquad\qquad \textrm{ for } \bullet=\mathrm{r},\mathrm b,\mathrm m. \label{eq:DP}
\end{align}
\begin{remark}\label{rmk:flip}
	Given  a site percolation $\omega$  on $\Map\in \frk T(\el)$  and  an edge $e\in \cE(\partial \Map)$, suppose the boundary condition for $\omega$ is monochromatic red. Then there is a unique way to orient $e$ such that after flipping  the color of the head of $e$ to blue, the percolation on $(\Map,e)$ belongs to $\DP^{\mry}(\el)$. This identifies $(\Map,e,\omega)$ as an element in  $\DPm(\el)$. The same holds if we swap the roles of red and blue.  
\end{remark}

\begin{definition}\label{def:interface}
	For $(\Map,\Be , \omega)\in \DP$, there is a unique interface from $\BB e$ to $\wh{\BB e}$ with red to its left and blue to its right.
	We can represent this interface as a function $\lambda: [0,m]_{\BB Z} \rta \mcl E(\Map)$ for some $m \in \BB N$ with the property that $\lambda(0) = \Be $, $\lambda(m) = \wh{\Be }$, $\lambda(i-1)$, and $\lambda(i)$ share an endpoint for each $i\in [1,m]_{\BB Z}$, and each edge $\lambda(i)$ has one red and one blue endpoint (see Figure~\ref{fig-tri-perc}, left).
	We extend the definition of $\lambda$ from $[0,m]_{\BB Z}$ to $\BB N_0$ by declaring that $\lambda(i) = \wh{\Be }$ for $i > m$.
	We call $\lambda$  the  \emph{percolation interface} of $(\Map,\Be,\omega)$ and $\wh{\Be }$ the \emph{target} of $\lambda$.
\end{definition}
\begin{figure}[ht!]
	\begin{center}
		\includegraphics[scale=.8]{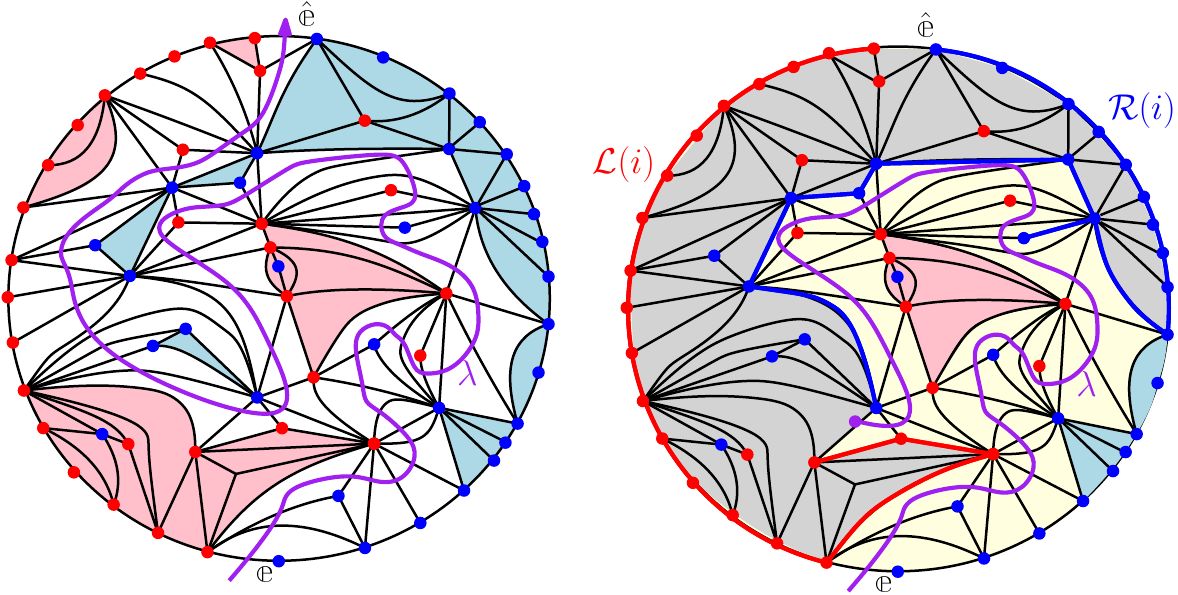} 
		\caption{\label{fig-tri-perc} A triangulation with simple boundary decorated with a percolation with dichromatic boundary condition  together with the percolation interface. {The percolated map is an element of $\DP(\elll,\ellr)$ with $\elll=16$ and $\ellr=17$.}
		\textbf{Left:} The 2-connected components of $\Map - \lambda$ lying to the left (resp.\ right) of $\lambda$ consist of the pink (resp.\ light blue) regions and the vertices and edges along their boundaries. By convention, an edge with two red (resp.\ blue) endpoints which is not on the boundary of any pink (resp.\ light blue) triangle also counts as a connected component. {See also Figure \ref{fig-2-conn} for an illustration of the 2-connected components of $\Map-\lambda$.} \textbf{Right:} The process $\mcl Z = (\mcl L , \mcl R)$ gives the left and right boundary lengths of the 2-connected component of $\Map - \lambda([0,i]_{\BB Z})$ with $\wh{\BB e}$ on its boundary (shown in gray).
		}
	\end{center}
\end{figure}

Given a planar map $\Map$ and an edge subset $E\subset \cE(\Map)$, let $\Map-E$ be the collection of maps obtained by removing edges in $E$. The graph corresponding to $\Map-E$ has vertex set $\cV(\Map)$ and edge set $\cE(\Map)\setminus E$. 
Throughout this paper we identify $\Map$ and $\Map-E$ with their corresponding graphs when a graph-theoretic notion is applied to them. For simplicity $\Map-E$ is written as $\Map-e$ if $E=\{e\}$.

Given a graph $G$, a nonempty subgraph $G'$ is called a \emph{2-connected component of $G$} if  $G'$ is the unique 2-connected subgraph of $G$ containing $G'$ itself. In Definition~\ref{def:interface} where $\lambda$ is the percolation interface, we write $\Map-\lambda(\N_0)$ as $\Map-\lambda$ for simplicity.
The graph $\Map-\lambda$ has two connected components, both of which are triangulations with (not necessarily simple) boundary, where each component contains one endpoint of $\lambda(i)$ on its boundary for each $i\in\N$. Moreover, the boundary of the two components are both monochromatic. We say that the component with red (resp.\ blue) boundary is to the left (resp.\ right) of $\lambda$. The interface
$\lambda$ traces the left  component  in the  counterclockwise direction in the sense that the left endpoint of $\lambda(i)$ as $i$ increases from $0$ to $m$ traces the boundary of the left component in the counterclockwise direction. Similarly, {the interface} $\lambda$ traces the right  component  in the clockwise direction. For each 2-connected component $U$ of the left (resp.\ right) component of $\Map-\lambda$,
let 
\begin{equation}\label{eq:IU}
I_U=\{i\in[0,m]_\Z: \cV(\lambda(i))\cap \cV(U) \neq \emptyset\}.
\end{equation}  
The set of left (resp.\ right) endpoints of edges in $\lambda(I_U)$ equals $\cV(\bdy U)\setminus\cV(\bdy \Map )$ if $U$ is on the left (resp.\ right) side of $\lambda$. Moreover,  $\lambda$ traces $\bdy U$ in the counterclockwise (resp.\ clockwise) direction.

One can also construct $\lambda$ by exploring $\Map$ one edge at a time, based on the information of $\omega$.
This is sometimes called \emph{peeling}; {see e.g.~\cite{angel-peeling}}. 
We start by defining some related quantities. {See Figure \ref{fig-peeling-Mpp} for an illustration.}
\begin{definition}\label{def:peel-one}
	Given $(\Map,\Be ,\omega)\in \DP$ which is not degenerate, let $t\defeq t(\Map,\Be)$ be the unique triangle (inner face) of $\Map$ which is incident to $\Be $ and let $v\defeq v(\Map,\Be)$ be the vertex on $t$ but not on $\Be $ (such a vertex exists since $\Map$ has no loops).
	Let $\Be'\defeq \Be'(\Map,\Be,\omega)$ be the edge on $t$ other than $\Be$ with its two vertices having the opposite color.
	Let $\Map'\defeq \Map'(\Map,\Be,\omega)$ be the $2$-connected component of $\Map-\Be$ containing the target edge $\wh\Be$ and let $\omega'$ be the restriction of $\omega$ to $\cV(\Map')$. If $\Map-\Be$ has two 2-connected components, let $\Map''\defeq \Map''(\Map,\Be,\omega)$ be the one other than $\Map'$ and let $\Be''\defeq  \Be''(\Map,\Be,\omega)$
	be the edge shared by $t$ and $\Map''$.  Moreover,  define a coloring  $\omega''\defeq \omega''(\Map,\Be,\omega)$ on $\cV(\Map'')$ by  letting $\omega''=\omega$ on $\cV(\Map'')\setminus\{v\}$ and $\omega''(v)$ be opposite to $\omega(v)$. We orient $\Be'$ and $\Be''$ so that $(\Map',\Be,\omega')\in \DP$ and $(\Map'',\Be'',\omega'')\in \DP^{\mrm}$.
\end{definition}

\begin{figure}[ht!]
	\begin{center}
		\includegraphics[scale=1]{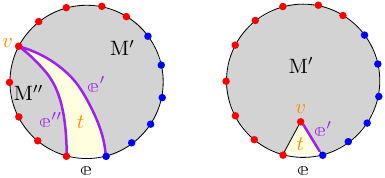}
		\caption{Illustration of Definition \ref{def:peel-one}. 
		{Note that $\Map-\Be$ has  two connected components if and only if $v\in\cV(\bdy \Map)$. Otherwise $\Map-\Be$ itself is still 2-connected.}
		}
		\label{fig-peeling-Mpp}
	\end{center}
\end{figure}

Decompositions similar to the one in Definition~\ref{def:peel-one} were used by Tutte \cite{Tutte} to enumerate planar maps of various types. 
It is also the building block of the peeling process.
\begin{definition}\label{def:peel}
	For $(M,\Be ,\omega)\in \DP$,
	if $\Map $ is degenerate, let $\Peel (\Map,\Be ,\omega)=(\Map,\Be ,\omega)$.
	Otherwise, let $\Peel (\Map,\Be ,\omega)=(\Map',\Be' ,\omega')$ with notation as in Definition~\ref{def:peel-one}. We call $\Peel : \DP \rta \DP$ the \emph{peeling operator}. 
\end{definition}

Now let $(\Map_0,\Be _0,\omega_0)=(\Map,\Be , \omega)$ and for $i\in \N$, inductively let $ (\Map_i,\Be _i,\omega_i)=\Peel (\Map_{i-1},\Be _{i-1}, \omega_{i-1})$. Then the percolation interface $\lambda$  for $(\Map,\Be,\omega)$ satisfies  $\lambda(i)=\Be _i$ for all $i\in \N_0$. For $i\in [0,m]_\Z$, the map  $\Map_i$ is the 2-connected component of $\Map-\lambda([0,i-1]_\Z)$ containing $\wh \Be$. Here we make the convention that $[0,-1]_\Z=\emptyset$. We note that this is an example of a peeling process on $\Map$; see, e.g., the lecture notes~\cite{curien-peeling-notes} for more information about such processes. Similar peeling processes to the one considered here are also used, e.g., in~\cite{angel-peeling,angel-curien-uihpq-perc,gwynne-miller-perc,richier-perc}. 

For all $i\in \N_0$, let  $\cL(i)$ (resp.\ $\cR(i)$) be the left (resp.\ right) boundary length of $(\Map_i,\Be_i, \omega_i)$, as illustrated in Figure~\ref{fig-tri-perc}, right. 
We call  \( \mcl Z:= (\mcl L, \mcl R) \) the \emph{boundary length process} of $\lambda$.  For $0\le i<   m$, if $\Map_i-\Be_i$ is 2-connected, then 
$\cZ(i+1)-\cZ(i)=(1,0)$ (resp.\ $(0,1)$) when the added boundary vertex is red (resp.\ blue).
Otherwise, let $U=\Map''(\Map_i,\Be_i,\omega_i)$ be the 2-connected component of $\Map_i-\Be_i$ other than $\Map_{i+1}$. 
Then $i$ and $i+1$ belong to $ I_U$ (recall~\eqref{eq:IU}). If $U$ is on the left of $\lambda$, then $\cL(i)-\cL(i+1)$ equals the boundary length of $U$ minus 1, and $\cR(i)=\cR(i+1)$. Moreover,  $\cL$ is nondecreasing on $I_U\cap [0,i]$ with the increments being 0  or 1, and $\cL$ is constant on $I_U\cap [i+1,\infty)$. We further observe that%
\eqb \label{eqn-I_U-property}
\cL(\max I_U)=\min_{j\in [\min I_U,\max I_U]_{\Z}}  \cL(j) .
\eqe 
Moreover, 
$\cL(\min I_U)=\cL(\max I_U)$  if and only if $\cE(U)\cap \cE(\bdy\Map)=\emptyset$.
If $U$ is on the right of $\lambda$, all statements hold with $\cL$ and $\cR$ swapped.

For $i\in \N_0$, let $\bdy_{i}$ be the boundary of $\Map- \lambda([0,i-1]_{\BB Z})$ (by convention $[0,-1]_\Z=\emptyset$) and $X_i:=(\bdy_{i},\Be_{i}, \omega|_{\bdy_{i}})$. 
Then $\{X_j\}_{j\in [0,i]_{\Z}}$ and $\mcl Z|_{[0,i]_{\BB Z}}$ determine each other. 

The main property of $\Peel$ that we will rely on is the following domain Markov property. 
\begin{lem} \label{lem-interface-markov}
	Given $\elll,\ellr\in\N_0$, suppose $(\Map,\Be,\omega)$ is a  site-percolated Boltzmann triangulation with $(\elll,\ellr)$-boundary condition.
	Given $i\in \N_0$, for each 2-connected component of $\Map- \lambda([0,i]_{\BB Z})$, suppose a root edge is chosen  in a manner depending only on $\{ X_j\}_{j\in [0,i+1]_\Z}$. 
	Then conditional on  $\{ X_j\}_{j\in [0,i+1]_\Z} $, 
	the 2-connected components of $\Map- \lambda([0,i]_{\BB Z})$ together with the percolation $\omega$ restricted to them are distributed as  independent  site-percolated  Boltzmann triangulations with given boundary condition.
\end{lem}
\begin{proof}
	Since $\lambda$ can be constructed by a peeling process on $\Map$, Lemma~\ref{lem-interface-markov} follows by iteratively applying the so-called \emph{Markov property of peeling} (see, e.g.,~\cite[Section 2.3.1]{angel-curien-uihpq-perc}). In the notion of Definition~\ref{def:peel-one}, it gives that  $(\Map',\Be',\omega')$ and $(\Map'',\Be'',\omega'')$ (when it is defined)   are independent site percolated Boltzmann triangulations given their respective boundary condition.
\end{proof}

\subsection{Space-filling peeling and random walk}
\label{sec-walk}

In this section, we  use the peeling process in Section~\ref{sec-perc-interface} to review  the bijection between random walks and site-percolated triangulations from~\cite{bernardi-dfs-bijection,bhs-site-perc}. We also prove some basic lemmas about the bijections.
\begin{definition}\label{def:spacefilling}
	Given $(\Map,\Be,\omega)\in\DP$, we  define a total ordering of $\cE(\Map)$,  i.e., a bijection $\acute{\lambda}: [0,\#\cE(\Map)-1]_\Z \to \cE(\Map)$. 
	Let $\acute{\lambda}(0)=\Be $. If $(\Map,\Be )$ is not degenerate,  we define  $\acute{\lambda}|_{[1,\#\cE(\Map)-1]_\Z}$ inductively as follows. Suppose that $\acute{\lambda}$ is defined for all elements in $\DP$ whose number of edges is smaller than $\#\cE(\Map)$.
	Recall the notation in Definition~\ref{def:peel-one}, including $t$, $v$,  $(\Map',\Be ',\omega')$, and $(\Map'',\Be'',\omega'')$.
	\begin{enumerate}[label=(\arabic{enumi})]
		\item\label{item:lambda1}
		If $v\notin\cV(\bdy\Map)$, then the total ordering on $\cE(\Map)$ induced by $\acute\lambda$, restricted to $\cE(\Map)\setminus \{ \Be  \}$, is given by the total ordering for  $(\Map',\Be ',\omega')$, which is already defined by our inductive hypothesis;
		\item  \label{item:lambda2}
		If $v\in\cV(\bdy \Map)$, then $\acute{\lambda}(e'') < \acute{\lambda}(e')$ for each $e'\in \cE(\Map')$ and $e''\in \cE(\Map'')$.
		Moreover, the total ordering of $\cE(\Map)$ induced by $\acute{\lambda}$ restricted to $\cE(\Map')$ (resp.\ $\cE(\Map'')$) is given by the total ordering for $(\Map',\Be ',\omega')$ (resp.\ $(\Map'',\Be '',\omega'')$).
	\end{enumerate} 
	We call $\acute{\lambda}$ the \notion{space-filling exploration} of $(\Map,\Be,\omega)$. 
\end{definition} 

In Definition~\ref{def:spacefilling}, let $N=\#\cE(\Map)$. For $i\in [0,N-1]_\Z$, let $\ol \Map_i=\Map-\acute{\lambda}([0,i-1]_\Z)$ and $\acute{e}_i=\acute\lambda(i)$, where we set $[0,-1]_\Z=\emptyset$ by convention. As $\acute\lambda$ is inductively defined, {the edge} $\acute{e}_j$ is assigned an orientation along the way (see Definition~\ref{def:peel-one}).
It can be checked inductively  that $(\ol \Map_i,\acute e_i)$ is a triangulation with (possibly non-simple) boundary such that  $ \wh\Be \in\cE(\partial \ol\Map_i)$. 

Let ${\acute\cL}(i)$ (resp.\ ${\acute\cR}(i)$) be the  number of edges on $\bdy \ol\Map_i$ that are traversed when tracing $\partial \ol \Map_{i}$ clockwise (resp.\ counterclockwise) from
the left (resp.\ right) endpoint of $\acute{e}_{i}$ to the  left (resp.\ right) endpoint $\wh\Be $. Let  $\acute{\cZ}(i)=({\acute\cL}(i),{\acute\cR}(i))$ for $0\le i< N$ and
$\acute{\cZ}(N)=(-1,-1)$.
We call $\acute{\cZ}=\{\acute\cZ_i\}_{i\in[0,N]_\Z}$ the \emph{boundary length process} of $\acute{\lambda}$. See Figure \ref{fig-Z} for illustration.

\begin{figure}[ht!]
	\begin{center}
		\includegraphics[scale=1]{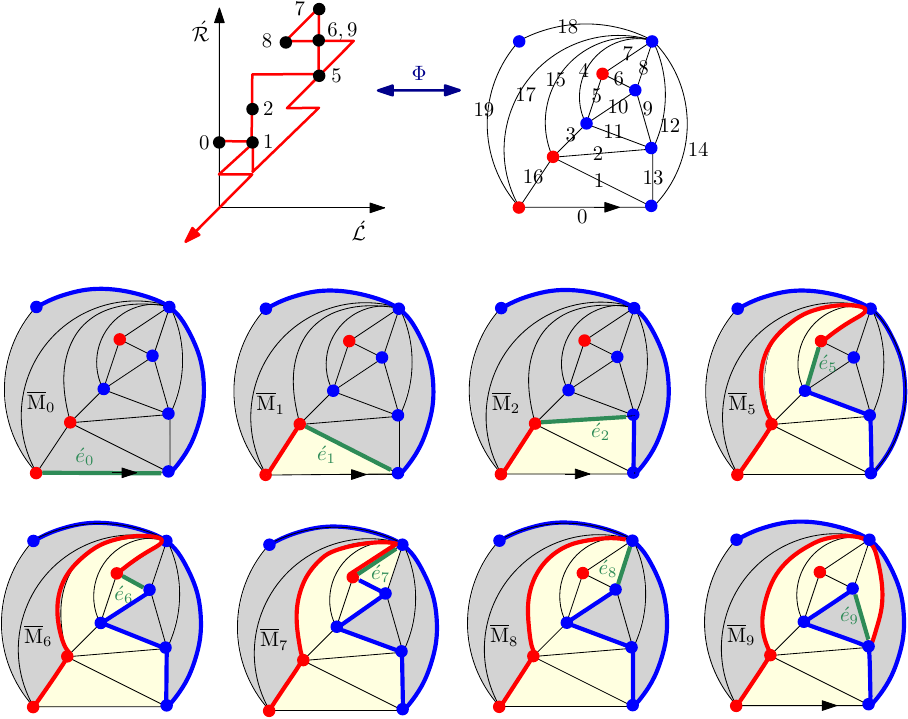} 
		\caption{ 
		{{\bf Top row}}: Illustration of the bijection $\Phi$ in Theorem \ref{thm:bijection} between $\cK$ and $\DP$. {In the left part of the figure we have added a black dot at $\acute\cZ_i=(\acute\cL_i,\acute\cR_i)$ for $i=0,1,2,5,6,7,8,9$ (corresponding to the figures in the middle and lower rows) and written $i$ next to the dot.} {In the right part of the figure} we have labeled the edge $\acute{e}_i=\acute\lambda(i)$ by $i$ for $i\in[0,N-1]_{\Z}$. 
		{{\bf Middle and bottom rows}: The map $\ol\Map_i$ for $i=0,1,2,5,6,7,8,9$ is shown in gray. The length of the red (resp.\ blue) curve is equal to $\acute{\cL}(i)$ (resp.\ $\acute{\cR}(i)$).}
		}
		\label{fig-Z} 
	\end{center}
\end{figure}

For each $i\in [1,N]_\Z$, we call $\Delta\acute{\cZ}_i\defeq\acute{\cZ}({i})-\acute{\cZ}(i-1) $ the \emph{step} of $\acute{\cZ}$ associated with the edge $\acute{e}_{i-1}$. 
By the definition of $\acute{\lambda}$, for each $i\in [1,N]_\Z$, if $\acute e_{i-1}$ itself is a 2-connected component of $\ol\Map_{i-1}$ as the  degenerate element in $\DP$, then $\Delta\acute{\cZ}_i=(-1,-1)$.
Otherwise, the edges  $\acute{e}_{i}$ and $\acute{e}_{i-1}$ share a unique vertex. 
If the vertex is the tail (resp.\ head) of $\acute{e}_{i-1}$, then $\Delta\cZ_i$ equals $(1,0)$ (resp.\ $(0,1)$).  

For any function $Z=(L,R)$ from $[a,b]_\Z$ to $\R^2$ for some $a<b\in \N_0$, 
given $i,j\in [a,b]_\Z$ with $i<j$, we write $j\prec_{Z} i$ if  $L(k) >L(j)$ and $R(k) >R(j)$  for all $k\in [i,j-1]_\Z$. We say that $j$ is an \emph{ancestor}%
\footnote{
	Our definition of ancestors and ancestor-free times, correspond to the definitions in~\cite[Section 10.2]{wedges} for the \emph{time reversal} of the process.}
of $i$.
Geometrically $j \prec_Z i$ means that $Z$ stays in the cone $Z(j) + (0,\infty)^2$ between times $i$ and $j-1$.
In other words $j$ is a simultaneous strict running minima of the two coordinates of $Z$ relative to time $i$.
Given $i\le m\in [a,b]_\Z$, {we call $i$}  an \emph{ancestor-free time relative to $m$} if each $j\in[i,m]_\Z$ is not an ancestor of $i$. 

\begin{thm}\label{thm:bijection}
	Suppose $(\Map,\Be, \omega)\in \DP(n;\elll,\ellr)$ for some $n,\elll,\ellr\in\N_0$ and let $\acute{\cZ}$ be its boundary length processes.
	Let $N=\#\cE(\Map)=3n+2\elll+2\ellr+1$. Then $\acute{\cZ}$ satisfies the following three properties.
	\begin{enumerate}[label=(\arabic{enumi})]
		\item $\acute{\cZ}$ is defined on $[0,N]_{\BB Z}$ starting at $(\elll,\ellr)$ and ending at $(-1,-1)$; \label{item:K1}
		\item $\Delta \acute{\cZ}_i:=\acute{\cZ}({i})-\acute{\cZ}(i-1) \in \{(1,0), (0,1) ,(-1,-1)   \}$ for all $i\in [1,N]_\Z$; \label{item:K2}
		\item  $N=\inf\left\{i\in [1,N]_\Z: i \prec_{\acute{\cZ}} 0 \right\}$, equivalently,  $N$ is an ancestor of $0$ while $0$ is ancestor-free relative to $N-1$. \label{item:K3}
	\end{enumerate} 
	Let $\cK(n;\elll,\ellr)$ be the set of walks satisfying the three properties above and define $\Phi(\Map,\Be,\omega) := \acute{\cZ}$. Then $\Phi$ is a bijection from $\DP(n;\elll,\ellr)$ to $\cK(n;\elll,\ellr)$.

	Given $n,\elll,\ellr,\el \in \N_0$ and $\bullet\in \{\mry,\mrb,\mrm\}$, let $\cK(\elll,\ellr)$, $\cK$, $\cK^{\bullet}(n;\el)$, $\cK^\bullet(\el)$, and $\cK^{\bullet}$ be defined as in \eqref{eq:DP} with $\DP$ replaced by $\cK$. Then $\Phi$ defines a bijection from each set defined in \eqref{eq:DP} to its counterpart for walks. 
\end{thm}
\begin{remark}\label{rmk:last-step}
	The bijection in Theorem \ref{thm:bijection} is slightly different from the one in \cite{bhs-site-perc}, since the walks considered in \cite{bhs-site-perc} do not have the final step $(-1,-1)$ but rather end at $(0,0)$. We include this step to make the connection to the peeling process cleaner.
\end{remark}

\begin{proof} 
	We first prove that $\acute{\cZ}\in \cK(n;\elll,\ellr)$, which is trivially true when $N=1$, where $\Delta\acute{\cZ}_1=(-1,-1)$ is the single step of $\acute{\cZ}$. 
	When $N>1$ we do an induction on $N$ and assume that 
	our assertion holds for every nonnegative integer smaller than $N$.
	Recall the triangle $t=t(\Map,\Be)$ and the vertex $v=v(\Map,\Be)$ in Definition~\ref{def:peel-one}. 
	If $v\notin \bdy\Map$, then we are in Case~\ref{item:lambda1} of Definition~\ref{def:spacefilling}. In this case, by the definition of $\acute{\cZ}$ and our induction hypothesis,  if $v$ is red, then $\Delta\acute{\cZ}_1=(1,0)$ and $\{\acute{\cZ}({i+1})-\acute{\cZ}(1) \}_{i\in [0,N-1]_\Z}$ belongs to  $\cK(n-1;\elll+1,\ellr)$. Therefore $\acute{\cZ}\in \cK(n;\elll,\ellr)$. The same arguments hold if $v$ is blue, in which case  $\Delta\acute{\cZ}_1=(0,1)$.

	If $v\in \bdy \Map$, so that we are in Case~\ref{item:lambda2} of Definition~\ref{def:spacefilling}, recall $(\Map',\Be ',\omega')$ and $(\Map'',\Be '',\omega'')$ in Definition~\ref{def:peel-one} and  let ${\acute{\cZ}}'$ and ${\acute{\cZ}}''$ be the corresponding boundary length process. Then by the inductive hypothesis, {we have that} ${\acute{\cZ}}' \in \cK(n';\elll',\ellr')$ and ${\acute{\cZ}}'' \in \cK(n'';\elll'',\ellr'')$ for some $(n';\elll',\ellr')$ and $(n'';\elll'',\ellr'')$.  By the definition of $\acute{\lambda}$ and $\acute{\cZ}$, we have 
	\begin{align}\label{eq:peel-Z}
	&\acute{\cZ}({i})={\acute{\cZ}}''({i-1})  +(\elll'+1,\ellr'+1)
	&&\textrm{ for }i\in [1,N''+1]_\Z \nonumber,\\
	&\acute{\cZ}({N-i})={\acute{\cZ}}'({N'-i}) 
	&&\textrm{ for } i\in [0,N']_\Z,
	\end{align}
	where $N'=\#\cE(\Map')$ and $N''=\#\cE(\Map'')$. In words, {we have that} $\acute\cZ|_{[1,N]_\Z}$
	is the concatenation of ${\acute{\cZ}}''$ and ${\acute{\cZ}}'$. In particular,
	$\acute{\cZ}(1)=(\elll''+1,\ellr''+1)+(\elll',\ellr')$. Since $\acute{\cZ}(0)=(\elll,\ellr)$,
	when $v$ is red, we have 
	\begin{equation}\label{eq:peel-Z2}
	\ellr'=\ellr,\qquad  \ellr''=0, \qquad\elll'+1+ \elll''=\elll.
	\end{equation}
	This shows that $\acute{\cZ}\in \cK(n;\elll,\ellr)$. When $v$ is blue, the same argument  works with the two coordinates of $\acute{\cZ}$ swapped. 
	As a byproduct, we see that 
	$N-N'=N''+1=\inf\left\{i\in [2,N]_\Z: i\prec_{\acute{\cZ}} 1\right\}$.
	
	To prove that $\Phi$ is a bijection, we now define a ``peeling'' operator $\Peel_{\op Z}$ on $\cK$.  Fix $\acute{\cZ}\in\cK(n;\elll,\ellr)$ and let $N=3n+2\elll+2\ellr+1$. If $N=1$, let $\Peel_{\op Z}(\acute{\cZ})=\acute{\cZ}$. Note that in this case we have $\Delta\acute{\cZ}_1=(-1,-1)$.  
	Now we suppose $N>1$ so that $N\prec_{\acute \cZ} 1$. Let $\Upsilon=\inf\left\{i\in [2,N]_\Z: i\prec_{\acute{\cZ}} 1\right\}$. (See Lemma~\ref{lem:induction} below for the geometric meaning of $\Upsilon$.)
	\begin{enumerate}[label=(\arabic{enumi})]
		\item If $\Upsilon=N$, then let $\Peel_{\op Z}(\acute{\cZ}) :=\{\acute{\cZ}_{i+1}-\acute{\cZ}_1  \}_{i\in [0,N-1]_\Z}$.
		\item If $\Upsilon<N$, let $\Peel_{\op Z}(\acute{\cZ}) :=({\acute{\cZ}}',{\acute{\cZ}}'')$ where ${\acute{\cZ}}''=\{\acute{\cZ}_{i+1}-\acute{\cZ}_{\Upsilon-1}\}_{i\in [0,\Upsilon-1]_\Z}$ and ${\acute{\cZ}}'=\{\acute\cZ_{i+\Upsilon}\}_{i\in[0,N-\Upsilon]_\Z}$.
	\end{enumerate}
	Let $\ol\Peel(\Map,\Be,\omega)=((\Map',\Be',\omega'),$ $(\Map'',\Be'',\omega'' ))$ if $\Map$ is not degenerate and $\Map-\Be$ has two 2-connected components (recall Definition~\ref{def:peel-one}). Let  $\ol\Peel=\Peel$ otherwise.
	Then the proof of $\Phi(\DP(n;\elll,\ellr))\subset \cK(n;\elll,\ellr) $ above shows that 
	$\Peel_{\op Z} \circ \Phi = \Phi\circ \ol\Peel$, where if $\ol\Peel(\Map,\Be,\omega)$ has two components we apply $\Phi$ to each of them.  
	This commuting  relation allows us to conclude the bijectivity of $\Phi$ by an induction on $N$.
\end{proof}

{Given $\elll,\ellr\in\N_0$, let $\PBT(\elll,\ellr)$ be the law of the  site-percolated Boltzmann triangulation with $(\elll,\ellr)$-boundary condition.
	For $\el\in \N_0$, let  $\PBT^{\mry}(\el)=\PBT(\el,0)$.  Under the identification in Remark~\ref{rmk:flip},  $\PBT^{\mry}(\el)$ can be throughout of as the law of a Boltzmann triangulation with boundary length $\el+2$ decorated with a  Bernoulli-$\frac12$ percolation with monochromatic red boundary condition. Note that a necessary condition for Property (3) in the definition of $\cK(n;\elll,\ellr)$ to hold is that $\cZ_i\in [0,\infty)^2$ for all $i\in [0,N-1]_\Z$. Although it is weaker than Property (3) in general, the equivalence does hold when $\elll=0$ or $\ellr=0$. This provides the following sampling method for  $\PBT^{\mry}(\el)$.
	\begin{cor}\label{cor:bijection}
		For $\el\in\N_0$, the set $\cK^{\mry}(\el)$ consists of walks  with step choices $(1,0)$, $(0,1)$, and $(-1,-1)$, starting at $(\el ,0)$ and  ending at $(-1,-1)$, such that  both of the two coordinates stay nonnegative except at the end point.
		Let $\acute{\cZ}\in\cK^{\mry}(\el)$ be a random variable sampled from the probability measure  where each element in  $\cK^{\mry}(n;\el)$ is assigned probability (recall \eqref{eq:Bol-def})
		\[
		\left(\frac{2}{27}\right)^n \frac{\el!(\el+2)!}{(2\el)!}\left(\frac{4}{9}\right)^{\el+1} \cdot\left(\frac{1}{2}\right)^n= \frac{1}{{\Sigma}_\el}\left(\frac{1}{3}\right)^{3n}, \qquad\textrm{where}\qquad
		\Sigma_\el=\frac{(2\ell)!}{\el!(\el+2)!}\left(\frac{9}{4}\right)^{\el+1}. 
		\]
		Then $(\Map,\Be,\omega):=\Phi^{-1}(\acute{\cZ})$ has the law of $\PBT^{\mry}(\el)$
	\end{cor}}

We now  list a few useful properties of $\Phi$ which essentially follow from more general results in \cite{bhs-site-perc}. We give self-contained proofs using the language of peeling, which is not elaborated on in \cite{bhs-site-perc}.  Detailed proofs will only be given to relatively involved statements. 
All the proofs are based on the same induction as in the proof of Theorem~\ref{thm:bijection}. We first summarize some useful facts from this proof. 
In the following lemma and the rest of this subsection, given any $j<k\in [0,N]_\Z$, if $\{\acute{\cZ}_{i+j}-\acute{\cZ}_{k-1}\}_{i\in[0,k-j]_\Z}\in \cK$, we identify $\acute{\cZ}|_{[j,k]_\Z}$ with this element in $\cK$. 
\begin{lem}\label{lem:induction}
	In the setting of Theorem~\ref{thm:bijection}, suppose $N>1$ and let  $\Upsilon=\inf\left\{i\in [2,N]_\Z: i\prec_{\acute{\cZ}} 1\right\}$. Recall notations in Definition~\ref{def:peel-one}.  When $v(\Map,\Be)\in \cV(\bdy\Map)$, we have
	$\Upsilon=\# \cE(\Map'')+1$.
	Moreover,  $\acute{\lambda}(\Upsilon-1)$ is the target edge of $(\Map'',\Be'',\omega'')$ and $\acute\lambda(\Upsilon)=\Be'$. 
	Let $\acute\cZ'=\Phi(\Map',\Be',\omega')$ and $\acute\cZ''=\Phi(\Map'',\Be'',\omega'')$. 
	Then 
	\begin{equation}\label{eq:concatenation}
	\textrm{$\acute{\cZ}|_{[1,N]_\Z}$ is the concatenation of $\acute\cZ'=\acute{\cZ}|_{[1,\Upsilon]_\Z}$ and $\acute\cZ''=\acute{\cZ}|_{[\Upsilon,N]_\Z}$.}
	\end{equation}
	When $v(\Map,\Be)\notin \cV(\bdy\Map)$, we have that $\Upsilon=N$ and 
	\begin{equation}\label{eq:restriction}
	\acute{\cZ}|_{[1,N]_\Z}=\Phi(\Map',\Be',\omega').
	\end{equation}
\end{lem}

Lemma~\ref{lem:induction} was proven in the proof of Theorem~\ref{thm:bijection}.  As an immediate corollary, we have the following.
\begin{lem}\label{lem:walk}
	In the setting of Theorem~\ref{thm:bijection}, let $\lambda:[0,m]_\Z\to \cE(\Map)$  be the interface of $(\Map,\Be, \omega)$, and let $\cZ$ be the boundary length process of $\lambda$ as defined in 
	Section~\ref{sec-perc-interface}.   For $i\in [0,m]_\Z$, let $T(i)$ be such that $\lambda(i) = \acute{\lambda}(T(i))$. 
	Then  $\{T(0), T(1),\cdots , T(m)\}$ is exactly the collection of ancestor-free times for $\acute{\cZ}$ relative to $N-1$, in increasing order.
	Moreover,
	$\cZ(i)=\acute{\cZ}(T(i))$ for all $i\in[0,m]_\Z$.
\end{lem}
\begin{proof}
	The result is immediate for $N=1$. By induction, we assume the first statement is true for maps with less than $N$ edges. For $N>1$, we see that $T(1)=1$ when $v\not\in \cV(\bdy\Map)$ and $T(1)=\Upsilon$ as defined in Lemma~\ref{lem:induction} when $v\in \cV(\bdy\Map)$. 
	In both cases, {the time} $T(1)$ is the first  ancestor-free time relative to $N-1$ after $T(0)=0$.
	By Lemma~\ref{lem:induction}, {we have that} $T(i)>\Upsilon$ for $i\ge 2$.  Applying the induction hypothesis to $(\Map',\Be',\omega')$, we see that 
	$\{T(1),\cdots , T(m)\}$ is the collection of ancestor-free times for $\Phi(\Map',\Be',\omega')$ relative to $\#\cE(\Map'')-1$, in increasing order.
	Now the first statement follows from \eqref{eq:concatenation}.
	
	The second  statement follows from the definition of   $\cZ$ and $\acute{\cZ}$. 
\end{proof}
Here is another immediate corollary of Lemma~\ref{lem:induction}, which we leave to the reader to verify.
\begin{lem}\label{lem:component}
	In the setting of Lemma~\ref{lem:walk}, for $i\in [0,m-1]_\Z$, recall $(\Map_i,\Be_i,\omega_i)$ in Section~\ref{sec-perc-interface} where $\Be_i=\lambda(i)$.
	Then $\Map_i-\Be_i$ has two 2-connected components  if and only if $T(i+1)>T(i)+1$. 
	Define $(\Map''_i,\Be''_i,\omega''_i)$ as $(\Map'',\Be'',\omega'')$ in Definition~\ref{def:peel-one} with $(\Map_i,\Be_i,\omega_i)$ in place of $(\Map, \Be, \omega)$.
	Then $\acute{\cZ}|_{[T(i)+1, T(i+1)]_\Z}=\Phi(\Map''_i,\Be''_i, \omega''_i)$. 
\end{lem}

Following \cite{bhs-site-perc}, given $N\in \N$,  suppose  $\acute{\cZ}=({\acute\cL},{\acute\cR})=\{\acute{\cZ}(i) \}_{i\in[0,N]_\Z}$ is a walk  
ending at $(-1,-1)$  satisfying Property~\ref{item:K2} in Theorem~\ref{thm:bijection}.
For $i\in[1,N]_\Z$, we call $\Delta \acute{\cZ}_i:=\acute{\cZ}({i})-\acute{\cZ}(i-1)$  
an $a$-step if $\Delta\acute{\cZ}_i=(1,0)$, a $b$-step if $\Delta\acute{\cZ}_i=(0,1)$, and a $c$-step if $\Delta\acute{\cZ}_i=(-1,-1)$.  
This identifies $\acute{\cZ}$ with a word  (i.e., sequence of letters)  of length $N$ on the alphabet $\{a,b,c\}$. 
Given $i,j\in [1,N]_\Z$ such that $i<j$, we say that an $a$-step $\Delta\acute{\cZ}_i$ and a $c$-step $\Delta\acute{\cZ}_j$ are \emph{matching} if
\[j=\inf\{m\in [i,N]_\Z: {\acute\cL}(m)={\acute\cL}(i-1) \}. \]  
The matching of $b$-steps and $c$-steps
is defined analogously in the same way with ${\acute\cL}$ replaced by ${\acute\cR}$. Using this terminology, and examining Property~\ref{item:K3} in Theorem~\ref{thm:bijection}, we see that $\acute{\cZ}\in \cK$ if and only if
\begin{enumerate}[label=(\arabic{enumi})]
	\item each $a$-step and $b$-step in $\acute{\cZ}$ has a matching $c$-step, and
	\item $\Delta\acute{\cZ}_{N}$ is a $c$-step and is the unique $c$-step of $\acute{\cZ}$ with neither a matching  $a$- nor $b$-step.
\end{enumerate}
This identification gives that Theorem~\ref{thm:bijection} is equivalent to the bijection in \cite[Corollary 2.12]{bhs-site-perc}.
Since a triangulation of a $2$-gon is equivalent to a rooted triangulation of the sphere by gluing the two boundary edges, 
the special case $\elll=\ellr=0$ of Theorem~\ref{thm:bijection} gives a bijection between site-percolated loopless triangulations and walks in $\cK(0;0)$, which is closely related to the one in \cite{bernardi-dfs-bijection}.

Recall that $\ol\Map_0=\Map$ and $\ol \Map_i=\Map-\acute{\lambda}([0,i-1]_\Z)$ for $i\in[1,N-1]_\Z$, which is a triangulation with boundary but not necessarily 2-connected.
The 2-connected components of $\ol\Map_i$
 also have an easy description in terms of $\acute{\cZ}$. See the left part of Figure \ref{fig-xi} for an illustration.

\begin{figure}[ht!]
	\begin{center}
		\includegraphics[scale=1]{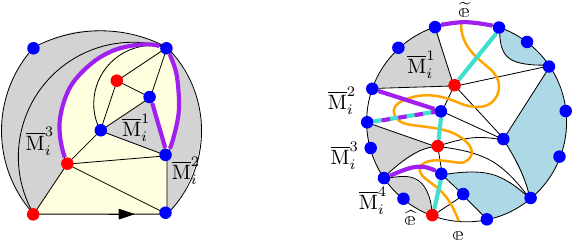} 
		\caption{ {\bf Left}: Illustration of the maps $\{\ol\Map_i^j\}_{j\in [1,k]_\Z}$  in Lemma \ref{lem:future} for $i=9$ and $k=3$. 
		The purple edge on the boundary of $\ol{\Map}_i^j$ is $e_i^j=\acute{\lambda}(\xi_{j-1})$. 
		{\bf Right}: Illustration of Lemmas~\ref{lem:interface} and~\ref{lem:future2}. The percolation interface $\wt\lambda$ from $\Be$ to $\wt\Be$ is shown in orange. Lemma \ref{lem:interface} describes the 2-connected components of $\Map-\wt\lambda$ that are visited by $\acute\lambda$ before $\wt\Be$ (blue).   Lemma~\ref{lem:future2} describes the components $\ol\Map_i^j$ visited after $\wt\Be$ (gray), where $i$ is such that $\wt\Be=\acute\lambda(i-1)$. In the notation of Lemma~\ref{lem:future2}, the edges $\acute\lambda(\sigma_j-1)$ are shown in turquoise, while the edges $\acute\lambda(\sigma_j)$ are shown in purple.
		}
	\label{fig-xi}
	\end{center}
\end{figure}

\begin{lem} \label{lem:future}
	Fix $i\in [0,N-1]_\Z$.  Set $\xi_0=i$. Let $\xi_1,\cdots, \xi_k$  be the  set \(\{\xi\in[i+1,N]_\Z: \xi\prec_{\cZ'} i \} \) listed in increasing order.  
	Then $\ol \Map_i$ has $k$ 2-connected components, which can be written  as $\{\ol\Map^j_i\}_{j\in[1,k]_\Z}$ such that $\cE(\ol\Map^j_i)=\acute{\lambda}([\xi_{j-1},\xi_j-1]_\Z)$ for $j\in[1,k]_\Z$.
	Moreover, let $e^j_i=\acute{\lambda}(\xi_{j-1})$  and  $\wh e^j_i=\acute{\lambda}(\xi_{j}-1)$. Let $\omega^j_i$ be such that 
	$(\ol\Map^j_i, e^j_i  ,\omega^j_i)\in \DP$ with target $\wh e^j_i$ and $\omega^j_i$ agrees with $\omega$
	on inner vertices of $\ol M^j_i$. Then
	\begin{equation}\label{eq:future}
	\cZ|_{[\xi_{j-1},\xi_j]_\Z}=\Phi(\ol\Map^j_i, e^j_i,\omega^j_i).
	\end{equation}
\end{lem}

\begin{proof}
	When $i=\xi_0=0$, we have $k=1$ and $\xi_1=N$. Lemma~\ref{lem:future} is trivial in this case.
	Therefore Lemma~\ref{lem:future} holds when $N=1$. 
	By induction, we assume  that Lemma~\ref{lem:future} holds for maps with fewer edges than $\Map$.
	If $\ol\Map_1$ has a single 2-connected component, then Lemma~\ref{lem:future} holds for $\Map$ by the induction hypothesis. Now suppose that $\ol\Map_1$ has two components $\Map',\Map''$, as in Case~\ref{item:lambda2} of the definition of $\acute{\lambda}|_{[1,\#\cE(\Map)-1]_\Z}$. For $i\in [T(1),N-1]_\Z$ so that $\acute{\lambda}(i)\in \cE(\Map')$,  we can apply the inductive hypothesis to $(\Map',\Be',\omega')$  to get the result.
	If $i\in [0,T(1)-1]_\Z$, so that $\acute{\lambda}(i)\in \cE(\Map'')$, then $\ol\Map^k_i=\Map'$. 
	Moreover, {we have that} $\xi_{k-1}=T(1)-1$ and $\xi_k=N$. Therefore $\cE(\ol\Map^k_i)=\acute{\lambda}([\xi_{k-1},\xi_k-1]_\Z)$  and \eqref{eq:future} holds for $j=k$. 
	 Since $\xi_1,\cdots, \xi_{k-1}$  is the  set \(\{\xi\in[i+1,T(1)-1]_\Z: \xi\prec_{\cZ'} i \} \), Lemma~\ref{lem:future} follows by applying the induction hypothesis to $(\Map'',\Be'',\omega'')$.
\end{proof}

The next lemma describes boundary edges in terms of $\acute{\cZ}$.

\begin{lem}\label{lem:boundary}
	In the setting of Theorem~\ref{thm:bijection}, except for the $c$-step $\Delta\cZ_N$, there are exactly $\elll$ (resp.\ $\ellr$) $c$-steps with no matching $a$-step (resp.\ matching $b$-step), which correspond to the $\elll$ left (resp.\ $\ellr$ right) boundary edges of $(\Map,\Be,\omega)$.
\end{lem}
\begin{proof} 
	We first use the induction on $N$ to show the following.
	For $j\in [1,\elll]_\Z$, let $e$ be the $j$-th left boundary edge when tracing $\bdy \Map$ clockwise from $\Be$ to $\wh\Be$ and $\tau_j=\acute{\lambda}^{-1}(e)$. 
	Then it must be the case that ${\acute\cL}(\tau_j)=\elll-j$ and  $\Delta\acute{\cZ}_{\tau_j+1}=(-1,-1)$. Moreover,  ${\acute\cL}(i)\ge \elll-j$ for $i\in [0, \tau_j]_\Z$. When $e
	\in \cE(\Map')$, we apply the  induction hypothesis and Lemma~\ref{lem:induction} to $(\Map',\Be',\omega')$. Otherwise, {we have} $e\in \cE(\Map'')$, in which case we apply the induction hypothesis and Lemma~\ref{lem:induction} to $(\Map'',\Be'',\omega'')$. We leave the details to the reader. Also see the proof of Lemma~\ref{lem:interface} for a detailed implementation of the same induction scheme. 
	
	The induction above shows that 
	$\tau_j+1=\inf\{i\in [0,N]_\Z: {\acute\cL}(i)=\elll -j-1 \}$. In particular, the step $\Delta\acute{\cZ}_{\tau_j+1}$ corresponding to $e$ is a $c$-step with no matching $a$-step. Since $e\neq \Be$ and $\acute\lambda(N-1)=\Be$,  we see that $\tau_j\le N-2$. The number of such $c$ steps is $\elll$ since $\acute{\cL}|_{[0,N-1]_\Z}$ reaches its record infima exactly at these $c$-steps. 
	Therefore every $c$-step without a matching $a$-step corresponds to a left boundary edge.
	The same argument works for right boundary edges. 
\end{proof}

We conclude this subsection with the target invariance properties of percolation interfaces. 
\begin{definition}\label{def:interface-target}
	Given $\elll,\ellr,n\in \N_0$ and $\el=\elll+\ellr$, let $(\Map,\Be,\omega)\in \cK^{\mrm}(\el)$ and  $\wt \omega$ be such that  $\wt\omega=\omega$ on $\cV(\Map)\setminus \cV(\partial \Map)$ and $(\Map,\Be,\wt\omega)\in \DP(n;\elll,\ellr)$. Let $\wt\lambda$ be the percolation interface  of  $(\Map,\Be,\wt\omega)$.  We call $\wt\lambda$ the percolation interface of $(\Map,\Be,\omega)$ with $(\elll,\ellr)$-boundary condition.  
	
	For an edge $e\in\wt\lambda(\N_0)$ other than the target $\wt \Be$ of $\wt \lambda$, let $i=\wt\lambda^{-1}(e)$ and let $\wt \Map_i$ be the 2-connected component of $\Map-\wt\lambda([0,i-1]_\Z)$ containing $\wt\Be$.  
	Then the 2-connected component of $\wt\Map_i-e$ not containing $\wt \Be$, if it exists, is called the 
	\notion{2-connected component disconnected from $\wt \Be$ when the edge $e$  is peeled by $\wt \lambda$}.
\end{definition} 

\begin{lem}\label{lem:interface}
	In the setting of Definition~\ref{def:interface-target}, let $\wt \tau=\acute{\lambda}^{-1}(\wt \Be)$ so that $\acute{e}_{\wt\tau}=\wt \Be$.	 Let $\tau(0)=0$, $\tau(1)$,$\cdots$, $\tau(\wt m)=\wt\tau$ be the increasing sequence  of ancestor-free times for $\acute{\cZ}$ relative to $\wt \tau$. Then  $\wt\lambda(i)= \acute{\lambda}(\tau(i))$ for all $i\in [0,\wt m]_\Z$.
	Let $(\wt \Map_0,\wt\Be,\wt \omega_0)=(\Map,\Be,\wt \omega)$ and  $(\wt \Map_i,\wt \Be_i,\wt \omega_i)=\Peel(\wt\Map_{i-1},\wt \Be_{i-1},\wt\omega_{i-1})$ for $i\in [1,m]_\Z$.
	For $i\in [0,\wt m-1]_\Z$,  $\tau(i+1)>\tau(i)+1$ if and only if $\wt\Map_i-\wt\Be_i$ has two 2-connected components and $\wh\Be \in  \cE(\wt\Map_{i+1})$.  
	In this case, define $(\wt \Map''_i,\wt \Be''_i,\wt\omega''_i)$ as $(\Map'',\Be'',\omega'')$  in Definition~\ref{def:peel-one} with $(\wt \Map_i,\wt \Be_i,\wt \omega_i)$ in place of $(\Map, \Be, \omega)$.
	Then $\acute{\cZ}|_{[\tau(i)+1, \tau(i+1)]_\Z}=\Phi(\wt \Map''_i,\wt \Be''_i, \wt \omega''_i)$. 
\end{lem}
\begin{proof}
	We will use induction on $N$. The result is immediate for $N=1$ since $\wt\Be=\wh\Be$.
	Assume  Lemma~\ref{lem:interface} is true for maps with less than $N$ edges. 
	When $N>1$ and $v(\Map,\Be)\notin \cV(\bdy\Map)$, then $\tau(1)=\tau(0)+1$ is the first ancestor-free time for $\acute{\cZ}$ relative to $\wt \tau$ after $0$.
	Moreover $\wt\lambda(1)=\lambda(1)$.
	Now Lemma~\ref{lem:interface} follows by applying the induction hypothesis and \eqref{eq:restriction} to $(\Map',\Be',\omega')$.
	
	Now suppose  $N>1$ and $v(\Map,\Be)\in \cV(\bdy\Map)$.
	If $\wt \Be\in \cE(\Map')$ and $\Upsilon$ is as define in Definition~\ref{lem:induction}, we have $\Upsilon\le \wt \tau$. By Lemma~\ref{lem:induction}, we see that $\tau(1)=\Upsilon>1$ and $\wt\lambda(1)=\lambda(1)=\Be'$.
	Moreover, {we have} $\{\tau(2),\cdots, \tau(\wt m) \}\subset [\Upsilon+1,N]_\Z$. Since $(\wt \Map''_0,\wt \Be''_0,\wt\omega''_0)=(\Map'',\Be'',\omega'')$,  we have $\acute{\cZ}|_{[1, \tau(1)]_\Z}=\Phi(\wt \Map''_0,\wt \Be''_0, \wt \omega''_0)$. Now Lemma~\ref{lem:interface} follows by applying the induction hypothesis and \eqref{eq:concatenation} to $(\Map',\Be',\omega')$.
	
	We are left with the case   $N>1$, $v(\Map,\Be)\in \cV(\bdy\Map)$, and $\wt \Be\notin\cE(\Map')$.  In this case,  $\wt \Be\in\cE(\Map'')$ and $\Upsilon>\wt\tau$.
	By the definition of $\Upsilon$, we see that $1$ is ancestor-free for $\acute\cZ$ relative to $\Upsilon-1$, hence also relative to $\wt\tau$. Therefore $\tau(1)=1$. Since $\wh\Be\notin \cE(\wt\Map_1)$, in this case $\wt\Map''_0$ is not defined. Now Lemma~\ref{lem:interface} follows by applying the induction hypothesis and \eqref{eq:concatenation} to $(\Map'',\Be'',\omega'')$.
\end{proof}
\begin{remark}\label{rmk:order}
	By  Lemma~\ref{lem:walk}  the order in which edges are visited for $\acute\lambda$ and $\lambda$ are consistent. The same statement holds for $\acute\lambda$ and $\wt\lambda$ in the setting of Lemma~\ref{lem:interface}. This geometric observation is fundamental to the construction in Section~\ref{sec-sle6-def-discrete}.
\end{remark}

Lemma~\ref{lem:interface} gives the random walk encoding of the 2-connected components of $\Map-\wt\lambda$ that are visited by $\acute{\lambda}$ before $\wt \Be$. 
Suppose $\wt \Be\neq \wh\Be$. Set $i=\wt\tau+1$ in  Lemma~\ref{lem:future} and recall the notations there, including  $\{\xi_j\}_{j\in [0,k]_\Z}$ and $\{\ol\Map^j_{i}\}_{j\in[1,k]_\Z}$. Then the 2-connected components visited after $\wt \Be$ are exactly $\{\ol\Map^j_{i}\}_{j\in[1,k]_\Z}$. Moreover, we can identify the time relative to $\acute{\lambda}$ when these 2-connected components are disconnected from $\wt \Be$ when exploring along $\wt\lambda$ using the following lemma.

\begin{lem}\label{lem:future2}
	In the setting of the above paragraph, for $j\in [1,k]_\Z$, let $\sigma_j$ be such that $\Delta\acute{\cZ}_{\sigma_j}$  is the unique matching step  of the $c$-step $\Delta\acute{\cZ}_{\xi_{j-1}}$, 
	whose existence is ensured by Lemma~\ref{lem:boundary}.  Then both $\acute\lambda(\sigma_j-1)$ and $\acute{\lambda}(\sigma_j)$ are  on the interface $\wt\lambda$.
	Moreover, the map $\ol\Map^{j}_i$ is the 2-connected component disconnected from $\wt \Be$ when the edge $\acute\lambda(\sigma_j-1)$  is peeled by $\wt \lambda$ (recall Definition~\ref{def:interface-target}).
\end{lem} 
\begin{proof}
	Let $m'$ be the largest nonnegative integer such that $\wh\Be$ and $\wt \Be$ are in the same 2-connected component of $\Map-\lambda([0,m']_\Z)$.
	By definition, {we have} $\lambda|_{[0,m']_\Z}=\wt\lambda|_{[0,m']_\Z}$ and $T(m')<\wt\tau< T(m'+1)$, where $T(\cdot)$ is as defined in Lemma~\ref{lem:walk}.
	For any $i' \in[T(m'), T(m'+1)]_\Z$, 
	we see that $\ol\Map_{T(m'+1)}$ 
	is the last 2-connected component of $\ol\Map_{i'}$ visited by $\acute\lambda$. 
	Therefore $\ol\Map_i^{k}=\ol\Map_{T(m'+1)}$ and $\xi_{k-1}=T(m'+1)$. By the definition of $\sigma_k$, we have $\sigma_k=T(m')+1$. By the definition of $m'$, we see that $\acute{\lambda}(\sigma_k-1)=\lambda(m')=\wt \lambda(m')$ is the last common edge of $\lambda$ and $\wt \lambda$. Moreover, the map $\ol\Map^{k}_i$ is the 2-connected component disconnected from $\wt \Be$ when the edge $\wt \lambda(m')$  is peeled by $\wt \lambda$.
	This gives  the desired result for $j=k$.
	The $j<k$ case follows by considering  $(\Map''_{m'}, \Be''_{m'},\omega''_{m'})$ (see Lemma~\ref{lem:component} for definition)  via induction.
\end{proof}

\subsection{A nested percolation interface exploration}
\label{sec-sle6-def-discrete}
Let $(\Map,\Be)$ be a triangulation with simple boundary and let $\omega$ be a site percolation on $(\Map,\Be)$ with monochromatic boundary condition. We identify $(\Map,\Be ,\omega)$ with an element of $\DPm$ as in Remark~\ref{rmk:flip} and let $\acute{\lambda}$ be its space-filling exploration. In this section, we represent $\acute\lambda$ as a nested percolation exploration which will be convenient when we take the scaling limit in Section~\ref{sec-scaling}. {We refer to Figure \ref{fig-2-conn} for an illustration.}

\begin{figure}
	\includegraphics[scale=0.75]{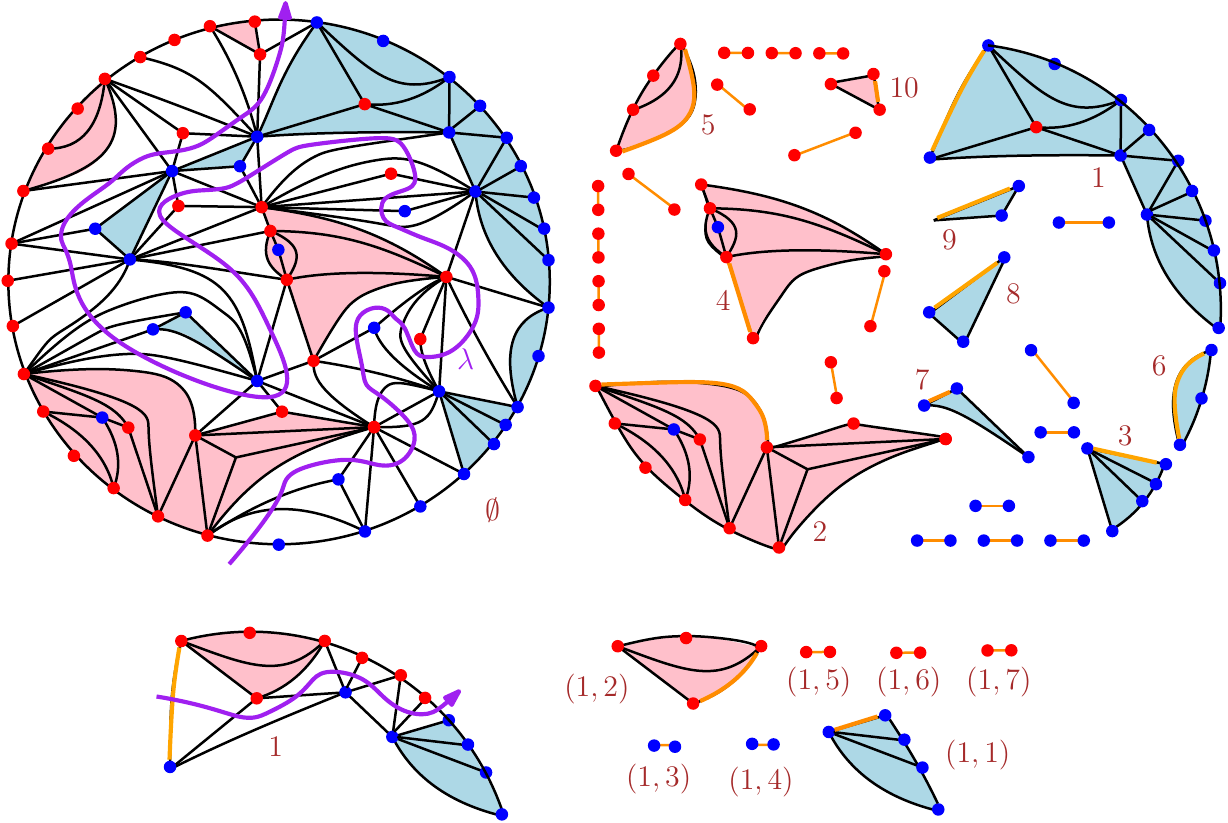}
	\caption{
		{{\bf Top left}: A map $(\Map_{\emptyset},\Be _{\emptyset},\omega_\emptyset)\in \DP(16,17)$ with the percolation interface shown in purple.
		{\bf Top right}: The 2-connected components of $\Map_\emptyset-\lambda_\emptyset$. For $k_1=1,\dots,11$ we have written $k_1$ next to the map $(\Map_{k_1},\Be_{k_1},\omega_{k_1})$. To simplify the drawing, we did not add numbers next to the maps $(\Map_{k_1},\Be_{k_1},\omega_{k_1})$ for $k_1=11,\dots,29$, which consist of only one edge. The edges $\Be_{k_1}$ are marked in orange and we have $\type_{k_1}=\op{mono}$ for all $k_1$. 
	{\bf Bottom}: The maps $(\Map_{(1,k_2)},\Be_{(1,k_2)},\omega_{(1,k_2)})$ for $k_2=1,\dots,7$. We have $\type_{\bk}=\op{mono}$ for $\bk=(1,1),(1,3),(1,4)$ and $\type_{\bk}=\op{di}$ for $\bk=(1,2),(1,5),(1,6),(1,7)$.}}
	\label{fig-2-conn}
\end{figure}

\begin{defn} \label{def-multi-index}
	A \emph{multi-index} is an element of $\bk \in \bigcup_{m=0}^\infty \BB N^m$ (where here $\BB N^0 = \{\emptyset\}$). For a multi-index $\bk = (k_1,\dots,k_m) \in \BB N^m$, we define its \emph{parent} $\bk^- := (k_1,\dots,k_{m-1}) \in \BB N^{m-1}$. By convention, we write $\BB N^0 = \{\emptyset\}$ so that $\bk^- =\emptyset$ 
	for $\bk\in \BB N$. 
	We write $\bk^{-r}$ for the $r$th parent of $\bk$, i.e., the $r$th iterate of the parent function applied to $\bk$.  We define the set of \emph{children} of $\bk$ by
	\eqbn
	\mcl C_{\bk} := \left\{ \bk' \in \BB N^{m+1} : \bk^- = \bk\right\} .
	\eqen 
\end{defn}

For each $n\in\BB N$, we iteratively define for each multi-index $\bk \in \bigcup_{m=0}^\infty \BB N^m$  an element $(\Map_{\bk} , \Be_{\bk},\omega_{\bk})$ in $\DP \cup\{\emptyset\}$. 
We also assign a \emph{type} in $\{\op {mono},\op {di} \}$ to  this element, denoted by $\type_\bk$. 
As we will explain in Remark~\ref{rmk:type}, the strings $\op{mono}$ and $\op{di}$ typically indicate whether  the associated element has monochromatic or dichromatic boundary condition, but there are some exceptions, and the precise definition of $\type_\bk$ is as given below.

For each multi-index $\bk=(\bk^-,k)$, the map $(\Map_{\bk} , \Be_{\bk},\omega_{\bk})$ will be defined as a particular percolated submap of $(\Map_{\bk^-}, \Be_{\bk^-},\omega_{\bk^-})$ (modulo minor modifications to the boundary condition of $\omega_{\bk}$). The idea of the iteration is to consider a percolation interface for $(\Map_{\bk^-}, \Be_{\bk^-},\omega_{\bk^-})$, and let $(\Map_{\bk} , \Be_{\bk},\omega_{\bk})$ be the $k$-th largest complementary component of the percolation interface. If $\type_{\bk^-}=\op{di}$, then the percolation interface is the natural percolation interface between blue and red as defined right after Remark~\ref{rmk:flip}, while if $\type_{\bk^-}=\op{mono}$ then the percolation interface is obtained by considering a percolation exploration towards the boundary edge which is directly opposite $\Be_{\bk^-}$.

We will now give the precise definition of the maps $(\Map_{\bk}, \Be_{\bk},\omega_{\bk})$ and $\type_\bk$. First consider the case $m=0$ and the empty multi-index $\emptyset$. 
Let $(\Map_{\emptyset},\Be _{\emptyset})= (\Map,\Be )$ and let $\el_\emptyset$ be the boundary length of $\Map$. 
Let $\omega_\emptyset$ be the site percolation on $\Map_\emptyset$ such that $\omega_\emptyset$ and $\omega$ agree on $\cV(\Map_\emptyset)\setminus\cV(\bdy \Map_\emptyset)$ and $(\Map_{\emptyset},\Be _{\emptyset},\omega_\emptyset)\in \DP(\lfloor\frac\ell2\rfloor-1, \lceil\frac\ell2\rceil-1)$.  
Let $\lambda_\emptyset$ be the percolation interface of $(\Map_\emptyset , \Be _\emptyset,\omega_\emptyset)$ 
and  let $\wh\Be _\emptyset$ be the target edge of $\omega_\emptyset$. 
We let $\type_\emptyset=\op {mono}$.

Now suppose $m=1$ and $\bk\in \cC_\emptyset$. Then $\bk=k_1$ for some $k_1\in \N$. Let 
$\Map_{\bk}$ be the 2-connected component of $\Map_{\emptyset}- \lambda_{\emptyset}$  with the $k_1$-th largest boundary length,
with ties broken in some arbitrary manner; or let $\Map_{\bk}=\emptyset$ if there are fewer than $k_1$ such components.  Let $\type_\bk=\op{mono}$ (resp.\ $\type_\bk=\op{di}$) if $\Map_\bk$ is a 2-connected component of  $\Map-\lambda_\emptyset$ visited by $\acute \lambda$ before (resp.\ after) $\wh\Be_\emptyset$.
If $\Map_\bk\neq \emptyset$, let $\Be _{\bk}$ be the first edge in  $\cE(\Map_{\bk})$ visited by  $\acute{\lambda}$ and let $\el_{\bk}$ be the boundary length of $\Map_\bk$.

We will now define $\omega_\bk$. If $\type_\bk=\op{mono}$, let $\omega_\bk$ be such that 
\begin{equation}\label{eq:swap}
(\Map_\bk,\Be_\bk, \omega_\bk)\in  \DP\Big(\Big\lfloor\frac{\el_\bk}2\Big\rfloor-1, \Big\lceil\frac{\el_\bk}2\Big\rceil-1\Big)\qquad\textrm{and}\qquad \omega_\bk=\omega\;\textrm{on}\;\cV(\Map_\bk) \setminus \cV(\bdy\Map_\bk).
\end{equation}
For the case $\type_\bk=\op{di}$, let  $i=\acute{\lambda}^{-1}(\wh\Be_\emptyset)+1$ and recall  the notations of Lemma~\ref{lem:future}, including $\ol\Map_i$,
$k$ and $(\ol\Map^j_i, e^j_i  ,\omega^j_i)_{j\in [1,k]_\Z}$.
By Lemma~\ref{lem:future2}, there must be $j\in[1,k]_\Z$ such that $(\Map_\bk,\Be_\bk)=(\ol\Map_i^{j},e^j_i)$. 
We let $\omega_\bk=\omega_i^j$ for this choice of  $j$. 
By Lemma~\ref{lem:future}, the  space-filling exploration of $(\Map_\bk,\Be_\bk,\omega_\bk)$ visits edges of $\cE(\Map_\bk)$ in the same order as $\acute\lambda$. 

For $\Map_\bk\neq \emptyset$, let $\lambda_\bk$ be the percolation interface of $(\Map_\bk , \Be _\bk,\omega_\bk)$ and let $\wh\Be _\bk$ be the target edge.
This concludes our construction for $\bk\in \cC_\emptyset$.

Now we assume $\bk\in\N^2$. If $\Map_{\bk^-}=\emptyset$, set $\Map_{\bk}=\emptyset$.
Otherwise, if $\type_{\bk^-}=\op{di}$, we still let $\Map_\bk$ be defined from $(\Map_{\bk^-},\Be_{\bk^-},\omega_{\bk^-})$ in the same way as $\Map_{\bk}$ is defined from $(\Map_\emptyset,\Be_\emptyset,\omega_\emptyset)$. However, in this case, if $\Map_\bk\neq \emptyset$, then $\Map_{\bk}$ is visited by $\acute{\lambda}$ before $\wh\Be_{\bk^-}$ and $\omega_{\bk^-}|_{\cV(\Map_\bk)}$ is monochromatic.
In light of this we assign $\type_\bk=\op{mono}$ for all $\bk$ such that $\Map_\bk\neq \emptyset$.  Still let  $\Be_\bk$ be the first edge in $\Map_\bk$ visited by $\acute{\lambda}$.  Let $\ell_\bk$, $\omega_\bk$, $\lambda_\bk$, $\wh\Be_\bk$ be defined from $(\Map_\bk,\Be_\bk,\omega|_{\cV(\Map_\bk) \setminus\cV(\bdy\Map_\bk)})$ in the same way as  $\ell_\emptyset$, $\omega_\emptyset$, $\lambda_\emptyset$, $ \wh\Be_\emptyset$ are defined from $(\Map_\emptyset,\Be_\emptyset,
\omega|_{\cV(\Map_\emptyset)\setminus \cV(\partial\Map_\emptyset)})$ described above.

Now let $m>1$ and consider a multi-index $\bk=(k_1,\cdots,k_m)$.
Inductively, suppose our construction has been done for all multi-indices in $\cup_{n=0}^{m-1} \N^n$.
If $\type_{k_1}=\op{mono}$, let $\ol\bk=k_1$ and $\bk'=(k_2,\cdots,k_m)$. If $m>2$, $\type_{k_1}=\op{di}$, and $\type_{(k_1,k_2)}=\op{mono}$, let $\ol\bk=(k_1,k_2)$ and $\bk' =(k_3,\cdots,k_m)$.
In both cases, let $(M_\bk,\Be_\bk,\omega_\bk)$, $\type_\bk$, $\ell_\bk$, $\lambda_\bk$, $\wh \Be_\bk$ be defined from  $(\Map_{\ol\bk}, \Be_{\ol\bk},\omega|_{\cV(\Map_{\ol \bk})\setminus \cV(\bdy \Map_{\ol \bk}) })$ in the same way as $(\Map_{\bk'},\Be_{\bk'}, \omega_{\bk'})$, $\type_{\bk'}$, $\ell_{\bk'}$, $\lambda_{\bk'}$, $\wh \Be_{\bk'}$  are defined from $(\Map_\emptyset,\Be_\emptyset,\omega|_{\cV(\Map_\emptyset)\setminus \cV(\bdy\Map_\emptyset)})$.  
Besides these two cases, we must have $\Map_{\ol\bk}=\emptyset$ and we set $\Map_\bk=\emptyset$.

For each multi-index $\bk$, the iterative construction  allows us define $\Map_\bk$ to be $\emptyset$ or $(\Map_\bk,\Be,\omega_\bk)\in \DP$ with $\type_\bk\in \{\op{mono},\op{di}\}$ and $\ell_\bk,\lambda_\bk,\Be_\bk$ being  the boundary length, percolation interface, and the target edge of $(\Map_\bk,\Be,\omega_\bk)$, respectively. 
If $\Map_\bk\neq \emptyset$,  we call $\Map_{\bk}$ the  \emph{bubble} of index $\bk$ associated with $(\Map,\Be,\omega)$.  We make some basic observations about these bubbles:
\begin{enumerate}[label=(\arabic{enumi})]
	\item $\acute{\lambda}$ visits $\lambda_\bk(\N_0)$ in the same order as $\lambda_\bk$;
	\item   $\type_\bk=\op{mono}$ if and only if  $\bk=\emptyset$ or $\Map_{\bk}$ is visited by $\acute{\lambda}$ before $\wh\Be_{\bk^-}$, in which case 
	$(\Map_{\bk},\Be _{\bk},\omega_\bk)\in \DP(\lfloor\frac{\el_\bk}2\rfloor-1, \lceil\frac{\el_\bk}2\rceil-1)$ and $\omega_\bk=\omega$ on $\cV(\Map_\bk)\setminus \cV(\bdy\Map_\bk)$;  
	\label{item:discrete-bdy1}
	\item $\type_\bk=\op{di}$ if and only if  $\Map_{\bk}$ is visited by $\acute{\lambda}$ after $\wh\Be_{\bk^-}$, in which case  the space-filling exploration of $(\Map_{\bk},\Be _{\bk},\omega_\bk)$  visits edges in $\cE(\Map_\bk)$ in the same order as  $\acute\lambda$. 
	\label{item:discrete-bdy2}
\end{enumerate}
\begin{remark}\label{rmk:type}
	When $(\Map,\Be,\omega)$ is sampled from $\PBT^{\mry}(\el)$, for large $\ell$, the types mono and di typically indicate whether the maps $\Map_{\bk}$ (with coloring as induced from $\Map$) are monochromatic or dichromatic. When $\type_\bk=\op{mono}$, we always have that $\omega_{\bk^-}$ is monochromatic on $\cV(\bdy\Map_\bk)$.
	However, when $\bk^-=\emptyset$ and $\type_\bk={\op {di}}$, {the percolation configuration} $\omega_{\bk^-}$ could still be monochromatic on $\cV(\bdy\Map_\bk)$. For example, it is possible that  $\cV(\bdy\Map_\bk)\subset \cV(\bdy \Map_{\bk^-})$, in which case $\omega|_{\cV(\bdy\Map_\bk)}$ is monochromatic. 
	As we will see later, as $\ell\to \infty$, with probability $1-o_\ell(1)$  $\type_\bk$  honestly indicates whether $\omega$ has monochromatic or dichromatic boundary condition  on $\Map_\bk$.
\end{remark}%

For each $\Map_\bk\neq \emptyset$, let $S_{\bk},T_{\bk}\in \N_0$ be such that  $\acute{\lambda}([S_{\bk},T_{\bk}]_\Z)=\cE(\Map_\bk )$ and let $\wh T_\bk\defeq\acute{\lambda}^{-1}(\wh\Be_\bk) \in [S_\bk , T_\bk]_{\BB Z}$.
Lemmas~\ref{lem:boundary},~\ref{lem:future}, and~\ref{lem:interface}  give the description  of $S_\bk,\wh T_\bk,T_\bk$ in terms of $\acute \cZ=(\acute{\cL},\acute{\cR})$.
\begin{enumerate}[label=(\arabic{enumi})]
	\item If $\type_\bk=\op{mono}$ and $\omega|_{\cV(\bdy\Map_\bk)}$ is red,  then $\wh T_\bk=\inf\left\{t\ge S_\bk: \acute{\cL}_t-\acute{\cL}_{S_\bk}=\lceil\frac{\ell_\bk}{2}\rceil-1 \right\}-1$;
	
	If $\type_\bk=\op{mono}$ and $\omega|_{\cV(\bdy\Map_\bk)}$ is blue, then $\wh T_\bk=\inf\left\{t\ge S_\bk: \acute{\cR}_t-\acute{\cR}_{S_\bk}=\lfloor\frac{\ell_\bk}{2}\rfloor-1 \right\}-1$.  
	
	If $\type_\bk=\op{di}$, then $\wh T_\bk= T_\bk$.  (See Lemma~\ref{lem:boundary}.) \label{item:hat}

	\item
	For $t\in [0,\#\cE(\Map)]_\Z$, let $\ansd(t)$ be the set of ancestor-free times for $\cZ'$ relative to $t$. For each $\Map_\bk\neq \emptyset$,
	the set  of  intervals  $\{ [S_{\bk'}, T_{\bk'}]_\Z : \bk' \in\cC_\bk\textrm{ and }\type_{\bk'}=\op{mono}\}$ equals the set of connected components of  $[S_\bk,\wh T_\bk]_\Z \setminus \ansd(\wh T_\bk)$.
	(See Lemma~\ref{lem:interface}.) \label{item:past}
	
	\item 	If $\type_\bk=\op{mono}$, the set  of intervals  $\{[S_{\bk'}, T_{\bk'}]_\Z: \bk' \in\cC_\bk \textrm{ and } \type_{\bk'}=\op{di}\}$ forms a disjoint union of $[\wh T_\bk+1,T_\bk]_\Z$. Moreover, 
	$t\in \{T_{\bk'}: \bk' \in\cC_\bk \textrm{ and } \type_{\bk'}=\op{di}\}$ if and only if  $t+1\in \{\xi\in[S_\bk,T_\bk]_\Z:  \xi \prec_{\cZ'} 1+\wh T_\bk
	 \}$.
	(See Lemma~\ref{lem:future}.) \label{item:future}

\end{enumerate}

Let $\mcl Z_{\bk} = (\mcl L_{\bk} , \mcl R_{\bk})$ be the boundary length process for $\lambda_{\bk}$. 
By iteratively applying Lemma~\ref{lem-interface-markov}, we have the following Markov property.
\begin{lem} \label{lem-perc-decomp-markov}
	Suppose $(\Map,\Be ,\omega)$ has law of $\PBT^{\mry}(\el)$ for some $\el\in\N_0$.
	For $m\in\BB N_0$, conditioned on $\{\mcl Z_{\bk}: \bk \in \bigcup_{i=1}^m \N^i \}$, the conditional law of $\{(\Map_{\bk'} , \Be _{\bk'} , \omega_{\bk} ): \bk'\in \BB N^{m+1}, \Map_{\bk'}\neq \emptyset \}  $ is that of a collection of independent critical site-percolated Boltzmann triangulations with given boundary condition.
\end{lem}

\section{Precise scaling limit statement and continuum background}
\label{sec-continuum}
In this section we describe the scaling, topology and continuum limit in  Theorem~\ref{thm-metric-peano} precisely. To do this, we will need to review a number of existing results from the literature. Our exposition will be far from self-contained, but we aim to provide all of the background needed to understand the present paper. 

\subsection{The Gromov--Hausdorff--Prokhorov--uniform metric}
\label{sec-ghpu}

We first review the definition of the \emph{Gromov--Hausdorff--Prokhorov--uniform} (GHPU) metric from~\cite{gwynne-miller-uihpq}, the natural generalization of the Gromov--Hausdorff topology to curve-decorated metric measure spaces. We will need the case of spaces decorated by multiple curves, rather than just a single curve. In this setting, which is also explained in~\cite[Section 2.2]{gwynne-miller-perc}, all of the statements and proofs are the same as in the one-curve case treated in~\cite{gwynne-miller-uihpq}. We follow closely the exposition of~\cite[Section 2.2]{gwynne-miller-perc}. 

\subsubsection*{Hausdorff, Prokhorov, and uniform distances}

We first need some basic definitions for metric spaces. 
For a metric space $(X,d)$, $A\subset X$, and $x\in X$, we write $d(x,A)=\sup\{d(x,y): y\in A\}$. For $r>0$, we
let $B_r(A;d)=\{x\in X: d(x,A)\le r \}$. If $A = \{y\}$ is a singleton, we write $B_r(y;d)$ instead of $B_r(\{y\};d)$.  

For two closed sets $E_1,E_2 \subset X$, their \emph{Hausdorff distance}   is given by 
\[
\BB d_d^{\op{H}}(E_1,E_2):=\inf\{r>0: E_1\subset B_r(E_2;{d})\;\textrm{and}\; E_2\subset B_r(E_1;{d}) \} .
\]
Let $\cB(X)$ be the Borel sigma algebra of $(X,d)$.
For two finite Borel measures $\mu_1,\mu_2$ on $X$, their \emph{Prokhorov distance}   is given by   
\[
\BB d^{\op{P}}_d (\mu_1,\mu_2) =\inf \{ \ep>0: \mu_1(A)\le \mu_2(B_\ep(A;d))+\ep \;\textrm{and}\; \mu_2(A)<\mu_1(B_\ep(A;d))+\ep \; \textrm{for all } A\in \cB(X) \}.
\]

Let $f_1: I_1\to X$ and $f_2: I_2 \to X$ be two  functions where $I_1,I_2 \subset \R$ are intervals. 
Their \emph{$d$-Skorokhod distance} is given by  
\[
\BB d^{\op{SK}}_d (f_1,f_2)=\inf_{\phi}\left\{ \sup_{t\in I_1}\left(d(f_1 (t) , (f_2\circ\phi)(t)  )+d(\phi(t),t)\right)\right\},
\]
where $\phi$ ranges over  all strictly increasing, continuous bijections from $[a_1,b_1]$ to $[a_2,b_2]$. 
If $I_1=I_2=\R$, then  the \emph{uniform distance} between $f_1,f_2$ 
is given by 
$$\BB d^{\op{U}}_d(f_1,f_2)=\sup_{t\in \R} d(f_1(t) ,f_2(t)).$$ 
Let $C_0(\BB R , X)$ be the space of continuous curves $\eta : \BB R\rta X$ which extend continuously to the extended real line $[-\infty,\infty]$, i.e., the limits of $\eta(t)$ as $t\rta\infty$ or $t\rta-\infty$ exist.  
For a finite interval $[a,b]$, we can view a curve $\eta : [a,b] \rta X$ as an element of $C_0(\BB R ,X)$ by defining $\eta(t) = \eta(a)$ for $t < a$ and $\eta(t) = \eta(b)$ for $t> b$.  
It is easy to see that 
$\BB d^{\op{SK}}_d $ and $\BB d^{\op{U}}_d$ induce the same topology on $C_0(\R;X)$.
{On the other hand, we will also use Skorokhod  topology for certain discontinuous functions later in the paper; see e.g.\ Theorem~\ref{thm-chordal-conv}.}

\subsubsection*{Definition of the GHPU metric}

For $k\in \BB N$, let $\BM_k^\GHPU$ be the set of $3+k$-tuples $\frk X  = (X , d , \mu , \eta_1,\dots,\eta_k)$ where $(X,d)$ is a compact metric space, $\mu$ is a finite Borel measure on $X$, and $\eta_1,\dots,\eta_k \in C_0(\BB R,X)$.  If we are given elements $\frk X^1 = (X^1 , d^1, \mu^1 , \eta^1_1,\dots , \eta_k^1) $ and $\frk X^2 =  (X^2, d^2,\mu^2,\eta^2_1 , \dots , \eta^2_k) $ of $ \BM_k^\GHPU$ and isometric embeddings $\iota^1 : (X^1 , d^1) \rta (W,D)$ and $\iota^2 : (X^2 , D^2) \rta  (W,D)$ for some metric space $(W,D)$, we define the \emph{GHPU distortion} of $(\iota^1,\iota^2)$ by
\begin{align}
\label{eqn-ghpu-var}
\op{Dis}_{\frk X^1,\frk X^2}^\GHPU\left(W,D , \iota^1, \iota^2 \right)   
:=  \BB d^{\op{H}}_D \left(\iota^1(X^1) , \iota^2(X^2) \right) +   
\BB d^{\op{P}}_D \left(( (\iota^1)_*\mu^1 ,(\iota^2)_*\mu^2) \right) + 
\sum_{j=1}^k \BB d_D^{\op{U}}\left( \iota^1 \circ \eta_j^1 , \iota^2 \circ\eta_j^2 \right) ,
\end{align}
The \emph{Gromov--Hausdorff--Prokhorov--Uniform distance} between $\frk X^1$ and $\frk X^2$ is given by
\begin{align} \label{eqn-ghpu-def}
\BB d^\GHPU\left( \frk X^1 , \frk X^2 \right) 
= \inf_{(W, D) , \iota^1,\iota^2}  \op{Dis}_{\frk X^1,\frk X^2}^\GHPU\left(W,D , \iota^1, \iota^2 \right)      ,
\end{align}
where the infimum is over all compact metric spaces $(W,D)$ and isometric embeddings $\iota^1 : X^1 \rta W$ and $\iota^2 : X^2\rta W$.
Using the same argument in~\cite{gwynne-miller-uihpq}, {we have that} $\BB d^\GHPU$ is a complete separable metric on~$\BM_k^\GHPU$ provided we identify any two elements of~$\BM_k^\GHPU$ which differ by a measure- and curve-preserving isometry. 

Convergence with respect to $\BB d^\GHPU$ can be rephrased in terms of convergence of subsets, measures, and curves embedded in a common metric measure space, as we now explain.  
\begin{defn}[HPU convergence] \label{def-hpu}
	Let $(W ,D)$ be a metric space. Fix $k\in\BB N$.
	Let $\frk X^n = (X^n , d^n , \mu^n , \eta_1^n, \dots$, $\eta_k^n)$ for $n\in\BB N$ and $\frk X = (X,d,\mu,\eta_1,\dots, \eta_k)$ be elements of $\BM^\GHPU_k$ such that $X$ and each $X^n$ is a subset of $W$ satisfying $D|_X = d$ and $D |_{X^n} = d^n$. We say that $\frk X^n\rta \frk X$ in the \emph{$D$-Hausdorff-Prokhorov-uniform (HPU) sense} if $X^n \rta X$ in the $D$-Hausdorff metric, $\mu^n \rta \mu$ in the $D$-Prokhorov metric, and for each $j \in [1,k]_{\BB Z}$, $\eta_j^n \rta \eta_j $ in the $D$-uniform metric.  
\end{defn}
The following result is the $k$-curve variant of~\cite[Proposition~1.9]{gwynne-miller-uihpq} and is identical to~\cite[Proposition 2.2]{gwynne-miller-uihpq}.
\begin{prop} \label{prop-ghpu-embed}
	Let $ \frk X^n = (X^n , d^n , \mu^n , \eta_1^n,\dots , \eta_k^n)$ for $ n\in\BB N $ and $\frk X = (X,d,\mu,\eta_1,\dots , \eta_k)$ be elements of $\BM_k^\GHPU$. Then $\frk X^n\rta \frk X$ in the GHPU topology if and only if there exists a compact metric space $(W,D)$ and isometric embeddings $X^n \rta W$ for $n\in\BB N$ and $X\rta W$ such that the following is true. If we identify $X^n$ and $X$ with their embeddings into $W$, then $\frk X^n \rta \frk X$ in the $D$-HPU sense.
\end{prop}

\subsubsection*{Graphs as curve-decorated metric measure spaces}
Given a graph $G$, let $d_G$ be the graph distance on $\cV(G)$ and let $\mu_{G}$ be the measure on $\cV(G)$ where the mass of each vertex equals half of its degree.   In order to study continuous curves on $G$, we identify each edge of $G$ with a copy of the unit interval $[0,1]$  and extend $d_G$ and $\mu_G$ by requiring that this identification is an isometric measure-preserving embedding  from $[0,1]$ to $(G,d_G,\mu_G)$.
For a discrete interval $[a,b]_\Z$, a function $ \rho : [a,b]_\Z \rta \mcl E(G)$ is called an \emph{edge path} if $\rho(i)$ and $\rho(i+1)$ share an endpoint for each $i\in [a,b-1]_\Z$. We can extend an edge path $\rho$ from $[a,b]_{\BB Z}$ to $[a-1,b] $ in such a way that $\rho$ is continuous and $\rho([i-1,i])$ lies on the edge $\rho(i)$. Note that there are multiple ways to extend $\rho$, but any two different extensions result in curves with uniform distance  at most~$1$. 
If $G$ is a finite graph with  edge paths $\rho_1,\dots,\rho_k$, then  $(G , d_G , \mu_G, \rho_1,\dots,\rho_k)$ is an element of $\BM_k^\GHPU$ under the extensions above. As an example, given $(\Map,\BB e, \omega)\in \DPm$ such that $\Map$ has boundary length $\ell$, let $\acute{\lambda}_\Map$ be the space-filling exploration of $\Map$ based on $\omega$, as in Section~\ref{sec-walk}, and let $\beta_\Map : [0,\el]_{\BB Z} \rta \mcl E(\bdy \Map)$ be defined by tracing $\bdy \Map$ counterclockwise with $\beta(0) = \beta(\ell)=  \BB e$. Then both $\acute{\lambda}$ and $\beta$ are edge paths of $\Map$ and
therefore $(\Map, d_{\Map}, \mu_{\Map}, \beta_\Map,\acute{\lambda}_\Map)$ can be viewed as an element in $\BM_2^\GHPU$. 

\subsubsection*{Space associated with site percolation on a triangulation}

We now give the precise definition of the curve-decorated metric measure space $\acute{\frk M}^n$ in Theorem~\ref{thm-metric-peano}, including the precise scaling for distances, areas, and boundary lengths. Throughout the rest of this paper, we define the scaling constants
\begin{equation}\label{eq:const}
\dcon=(2/3)^{1/4},\qquad \mcon= 3,\qquad \bcon=(3/2)^{1/2}.
\end{equation}
Recall the setting in Theorem~\ref{thm-metric-peano} including $\LL >0$ and the sequence $\{\el^n\}_{n\in\BB N}$. In our new terminology,  $(\Map^n,\BB e^n,\omega^n)$ has the law of $ \PBT^{\mrb}(\el^n)$, where $\omega^n|_{\cV(\bdy \Map^n)}$ can be thought of as monochromatic red in the sense of Remark~\ref{rmk:flip}. 
Let  $d^n=\dcon^{-1} n^{-1/4}d_{\Map^n}$ and $\mu^n=\mcon^{-1} n^{-1}\mu_{\Map^n}$.
Let $\xi_{\bk}^n(s) := \beta_{\Map^n}(\bcon n^{1/2} s)$ for $s\in [0,\bcon^{-1} n^{-1/2} \el^n]$. Let  $\acute\eta^n$ be a reparametrization of $\acute{\lambda}^n$ such that in each unit of time  $\acute\eta^n(t)$ traverses one unit of $\mu^n$-mass of $\Map^n$. Then the precise definition of  $\acute{\frk M}^n$ in Theorem~\ref{thm-metric-peano} is given by 
\begin{equation}\label{eq:frkM}
\acute{\frk M}^n:=\left( \Map^n , d^n , \mu^n , \xi^n , \acute\eta^n \right)\in \BM^\GHPU_2.
\end{equation}

\subsection{Background on  Brownian disk, $\sqrt{8/3}$-LQG and chordal SLE$_6$}
\label{sec-LQG}

We retain the notation $\acute{\frk M}^n$ given in~\eqref{eq:frkM}.  
Fix $\frk l_L , \frk l_R > 0$ such that $\frk l_L+\frk l_R= \LL$ and a sequence of pairs of positive integers $\{(\el_L^n ,\el_R^n)\}_{n\in\BB N}$ such that $\bcon^{-1} n^{-1/2} (\el_L^n,\el_R^n)  \rta (\frk l_L,\frk l_R)$ and $\el_L^n+\el_R^n+2=\ell^n$. Let $\lambda^n$ be the percolation interface of  $(\Map^n,\BB e^n,\omega^n)$ with $(\el_L^n,\el_R^n)$-boundary condition (recall Lemma~\ref{lem:interface}).  Then $\lambda^n$ is an edge path
on $\Map^n$ that can be extended to a curve as in the preceding subsection.\footnote{
	Although the edges of $\lambda^n$ may not be oriented in a consistent manner if we view $\lambda^n$ as an edge path, they can be oriented in a consistent manner if we instead think of $\lambda^n$ as a path on the dual map $\Map^*$.\label{fn:edge-path}}
Set $\BB s=6\BB c^{3/2}$.
For $t\geq 0$ let $\eta^n(t ) := \lambda^n(\tcon n^{ 3/4} t)$. Define 
\eqb
\frk M^n := \left( \Map^n , d^n , \mu^n , \xi^n , \eta^n \right) \in \BM^\GHPU_2 .
\eqe
Note that $\frk M^n$ differs from $\acute{\frk M}^n$ in that it is decorated by only a single interface, instead of the whole space-filling exploration.

It is proved in~\cite{aasw-type2} that $(\Map^n,d^n,\mu^n, \xi^n)$ converges in the GHPU topology to a random curve-decorated metric measure space $\BD_{\LL} :=(H,d,\mu, \xi) \in \BM^\GHPU_1$ called the \emph{Brownian disk with boundary length $\LL$}, which was introduced in~\cite{bet-mier-disk}. The space $(H,d)$ is homeomorphic to a closed disk~\cite{bet-disk-tight} and the curve $\xi$ parametrizes its boundary $\bdy H$ according to its natural length measure. 
For concreteness, we will orient $\bdy H$ by requiring that $\xi$ traces $\bdy H$ in the counterclockwise direction.
The construction of $\BD_{\LL}$ given in~\cite{bet-mier-disk} is based on a standard linear Brownian motion stopped upon reaching $-\LL$ and a Brownian snake on top of it. We refer to \cite{bet-mier-disk} for more details, which are not required to understand our paper. We will need the following basic scaling relation for Brownian disks (see~\cite[Section~2.3]{bet-mier-disk}): if $(H,d,\mu,\xi)$ is a Brownian disk of boundary length $\LL$, then
\eqb \label{eqn-disk-scaling}
(H,\LL^{-1/2}d, \LL^{-2}\mu, \xi(\LL\cdot)) \eqD \BD_1  .
\eqe

We will now describe the scaling limit of the rescaled interface $\eta^n$, as obtained in~\cite{gwynne-miller-perc}.  
Given a simply connected domain $D \subset\BB C$ whose boundary is a continuous curve and two distinct points $a,b\in \bdy D$,
the chordal $\SLE_6$ on $(D,a,b)$  is a random non-simple curve from $a$ to $b$ whose law modulo parametrization is singled out by the so-called conformal invariance, domain Markov property, and target invariance~\cite{schramm0}. We refer to \cite{lawler-book} for more on SLE$_6$ and its basic properties. 
Following Smirnov's breakthrough \cite{smirnov-cardy}, it has been proven that under various topologies, percolation interfaces for critical site percolation on the regular triangular lattice converge to chordal $\SLE_6$ \cite{camia-newman-sle6,gps-pivotal,hls-sle6}.

In light of the scaling limit of percolation interface on the triangular lattice, it is reasonable to expect that $\frk M^n$ defined above converge in law to the ``$\SLE_6$  on the Brownian disk''. 
A priori, the Brownian disk is only defined as a curve-decorated metric measure space, without an embedding into the unit disk $\BB D$, so it is not a priori clear how to define an SLE$_6$ on it. 
However, it was shown in~\cite{lqg-tbm1,lqg-tbm2,lqg-tbm3} that there is a canonical way of embedding the Brownian disk into $\BB D$, i.e., there is a natural way to define a random metric and a random measure on $\ol{\BB D}$ and a random parametrization of $\bdy\BB D$ in such a way that the resulting curve-decorated metric measure space agrees in law with $\BD_{\LL}$. This in particular allows us to define SLE$_6$ on the Brownian disk (Definition~\ref{def:disk-curve}). 

The embedding of the Brownian disk is constructed using the theory of Liouville quantum gravity, which we now briefly review. Our presentation is by no means self-contained. But once Definition~\ref{def:disk-curve} is assumed, detailed knowledge of LQG and SLE is not required to understand the rest of the paper.

Consider a pair $(D,h)$, where $D\subset \BB C$ is a domain and $h$ is some variant of the Gaussian free field (GFF) on $D$ (see \cite{shef-gff,ss-contour,shef-zipper,ig1} for background on the GFF). For $\gamma\in (0,2)$, the random measures $\mu_h=e^{\gamma h}d^2z$ (resp.\ $\nu_h=e^{\gamma h}dz$) supported on $D$  (resp.\ $\bdy D$)  are constructed in\cite{shef-kpz} via a regularization procedure. In this paper, we will only need the special case when $\gamma=\sqrt{8/3}$.  In this case, given two distinct points $a,b\in \bdy D$ and a chordal $\SLE_6$ curve $\eta$  on $(D,a,b)$ independent of $h$, the field $h$ determines a parametrization of $\eta$ called the \emph{quantum natural time} \cite[Definition 6.23]{wedges}. 
Roughly speaking, parametrizing by quantum natural time is equivalent to parametrizing by ``quantum Minkowski content", analogously to the Euclidean Minkowski content parametrization studied in~\cite{lawler-shef-nat,lawler-zhou-nat,lawler-rezai-nat}.
Hereafter we always assume that $\SLE_6$ is given this parametrization. In \cite{lqg-tbm1,lqg-tbm2,lqg-tbm3}, a metric $\frk d_h $ on $D$ depending on $h$ was constructed via 
a growth process called the quantum Loewner evolution~\cite{qle}. 
We will not need the precise definition of this metric here. 

It is shown in~\cite[Corollary 1.4]{lqg-tbm1} that there exists a variant $\fh$ of the GFF on $\D$ (corresponding to the so-called \emph{quantum disk with boundary length $\LL$}) with $\nu_\fh(\bdy \D)=\LL$ and two universal constants $c_{\op d}, c_{\op m} > 0$ such that if $\xi_\fh$ denotes the curve which parametrizes $\bdy\BB D$ according to its $\nu_\fh$ length, then
\begin{equation}\label{eq:disk}
(\D, c_{\op d} \frk d_\fh,  c_{\op m} \mu_\fh, \xi_\fh )\overset{d}{=}\BD_{\LL} \qquad \text{as  random variables in $\BM^\GHPU_1$}.
\end{equation} 
See \cite[Section 4.5]{wedges} for a precise definition of the quantum disk (which will not be needed in this paper). 
The values of $c_{\op d} $ and $c_{\op m}$ are currently unknown. 
It is shown in~\cite{lqg-tbm3}, that the Brownian disk $(\D, c_{\op d} \frk d_\fh,  c_{\op m} \mu_\fh, \xi_\fh )$ (viewed as a curve-decorated metric measure space) a.s.\ determines the field $\frk h$, and hence its parametrization by $\BB D$. 

The equality in law~\eqref{eq:disk} allows us to define chordal SLE$_6$ on the Brownian disk $H$ between two marked points in $ \bdy H$ as follows.

\begin{definition}[SLE$_6$ on the Brownian disk] \label{def:disk-curve}
	For $\frk l_L,\frk l_R>0$, let $\BD_{\frk l_L+\frk l_R} = (H,d,\mu,\xi)$ be a Brownian disk with boundary length $\frk l_L + \frk l_R$. Using~\eqref{eq:disk}, we can identify $\BD_{\frk l_L + \frk l_R}$ with a quantum disk and thereby parametrize our Brownian disk by $\BB D$, i.e., we can take $H=\BB D$. 
	Let $\eta_h$ be a random curve such that conditioned on $\BD_{\frk l_L + \frk l_R}$, the conditional law of $\eta_h$ is a chordal $\SLE_6$ on $(\D,\xi_h(0),\xi_h(\frk l_R))$ parametrized by the quantum natural time with respect to $\fh$.  Then $(D, c_{\op d} \frk d_\fh,  c_{\op m} \mu_\fh, \xi_\fh , \eta_\fh)$ can be viewed as a random variable in  $\BM^\GHPU_2$,  which we call  the  $\SLE_6$-decorated Brownian disk with  $(\frk l_L,\frk l_R)$-boundary condition.  
\end{definition}

Assuming the GHPU convergence $(\Map^n,d^n,\mu^n, \xi^n) \rta \BD_\LL$ (which is proven in~\cite{aasw-type2}), it is proved in \cite[Theorem 8.3]{gwynne-miller-perc} that $\frk M^n$ converges in law in the space $\BM^\GHPU_2$ to an $\SLE_6$-decorated Brownian disk with  $(\frk l_L,\frk l_R)$-boundary condition. 
Indeed, the  procedure $\Peel$ in Section~\ref{sec-perc-interface} is precisely  the so-called \emph{percolation peeling process} for site percolation on triangulations which is studied in detail in~\cite{gwynne-miller-perc} in the setting of face percolation on a quadrangulation.
In fact,~\cite{gwynne-miller-perc} obtains a slightly stronger result which also gives convergence of the boundary length process for the percolation interface (as described in Section~\ref{sec-perc-interface}); see Theorem~\ref{thm-chordal-conv} below.

\subsection{Continuum boundary length process, Markov property, and single interface scaling limit}
\label{sec-sle-markov}

We now review the the definition of the boundary length process for SLE$_6$ on the Brownian disk as well as a continuum domain Markov property analogous to Lemma~\ref{lem-interface-markov}, which is a consequence of results in~\cite{wedges,gwynne-miller-sle6} and the equivalence between the Brownian disk and the $\sqrt{8/3}$-LQG disk. Let us first recall the notion of \emph{internal metric}. Let $(X,d_X)$ be a metric space.
For a curve $\gamma : [a,b] \rta X$, the \emph{$d_X$-length} of $\gamma$ is defined by 
\eqbn
\op{len}\left( \gamma ; d_X  \right) := \sup_P \sum_{i=1}^{\# P} d_X (\gamma(t_i) , \gamma(t_{i-1})) 
\eqen
where the supremum is over all partitions $P : a= t_0 < \dots < t_{\# P} = b$ of $[a,b]$. Note that the $d_X$-length of a curve may be infinite.
For $Y\subset X$, the \emph{internal metric $d_Y$ of $d_X$ on $Y$} is defined by
\eqb \label{eqn-internal-def}
d_Y (x,y)  := \inf_{\gamma \subset Y} \op{len}\left(\gamma ; d_X \right) ,\quad \forall x,y\in Y 
\eqe 
where the infimum is over all curves in $Y$ from $x$ to $y$. 
The function $d_Y$ satisfies all of the properties of a metric on $Y$ except that it may take infinite values. 

Now let $\frk H=(H,d,\mu,\xi, \eta)$ be the $\SLE_6$-decorated Brownian disk with $(\frk l_L,\frk l_R)$-boundary condition as  in Definition~\ref{def:disk-curve}. 
Given $t\ge 0$, let $\cB$ be a  connected component of $H\setminus \eta[0,t]$, let $x$ be the first point on $\partial \cB$ visited by $\eta$,  and let $d_\cB$ be the interval metric of $d$ on $\cB$. A \emph{boundary length} measure can be defined on $\bdy \cB$ by first identifying $(H,d,\mu,\xi)$ with a quantum disk $(\D, \frk h,1)$ and then considering the $\sqrt{8/3}$-LQG measure on $\bdy \cB$ (this is  made sense of  in  \cite{shef-zipper}). 
An equivalent but more intrinsic construction of this boundary length measure is given in \cite{legall-disk-snake} by taking limit of the suitably normalized  $\mu$-measure on a small tubular neighborhood of  $\bdy \cB$.

\begin{figure}[ht!]
	\begin{center}
		\includegraphics[scale=.8]{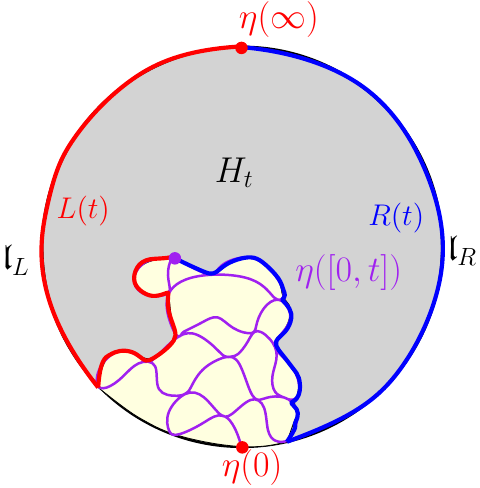} 
	\end{center}
	\caption{\label{fig-continuum-bdy} Illustration of the definition of the left/right boundary length process $Z=(L,R)$ for SLE$_6$ on the Brownian disk.  The red (resp.\ blue) curve has length $L(t)$ (resp.\ $R(t)$). Note that $Z(0) = (\frk l_L , \frk l_R)$ and $Z(\sigma) = (0,0)$. This is the continuum analog of Figure~\ref{fig-tri-perc}, right. }
\end{figure}

Let $\sigma$ be the first time when $\eta$ reaches the target point $\eta(\infty)$ of $\eta$.  (Recall that $\eta$ is viewed as an element of $C_0(\R,H)$.) 
For $t\le\sigma$, let $H_t$ be the connected component of $H\setminus \eta([0,t])$ containing  $\eta(\infty)$ on its boundary.
Let $L(t)$ and $R(t)$ be the boundary lengths of the  clockwise and counterclockwise arcs on $\bdy H_t$ from $\eta(t)$ to $\eta(\infty)$ respectively, we call $Z:=(L,R)$ the \emph{boundary length process} of $\frk H$. 
To describe the law of $Z$, let $L^\infty,R^\infty$ be two independent standard   $3/2$-stable processes starting from 0
with no positive jumps, i.e., the L\'evy process with L\'evy measure $\frac{3}{4\sqrt{\pi}}|x|^{-5/2} 1_{x<0} \rd x$ .  
Let $Z^\infty:=(L^\infty,R^\infty)$. The following lemma is \cite[Theorem 1.2]{gwynne-miller-sle6}.\footnote{In fact, the quantum natural time is defined only up to a multiplicative constant. There is a unique constant such that Lemma~\ref{lem:Levy} holds. This is the way we fix the multiplicative constant.}
\begin{lem}[\cite{gwynne-miller-sle6}]
	Let $\sigma^\infty=\inf
	\{t\ge 0:  L^\infty(t)\le 0 \; \textrm{or} \; R^\infty(t)\le 0 \}$.  
	For each $t>0$, the law of $Z|_{[0,t]} \cdot 1_{t<\sigma}$ is absolutely continuous with respect to $(Z^\infty +(\frk l_L, \frk l_R)) |_{[0,t]}  \cdot 1_{t<\sigma^\infty}$   
	with Radon-Nikodym derivative given by $(L^\infty(t) +R^\infty(t) )^{-5/2} 1_{t<\sigma^\infty}$. Moreover, $\lim\limits_{t\to \sigma}Z(t)=(0,0)$ almost surely.
	\label{lem:Levy} 
\end{lem}

As in the discrete setting, the downward jumps of $Z|_{[0,t]}$ give the boundary length of the complementary connected components of $\eta([0,t])$.
The following lemma, which is a restatement of~\cite[Theorem 1.1]{gwynne-miller-sle6}, is the continuum analog of Lemma~\ref{lem-interface-markov}.

\begin{lem}[\cite{gwynne-miller-sle6}]
	\label{lem-domain-SLE}
	For $t\geq 0$, the connected components of $H\setminus \eta([0,t])$ lying to the left (resp.\ right) of $\eta$ are in one-to-one correspondence with the downward jumps of $L$ (resp.\ $R$) up to time $t$. 
	If $t\geq 0$ and we condition on $Z|_{[0,t]}$, then the conditional law of the connected components of  $H\setminus \eta([0,t])$, each equipped with the internal metric of $d$, the restriction of $\mu$, and the path which parametrizes its boundary according to the natural length measure, are Brownian disks with given boundary lengths.
	Moreover, $Z$ together with this collection of Brownian disks for $t=\sigma$ determines $\frk H$ a.s. 
\end{lem}
\begin{proof}
	The first two statements from \cite[Theorem 1.1]{gwynne-miller-sle6} and the equivalence of Brownian disk and $\sqrt{8/3}$-LQG disks~\cite[Corollary 1.4]{lqg-tbm2}.
	The second statement follows from \cite[Theorem 1.16]{wedges} and local absolute continuity (see the proof of~\cite[Lemma 3.8]{gwynne-miller-sle6} for a similar argument). 
\end{proof}
\begin{remark}
	It is shown in~\cite[Theorem 7.12]{gwynne-miller-char} that SLE$_6$ on the Brownian disk is uniquely characterized by the Markov property of Lemma~\ref{lem-domain-SLE} together with the topology of the curve. This statement is a key input in the proof of the scaling limit result for a single interface in~\cite{gwynne-miller-perc}.
\end{remark} 

Recall the triangulation decorated by a percolation interface $\frk M^n=\left( \Map^n , d^n , \mu^n , \xi^n , \eta^n \right)$ from the beginning of Section~\ref{sec-LQG}. Let $\cZ^n=(\cL^n,\cR^n)$ be the boundary length process of the percolation interface $\lambda^n$ as defined in Section~\ref{sec-perc-interface}.  Let  
\begin{equation}\label{eq:renormalized-Z}
Z^n(t) := (L^n(t) , R^n(t)):=  \bcon^{-1} n^{-1/2}\mcl Z^n(\lfloor \tcon t n^{3/4} \rfloor)\qquad \textrm{for all}\; t\geq 0 ,
\end{equation}
where $\BB s = 6\BB c^{3/2}$.\footnote{The constant $\BB s$ is not explicitly given in \cite{gwynne-miller-char}. With our choice of normalizations, it follows from \cite[Page 23]{angel-curien-uihpq-perc}  that $\BB s=6\BB c^{3/2}$.}
The starting point of the proof of Theorem~\ref{thm-metric-peano} is the following theorem, which follows from results in~\cite{gwynne-miller-perc,aasw-type2}.

\begin{thm}\label{thm-chordal-conv}
	The pair $(\frk M^n , Z^n) $ converges in law to the SLE$_6$ decorated quantum disk $\frk H$ together with its boundary length process $Z$  with respect to the GHPU topology on the first coordinate and the Skorokhod topology for c\`adl\`ag processes on the second coordinate.
\end{thm}

Theorem~\ref{thm-chordal-conv} is proved as \cite[Theorem~8.3]{gwynne-miller-perc} conditioned on the assumption that $(\Map^n,d^n,\mu^m,\xi^n)$ converge to $(H,d,\mu,\xi)$ in the GPHU topology. This convergence is proved in \cite{aasw-type2}.

\subsection{Nested exploration, space-filling SLE$_6$, and  Brownian motion}
\label{sec-mating}  

In this section, we  describe the continuum limit of $\acute{\lambda}^n$ in Theorem~\ref{thm-metric-peano}, which is the space-filling $\SLE_6$ on the Brownian disk.
This will allow us to give a precise statement of our main theorem (Theorem~\ref{thm:main-precise}). 

The space-filling $\SLE_6$ is a conformally invariant random space-filling curve whose precise definition relies on the full strength of imaginary geometry developed in \cite{ig1,ig4}. In Theorem~\ref{thm:space-fillng} we give an alternative construction based on the continuum analog of the nested exploration procedure in Section~\ref{sec-sle6-def-discrete}, and we use the imaginary geometry construction as a black box in the proof of the theorem to argue existence of the curve. The construction in Theorem~\ref{thm:space-fillng} provides all information about space-filling $\SLE_6$ which is needed for Theorem~\ref{thm-metric-peano}.

Start with a sample $(H,d,\mu,\xi)$ of the Brownian disk $\BD_\LL$.
We will iteratively define for each multi-index $\bk \in \bigcup_{m=0}^\infty \BB N^m$ a Brownian disk decorated by an independent chordal SLE$_6$ curve $\frk H_{\BB k} = (H_{\BB k} , d_{\BB k} , \mu_{\BB k} , \xi_{\BB k } , \eta_{\BB k})$ in a manner analogous to the discrete construction in Section~\ref{sec-sle6-def-discrete}. See Figure~\ref{fig-sle6-decomp} for an illustration.

\begin{figure}[ht!]
	\begin{center}
		\includegraphics[scale=.78]{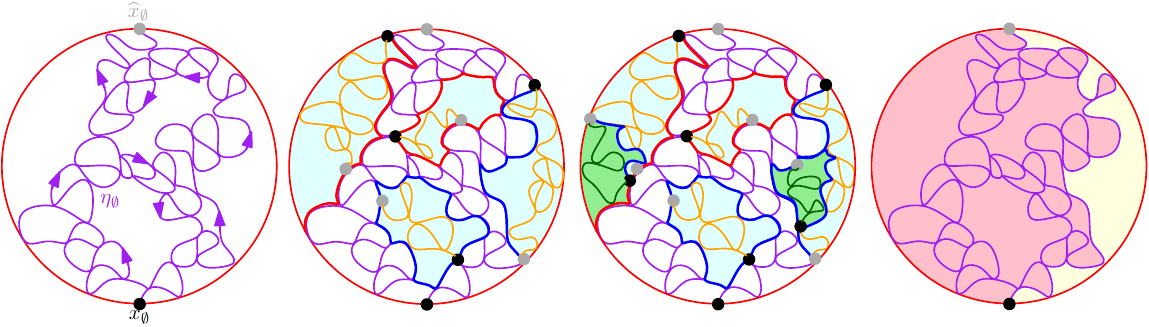} 
	\end{center}
	\caption{\label{fig-sle6-decomp} {\textbf{Left (three figures):} The main Brownian disk $H$ (shown as the unit disk) with several stages of the nested chordal SLE$_6$ process which we use to construct the space-filling SLE$_6$ $\eta'$. The stage-0 curve $\eta_\emptyset$ is shown in purple in the first figure. Four stage-1 domains $ H_{\bk} $ for $\bk \in \BB N$ are shown in light blue in the second figure, and their associated curves $\eta_\bk$ are shown in orange. Two stage-2 domains $H_{\bk}$ for $\bk \in \BB N^2$ are shown in light green in the third figure, and their associated chordal SLE$_6$ curves $\eta_\bk$ are shown in dark green. The coloring of the boundary of each of the domains, as defined in Section~\ref{sec-mating}, is also shown. The initial and target points $x_{\bk}$ and $\wh x_{\bk}$ for the curves $\eta_\bk$ are shown as black and gray dots, respectively; note that these are necessarily the points separating the red and blue arcs for domains with dichromatic boundary.}
		\textbf{Right:} The Brownian disk $H$ and the curve $\eta_\emptyset$ (still in purple) with points colored based on the order in which they are hit by $\eta'$: first $\eta'$ traces $\eta_\emptyset$ in order and fills in each of the bubbles in the pink region (which are monochromatic) immediately after tracing the segment of $\eta_\emptyset$ which cuts it off. Then, $\eta'$ fills in the dichromatic bubbles (yellow region) in order as it travels from $\wh x_\emptyset$ to $x_\emptyset$. }
\end{figure}

Let $H_\emptyset=H$, $d_\emptyset = d$, $\mu_\emptyset=\mu$, $\xi_\emptyset= \xi$, and $x_\emptyset= \xi_\emptyset(0)$. Let $\ell_\emptyset$ be the total length of $\bdy H_\emptyset$ and  $\wh x_\emptyset = \xi_\emptyset(\ell_\emptyset /2)$.
Let $\eta_\emptyset$ be a chordal $\SLE_6$ on $(H,x_\emptyset,\wh x_\emptyset)$. 
Set $\frk H_\emptyset=(H_\emptyset, d_\emptyset, \mu_\emptyset, \xi_\emptyset,\eta_\emptyset)$ so that $\frk H_\emptyset$ satisfies  $(\ell/2,\ell/2)$-boundary condition (Definition~\ref{def:disk-curve}).  
For convenience, we call the counterclockwise arc of $\bdy H$ from $x_\emptyset$ to $\wh x_\emptyset$ --- i.e., the set $\xi_\emptyset([0,\ell/2])$ --- the \emph{right} boundary arc of $H$ and its complementary arc the \emph{left} boundary arc of $H$.

Inductively, suppose $m\in\BB N$ and $\frk H_{\bk'}= (H_{\bk'},d_{\bk'},\mu_{\bk'},\xi_{\bk'},\eta_{\bk'})$,  $\ell_{\bk'}$, $x_{\bk'}$, and $\wh{x}_{\bk'}$  have been defined  for each $\bk'\in \BB N^{m-1}$
in such a way that $x_{\bk'}=\xi_{\bk'}(0)=\eta_{\bk'}(0)$, $\wh x_{\bk'}=\eta(\infty)$, and $\ell_{\bk'}$ is the boundary length of $\bdy H_{\bk'}$. 
Let $\bk = (k_1,\dots,k_m) \in \BB N^m$ and let $\bk^- := (k_1,\dots,k_{m-1}) \in \BB N^{m-1}$ be its ancestor. 
Let $H_\bk $ be  the closure of the (almost surely unique) connected component of $H_{\bk^-}\setminus \eta_{\bk^-}$  with the  $k_m$-th largest boundary length.
We call $H_{\bk}$ the  \emph{bubble} of index $\bk$. 
  
Let $d_{\BB k}$ be the internal metric of $d_{\BB k^-}$ on $H_{\BB k}$, let $\mu_{\BB k} := \mu_{\BB k^-} |_{H_{\BB k}}$, and let $\xi_{\BB k}$ be the path which traverses $\bdy H_{\BB k}$ according to its natural boundary length measure as in Section~\ref{sec-sle-markov}, started from the point $x_{\BB k} = \xi_{\BB k}(0)$ where $\eta_{\BB k^-}$ finishes tracing $\bdy H_{\BB k}$ and oriented in the counterclockwise (resp.\ clockwise) direction if $H_{\BB k}$ is to the left (resp.\ right) of $\eta_{\BB k}$. We know from Lemma~\ref{lem-domain-SLE} that $(H_{\BB k}, d_{\BB k} , \mu_{\BB k} , \xi_{\BB k})$ is a Brownian disk conditional on its boundary length. 
Let $\ell_{\BB k}$ be this boundary length.

In order to choose the target point for the chordal SLE$_6$ curve $\eta_{\BB k}$, we need to specify a notion of ``color" for  the boundaries of the bubbles $H_{\BB k}$ analogous to the coloring of the boundaries of the bubbles in Section~\ref{sec-sle6-def-discrete}.  We say that the $H_\emptyset$ is \emph{monochromatic red}. 
Moreover, by induction we assume that the color of $H_{\bk'}$ is defined for each $\bk'\in \N^{m-1}$.
\begin{enumerate}[label=(\arabic{enumi})]
	\item If either $H_\bk \cap \bdy H_{\bk^-} = \emptyset$ or $\bdy H_{\BB k} \cap \bdy H_{\bk^-}$ is traced in the same order by $\xi_{\BB k}$ and $\xi_{\BB k^-}$, we say that $\bdy H_{\BB k}$ is \emph{monochromatic}. We say that it is \emph{monochromatic red} (resp.\ \emph{monochromatic blue}) if $H_{\BB k}$ lies to the left (resp.\ right) side of $\eta_{\BB k^-}$, equivalently if $\xi_{\BB k}$ travels in the counterclockwise (resp.\ clockwise) direction. In this case, we let $\wh x_\bk = \xi_\bk(\ell_\bk/2)$. \label{item:color1}
	\item If $H_\bk \cap \bdy H_{\bk^-} \not= \emptyset$ and $\bdy H_{\BB k} \cap \bdy H_{\bk^-}$ is traced in the opposite order by $\xi_{\BB k}$ and $\xi_{\BB k^-}$, we say that $\bdy H_{\BB k}$ is \emph{dichromatic}. In this case, we let $\wh x_\bk$ be the common endpoint of $\bdy H_{\bk} \cap \bdy H_{\bk^-}$ and $\bdy H_{\bk} \setminus \bdy H_{\bk^-}$ other than $x_\bk$. \label{item:color2}
\end{enumerate}
With these definitions, {we have that} $H_{\BB k}$ is monochromatic whenever $H_{\BB k^-}$ is dichromatic. If $H_{\BB k^-}$ is monochromatic red (resp.\ blue), then $H_{\BB k}$ is dichromatic if and only if $H_{\BB k}$ intersects the right (resp.\ left) boundary of $H_{\BB k^-}$.
We now let $\eta_{\BB k}$ be a chordal SLE$_6$ on the Brownian disk $H_{\BB k}$ from $x_{\BB k}$ to $\wh x_{\BB k}$ (using Definition~\ref{def:disk-curve}). This completes the inductive construction of our SLE$_6$-decorated Brownian disks $\frk H_{\BB k} = (H_{\BB k},d_{\BB k} , \mu_{\BB k} , \xi_{\BB k} ,\eta_{\BB k})$.

For each multi-index $\bk$, let $Z_{\bk}  = (L_{\bk} , R_{\bk})$ be the  boundary length process associated with $\frk H_\bk$ as in Section~\ref{sec-sle-markov}. 
Then one has the following domain Markov property by iteratively applying Lemma~\ref{lem-domain-SLE}.

\begin{lem} \label{lem-sle6-decomp-markov}
	Let $m\in\BB N_0$. Conditioned on  $\{Z_{\bk}\}_{\bk \in \BB N^m}$,  the conditional law of $\{\frk H_{\bk'}\}_{\bk' \in \BB N^{m+1}}$ is that of a collection of independent $\SLE_6$-decorated  Brownian disks with specified left/right boundary  condition. 
\end{lem}  

Now we move  to construct the space-filling $\SLE_6$ on the Brownian disk following the discrete intuition in Section~\ref{sec-discrete}. 
For each multi-index $\bk$ and child multi-indices $\bk_1,\bk_2 \in \cC_\bk$, if $H_{\bk_1}$ is monochromatic and $H_{\bk_2}$ is dichromatic, then  we write $H_{\bk_1}\prec_\bk H_{\bk_2}$. If both $H_{\bk_1}$  and $H_{\bk_2}$ are monochromatic (resp.\ dichromatic) and the time at which $\eta_{\BB k}$ finishes tracing $\bdy H_{\bk_1}$ is smaller than that for $\bdy H_{\bk_2}$, then we write $H_{\bk_1}\prec_\bk H_{\bk_2}$  (resp.\  $H_{\bk_2}\prec_\bk H_{\bk_1}$). By the color convention for bubbles, this uniquely defines $\prec_\bk$ as a total ordering $\prec_\bk$  on $\{H_{\bk'} \}_{\bk'\in \cC_\bk}$.

The following theorem is essentially proven in  \cite[Section 4.3]{ig4}. 

\begin{thm}\label{thm:space-fillng}
	In the above setting, with probability 1, there exists a unique continuous curve $\eta' : [0,\mu(H)] \rta H$ with the following properties: 
	\begin{enumerate}[label=(\arabic{enumi})]
		\item $\eta'(0)   =x_\emptyset$, $\eta'([0,\mu(H)])=H$, and  $\mu( \eta'([0,s]) ) = s$  for each $s\in [0,\mu(H)]$; \label{item-sfsle1}
		\item For each multi-index $\bk$, given any $\bk_1, \bk_2 \in \cC_\bk$ and times $s_1,s_2$ such that $\eta'(s_i)$ is in the interior of  $H_{\bk_i}$ for $i=1,2$, we have $H_{\bk_1}\prec H_{\bk_2}$ if and only if $s_1<s_2$.   \label{item-sfsle2}
	\end{enumerate}
	As a random variable in $ \BM^\GHPU_2$, {the tuple} $\frk H':=(H,d,\mu,\xi,\eta')$ is called the the space-filling $\SLE_6$-decorated Brownian disk with boundary length $\LL$.
\end{thm}

In words, $\eta'$ is the unique curve that visits $\{H_{\bk'}\}_{\bk'\in \cC_\bk}$ in the order of $\prec_\bk$ for each multi-index $\bk$ and is parametrized by $\mu$-mass.
Note that the definition of our ordering implies that $\eta'(0) =\eta'(\mu(H)) = x_\emptyset$.
We first record a lemma which is useful for the uniqueness part of Theorem~\ref{thm:space-fillng}.
\begin{lem}\label{lem:area}
	Let $z$ be a point sampled according to $\mu$.  For $m\in \N$, let $\bk_m$ be the index of the $m$-th largest bubble containing $z$. Then $\lim_{m\to\infty}\mu(H_{\BB k_m})=0$ a.s.
\end{lem}
\begin{proof}
	Let $(H,d,\mu,\xi)$ be embedded in $(\D,1)$ so that $z$ is mapped to $0$. By the target invariance of $\SLE_6$, the concatenation of $\eta_{\bk_m}$ from $x_{\bk_m}$ to $x_{\bk_{m+1}}$ for all $m\in \N$ gives  the so-called radial $\SLE_6$ on $\D$ starting from 1 and  targeting at $0$. 
	In particular, the Euclidean diameter of $H_{\bk_m}$  a.s.\ shrinks to $0$. Since $\mu$ is  a.s.\ non-atomic,  this concludes the proof. 
\end{proof}

\begin{proof}[Proof of Theorem~\ref{thm:space-fillng}]
	Conditional on $(H,d,\mu,\xi)$, suppose $z,z'\in H$ are two points sampled independently from $\mu$ (normalized to be a probability measure). Let $\{H_{\BB k_m}\}_{m\in\N}$ be the sequence of bubbles containing $z$ in decreasing order. Then Lemma~\ref{lem:area} implies that $\lim_{m\to\infty}\P[z'\in H_{\BB k_m}]=0$.
	Therefore  $z$ and $z'$ are almost surely contained in two disjoint bubbles with a common ancestor.
	Now let $\{z_n\}_{n\in\N}$ be a sequence of points in $H$ sampled independently  according to $\mu$.  Then  almost surely   any two points in  $\{z_n\}_{n\in\N}$ are contained in two disjoint bubbles with a common ancestor. Therefore two space-filling curves satisfying Condition~\ref{item-sfsle2} must visit $\{z_n\}_{n\in \N}$ in the same order. Since $\{z_n\}_{n\in\N}$ is almost surely dense in $H$,  these two curves must agree if they both satisfy Condition~\ref{item-sfsle1}. This gives the  uniqueness part of Theorem~\ref{thm:space-fillng}.
	
	For the existence part, it suffices to show that the ordering of $\{z_n\}_{n\in\N}$ above can be extended to a continuous curve.  
	As explained in \cite[Section 4.3]{ig4}, this ordering is the same as the one coming from the definition in the introduction of \cite{ig4} based on the flow lines of the GFF. 
	In light of this, the desired  continuous extension is achieved in \cite[Theorem 1.16]{ig4}.  
\end{proof}

For $\bk\in \bigcup_{m=0}^\infty \BB N^m$,  there exists a unique interval $[s_\bk , t_{\bk}]$ such that $H_\bk=\eta'[s_\bk,t_\bk]$. We also write $\wh t_{\bk} \in [s_\bk , t_{\bk}]$ for the a.s.\ unique time in this interval such that $\eta'(\wh t_\bk)=\wh x_{\bk}$. Note that $\wh t_{\bk} = t_{\bk}$ for dichromatic bubbles. Then $\eta_\bk$ can be obtained from $\eta'|_{[s_\bk , t_{\bk}]}$ by skipping the times during which it is filling in bubbles disconnected from $\wh x_{\bk}$ and then parameterizing by quantum natural time.
By the definition of the ordering of bubbles above, the curve $\eta'|_{[s_\bk , \wh t_{\bk}]}$  fills in each monochromatic bubble cut out by $\eta_\bk$ immediately after it finishes tracing its boundary, but does not fill in the dichromatic bubbles cut out by $\eta_\bk$ until the time interval $[\wh t_{\bk} , t_{\bk}]$.  In particular, $\eta'(\mu(H))=\eta'(0)=\xi(0)$. This means that $\eta'$ is a space-filling loop based at $\xi(0)$.

If we view $(H , d , \mu , \xi)$ as being embedded into $\BB D$ via~\eqref{eq:disk}, then the curve $\eta'$ modulo parametrization is determined by the trace of chordal $\SLE_6$ curves $\{\eta_\bk:\bk\in \bigcup_{m=0}^\infty \N^m \}$, so is  independent from the field $\frk h$ which describes the associated quantum disk and hence also from $(H,d,\mu,\xi)$.  
This puts the space-filling SLE$_6$ decorated Brownian disk $\frk H' = (H,d,\mu,\xi,\eta')$ into the mating-of-trees framework developed  in \cite{wedges} and gives the continuum analog of the random walk encoding in Section~\ref{sec-walk}.

To be more precise, by identifying $(H,d,\mu)$ with a quantum disk as in~\eqref{eq:disk}, we can (using~\cite{shef-zipper}) define for each $t\geq 0$ the quantum lengths of the clockwise and counterclockwise arcs of $\eta'[t,\infty]= \ol{H\setminus\eta'[0,t]}$ from $\eta'(t)$ to $\eta(\mu(H)) = x_\emptyset$. Denote these lengths by $L_t'$ and $R_t'$, respectively.
We call $Z':=(L',R')$ the \emph{boundary length process} of $\frk H'$. Note that $\bdy H$ counts as part of the left boundary.
The following fundamental result is immediate from~\cite[Theorem 2.1]{sphere-constructions} and the equivalence of the $\sqrt{8/3}$-quantum disk and the Brownian disk. 
\begin{thm}[\cite{sphere-constructions}]
	\label{thm:mating}
	There is a deterministic constant $c > 0$ such that the law of $Z'$ can be described as follows. 
	Let $(X_t,Y_t)_{t\ge 0}$ be a pair of correlated linear Brownian motions with
	\eqbn
	\op{Var}(X_t) = \op{Var}(Y_t) =\frac23 t \quad \op{and} \quad \op{Cov}(X_t,Y_t) =   \frac13 t,\quad \forall t \geq 0 , 
	\eqen
	started from $(X_0, Y_0) = (\frk l , 0)$.
	For $\ep>0$, let $\tau_\ep=\inf\{t\ge 0: X_t=-\ep\textrm{ or }Y_t=-\ep\}$. 
	The conditional law of $(X_t,Y_t)_{[0,\tau_\ep]}$ given that $X_{\tau_\ep} \leq \ep$ and $Y_{\tau_\ep} \leq \ep$ converges to the law of $(Z'_t)_{[0,\mu(H)]}$ as $\ep\rta 0$.
	
	Moreover, $Z'$ and $\frk H'$ are measurable with respect to each other.
\end{thm}

The process $Z'$ is illustrated in Figure~\ref{fig-disk-bm}. 

\begin{figure}[ht!]
	\begin{center}
		\includegraphics[scale=.8]{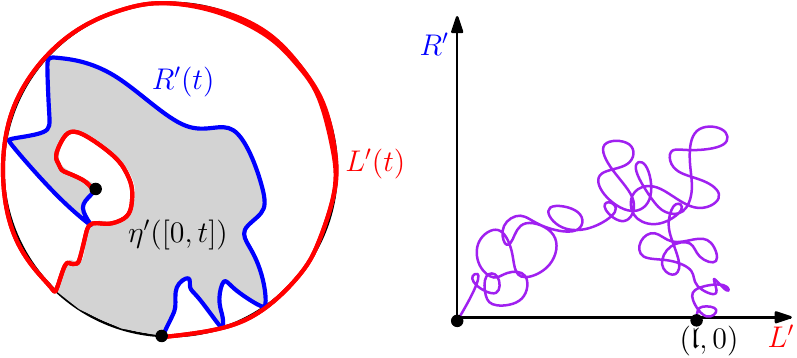} 
	\end{center}
	\caption{\label{fig-disk-bm} \textbf{Left:} The process $Z' = (L',R')$ gives the left (red) and right (blue) boundary lengths of $H\setminus \eta'([0,t])$ for each $t\geq 0$. \textbf{Right:} By Theorem~\ref{thm:mating}, the law of $Z'$ is that of a planar Brownian motion with correlation $1/2$, started from $(\frk l ,0)$, and conditioned to first exit $[0,\infty)^2$ at $(0,0)$. This conditioning is made sense of via a limiting procedure. }
\end{figure}

\begin{remark}[A comment on constants]
	In \cite{wedges,sphere-constructions}, the variance of the Brownian motion encoding a space-filling SLE$_6$-decorated quantum disk is unknown; only the correlation is specified. However, in our setting the variance has the explicit value $\frac23$ because  $\mu(H)$ is known to have density  $\frac{1}{\sqrt{2\pi t^5}}\exp\{-\frac1{2t}\}1_{t>0}dt$ when $\LL=1$ \cite{bet-mier-disk}. 
	{Recently it has been shown in~\cite[Theorem 1.3]{ars-fzz} that the Brownian motion in~\cite{wedges,sphere-constructions} has variance $\frac{4}{\sqrt3}$.
	Therefore the area constant $c_{\op m}$ in \eqref{eq:disk}  satisfies  $\frac{4c_{\op m}}{\sqrt3}=\frac{2}{3}$ hence $c_{\op m}= \frac{\sqrt3}6$.}
\end{remark}

We  conclude this section with a precise statement of Theorem~\ref{thm-metric-peano}.
\begin{thm}\label{thm:main-precise}
	Given $\LL >0$, let $\acute{\frk M}^n \in \BM^\GHPU_2$ for $n\in\BB N$ be the site-percolated triangulations as in \eqref{eq:frkM}. Let $\acute{\cZ}^n=\Phi(M^n,\Be^n,\omega^n)$ be the encoding walk as in Theorem~\ref{thm:bijection} and let
	\eqb \label{eqn-rescaled-walk}
	\acute Z^n(t) := (\acute L^n(t) , \acute R^n(t))=\bcon^{-1} n^{-1/2} \acute{\mcl Z}^n(\lfloor \mcon n t \rfloor)\qquad \textrm{for all}\;  t \in [0,\mu^n(\Map^n)].
	\eqe 
	Let $\frk H',Z'$ be the space-filling SLE$_6$-decorated Brownian disk and its associated boundary length process as above. Then the joint law of  $(\acute{\frk M}^n , \acute Z^n)$ converges the that of $(\frk H',Z')$ with respect to the GHPU topology on the first coordinate and the uniform topology on the second coordinate. 
\end{thm}

\subsection{The conformal loop ensemble associated with $\eta'$}
\label{sec-cle-def}  

The space-filling SLE$_6$ curve $\eta'$ traces the loops of a conformal loop ensemble (CLE$_6$;~\cite{shef-cle}) on $H$, as we now explain. 
Given $\bk\in \bigcup_{m\in \N_0}\N^m$, if $H_\bk$ is dichromatic, we define a loop $\gamma_\bk$ as follows. 
Let $\bar \sigma_\bk$ and $\bar{\tau}_\bk$ be the almost surely unique times when the chordal SLE$_6$ curve $\eta_{\bk^-}$ visits the points $\eta'(t_\bk)$ and $\eta'(s_\bk)$, respectively, where $\eta'$ starts and finishes filling in $H_{\bk}$. 
Let $\gamma_\bk(t)=\eta_{\bk^-}(t+\bar\sigma_\bk)$ for $t\in [0,\bar{\tau}_\bk-\bar{\sigma}_\bk]$ and $\gamma_\bk(t)=\eta_\bk(t+\bar{\sigma}_\bk-\bar{\tau}_\bk)$ for $t>\bar{\tau}_\bk-\bar{\sigma}_\bk$. 
In words, $\gamma_\bk$ is the parametrized loop obtained by concatenating $\eta_{\bk^-}|_{[\bar{\sigma}_\bk,\bar{\tau}_\bk]}$ and $\eta_\bk$.  
See Figure~\ref{fig-cle-loop} for an illustration.

\begin{figure}[ht!]
	\begin{center}
		\includegraphics[scale=.6]{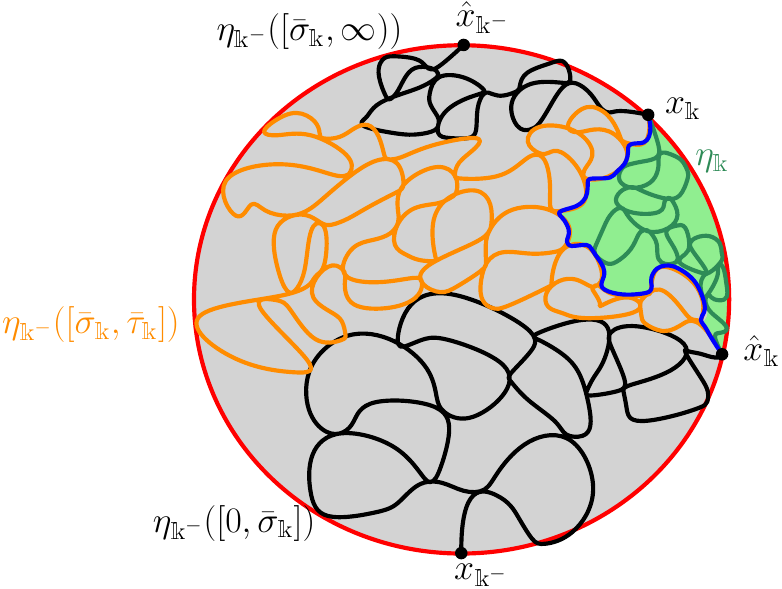} 
	\end{center}
	\caption{\label{fig-cle-loop} Shown are two nested bubbles $H_{\bk^-}$ (the whole disk) and $H_{\bk}$ (light green) with $H_\bk$ dichromatic. The CLE$_6$ loop $\gamma_{\bk}$ is the concatenation of the orange curve and the green curve. The collection of all such loops, as $\bk$ ranges over all dichromatic bubbles, is a CLE$_6$. }
\end{figure}

\begin{definition}\label{def:CLE}
	The collection of loops $\Gamma := \left\{\gamma_\bk: \text{$\bk\in \bigcup_{m\in \N_0}\N^m$ and $H_\bk$ is dichromatic} \right\}$ is called $\CLE_6$ on the Brownian disk $H$ under its natural parametrization.
\end{definition}

Definition~\ref{def:CLE} is equivalent to the construction of $\CLE_6$ on $\D$ given in \cite[Section 3]{camia-newman-sle6} once $(H , d , \mu , \xi)$ is viewed  as being embedded into $\BB D$ via~\eqref{eq:disk} and $\Gamma$ is viewed a collection of unrooted, oriented, unparametrized loops on $\D$, which in turn is equivalent to the one defined in \cite{shef-cle}.

\section{Scaling limit of the nested exploration}
\label{sec-scaling}
We retain the notations in Theorem~\ref{thm-metric-peano}, which was just made precise by Theorem~\ref{thm:main-precise}. In this section we carry out the first main step of the proof of Theorem~\ref{thm:main-precise} by showing that the joint law of the  rescaled map $(\Map^n, d^n , \mu^n,\xi^n)$ and the rescaled boundary length process $\acute Z^n$ associated with the space-filling exploration of $(\Map^n, \Be^n ,\omega^n)$ converges to the joint law of the Brownian disk $(H,d,\mu,\xi)$ and the boundary length process $Z'$ associated with a space-filling SLE$_6$ on $H$. The uniform convergence of the associated space-filling curves will be established in the next section. 

To prove the above joint convergence, we will start in Section~\ref{sec-nested-conv} by showing that the joint law of all of the triangulations $\Map_\bk^n$ for $\bk \in \bigcup_{m=0}^\infty \BB N^m$ decorated by percolation interfaces $\lambda_\bk^n$ along with their boundary length processes $Z_\bk^n$, as constructed in Section~\ref{sec-sle6-def-discrete}, converges to the joint law of the analogous continuum processes from Section~\ref{sec-mating}. We will then show in Section~\ref{sec-disk-walk} that one also has the convergence $\acute Z^n \rta Z'$ by ``concatenating" the boundary length processes associated with the nested SLE$_6$ curves. 

Identify $(\Map^n,\Be^n,\omega^n)$ with an element  of $\DPm$ as in Remark~\ref{rmk:flip}. 
For $\bk\in \bigcup_{m=0}^\infty \BB N^m$, let $\Map^n_\bk$, $\Be^n_\bk$, $\omega^n_\bk$, $\wh \Be^n_\bk$, $\ell^n_\bk$, $\lambda^n_\bk$ be defined in the same way as $\Map_\bk$, $\Be_\bk$, $\omega_\bk$, $\wh \Be_\bk$, $\ell_\bk$, $\lambda_\bk$
with $(\Map^n,\Be^n, \omega^n)$ in place of $(\Map,\Be,\omega)$  in the nested exploration of Section~\ref{sec-sle6-def-discrete}. 
Let  $d_\bk^n=\dcon^{-1} n^{-1/4}d_{\Map_\bk^n}$ and $\mu_\bk^n=\mcon^{-1} n^{-1}\mu_{\Map_\bk^n}$. For $s\in [0,\bcon^{-1} n^{-1/2} \el^n]$ let $\xi_{\bk}^n(s) := \beta_{\Map_\bk^n}(\bcon n^{1/2} s)$.  For $t\geq 0$ let $\eta_\bk^n(t ) := \lambda_\bk^n(\tcon n^{ 3/4} t)$.  Let  $\frk M_{\bk}^n= \left( \Map_{\bk}^n , d_{\bk}^n , \mu_{\bk}^n , \xi_{\bk}^n , \eta_\bk^n \right)$.
Let   $Z^n_\bk=(L^n_\bk,R^n_\bk)$ be  the boundary length process of $\lambda_\bk^n$ renormalized as in \eqref{eq:renormalized-Z}, with  $\eta^n_\bk$ in place of  $\eta^n$ there.

Recall the Brownian disk $(H,d,\mu,\xi)$ and the nested exploration $\{\frk H_{\bk}:\bk\in \bigcup_{m=0}^\infty \N^m \}$ from Section~\ref{sec-mating}. Also recall the boundary length processes $\{Z_{\bk}:\bk\in \bigcup_{m=0}^\infty \N^m \}$ and $Z' = (L',R')$ from the same section.

\subsection{Joint convergence of nested explorations}
\label{sec-nested-conv} 

In this subsection we will prove that $ \left(\frk M_{\bk}^n , Z_{\bk}^n \right) $   converges jointly in law to  $\left(\frk H_{\bk}  , Z_{\bk} \right)$ by iteratively applying Theorem~\ref{thm-chordal-conv}. We will use this to derive the joint convergence of the nested exploration and $\acute{Z}^n$ in Section~\ref{sec-disk-walk}.

\begin{prop} \label{prop-component-conv}
	One has the following {convergence in law} as $n\rta\infty$:
	\eqb \label{eqn-component-conv}
	\left\{ \left(\frk M_{\bk}^n , Z_{\bk}^n \right) : \bk \in \bigcup_{m=0}^\infty \BB N^m \right\} \rta 
	\left\{ \left(\frk H_{\bk}  , Z_{\bk} \right) : \bk \in \bigcup_{m=0}^\infty \BB N^m  \right\} ,
	\eqe 
	where for each $\bk$, the first coordinate is given the GHPU topology and the second coordinate is given the Skorokhod topology.
\end{prop}

Proposition~\ref{prop-component-conv} will be a consequence of the convergence results for a single interface proven in~\cite{gwynne-miller-perc} (Theorem~\ref{thm-chordal-conv}) together with the iterative constructions of the pairs $\left(\frk M_{\bk}^n , Z_{\bk}^n \right)$ and $\left(\frk H_{\bk}  , Z_{\bk}  \right)$ from Sections~\ref{sec-sle6-def-discrete} and~\ref{sec-mating}, respectively. Before we give the proof, let us discuss what we need from~\cite{gwynne-miller-perc}. 
Theorem~\ref{thm-chordal-conv} immediately implies that $(\frk M_{\emptyset}^n , Z_{\emptyset}^n) \rta (\frk H_\emptyset, Z_\emptyset)$ in law as $n\rta\infty$. 
In fact, the proof of Theorem~\ref{thm-chordal-conv} from~\cite{gwynne-miller-perc} (in the case of site percolation on a loopless triangulation) yields the following slightly stronger statement.

\begin{lem} \label{lem-Z-law-conv1}
	One has the convergence of laws
	\eqb \label{eqn-Z-law-conv1}
	\left(    \frk M_\emptyset^n , Z_\emptyset^n, \left\{ (\Map_k^n , d_k^n , \mu_k^n , \xi_k^n(0) ) \right\}_{k\in\BB N} \right) \rta \left(  \frk H_\emptyset ,Z_\emptyset , \left\{ (H_k ,d_k , \mu_k  , \xi^n(0)  )\right\}_{k\in\BB N} \right) ,
	\eqe 
	where the first coordinate is given the GHPU topology, the second coordinate is given the Skorokhod topology, and the third coordinate is given the countable product of the pointed Gromov--Hausdorff--Prokhorov topology (i.e., the GHPU topology but restricted to spaces where the curve is constant). 
\end{lem}
Recall that $ \left\{ (H_k ,d_k , \mu_k , \xi_k(0) )\right\}_{k\in\BB N} $ is the set of connected components of $H_\emptyset \setminus \eta_\emptyset$, each equipped with the internal metric of $d_\emptyset$, the restriction of $\mu_\emptyset$, and the last point on its boundary hit by $\eta_\emptyset$.
In particular, these pointed metric measure spaces are a.s.\ determined by $\frk H_\emptyset$. 
The pointed metric measure spaces $(\Map_k^n , d_k^n , \mu_k^n , \xi_k^n(0) )$ are determined by $\Map_\emptyset^n$ and $\eta_\emptyset^n$ in an analogous manner.  

We briefly explain why the proof in~\cite{gwynne-miller-perc} yields Lemma~\ref{lem-Z-law-conv1}.
In our notation, the proof of Theorem~\ref{thm-chordal-conv} in~\cite[Section 7]{gwynne-miller-perc} starts out with a subsequence along which the left side of~\eqref{eqn-Z-law-conv1} converge in law, then shows that the subsequential limit coincides with the right side of~\eqref{eqn-Z-law-conv1}.
In particular, it was shown in~\cite[Section 7.2]{gwynne-miller-perc} that the subsequential limits of the $(\Map_k^n , d_k^n , \mu_k^n , \xi_k^n(0) )$'s coincide with the $(H_k ,d_k , \mu_k  , \xi^n(0)  )$'s (although at that point in the proof in~\cite{gwynne-miller-perc}, the subsequential limit of $(\frk M_\emptyset^n , Z_\emptyset^n)$ had not yet been shown to agree in law with $(\frk H_\emptyset ,Z_\emptyset)$ --- this was not established until the very end of~\cite[Section 7]{gwynne-miller-perc}).

\begin{proof}[Proof of Proposition~\ref{prop-component-conv}]
	If we couple so that $Z_\emptyset^n \rta Z_\emptyset$ a.s.\ in the Skorokhod sense (which we can do by Theorem~\ref{thm-chordal-conv}), then for $k\in\BB N$, the magnitude of the $k$th largest jump of $Z_\emptyset^n$ converges a.s.\ to the magnitude of the $k$th largest jump of $Z_\emptyset$. In other words,
	\eqbn
	\bcon^{-1} n^{-1/2} \# \bdy \Map_{ k}^n \rta \nu(\bdy H_{k}) .
	\eqen
	By the discrete and continuum Markov properties (Lemmas~\ref{lem-perc-decomp-markov} and~\ref{lem-sle6-decomp-markov}) and Theorem~\ref{thm-chordal-conv} (the latter applied once for each $k\in\BB N$), a.s.\ the joint conditional law of $\left\{ (\frk M_k^n , Z_k^n )\right\}_{k\in\BB N}$ given $Z_\emptyset^n$ converge a.s.\ to the joint conditional law of $ \left\{ (\frk H_k  , Z_k  )\right\}_{k\in\BB N} $ given $Z_\emptyset$.  
	This yields the convergence of joint laws 
	\eqb \label{eqn-Z-law-conv2}
	\left( Z_\emptyset^n, \left\{ (\frk M_k^n , Z_k^n )\right\}_{k\in\BB N} \right) \rta \left( Z_\emptyset , \left\{ (\frk H_k  , Z_k  )\right\}_{k\in\BB N} \right) .
	\eqe 
	
	We will now deduce from Lemma~\ref{lem-Z-law-conv1} and~\eqref{eqn-Z-law-conv2} the convergence in law  
	\eqb \label{eqn-Z-law-conv-base}
	\left(  \frk M_\emptyset^n , Z_\emptyset^n,  \left\{ (\frk M_k^n , \frk Z_k^n) \right\}_{k\in\BB N} \right) \rta \left( \frk H_\emptyset, Z_\emptyset , \left\{ (\frk H_k  , Z_k  )\right\}_{k\in\BB N}  \right) ,
	\eqe 
	which is the special case of~\eqref{eqn-component-conv} when we restrict to $m\in \{0,1\}$. 
	Indeed, by Lemma~\ref{lem-Z-law-conv1}, \eqref{eqn-Z-law-conv2}, and the Prokhorov theorem for any sequence of $n$'s tending to $\infty$ there is a coupling of $(Z_\emptyset, \frk H_\emptyset)$ with a collection of pairs $ \left\{ (\rng{\frk H}_k  , \rng Z_k  )\right\}_{k\in\BB N} \eqD  \left\{ ( \frk H_k  , Z_k  )\right\}_{k\in\BB N}$ and a subsequence along which
	\eqb \label{eqn-Z-law-conv-ssl}
	\left( Z_\emptyset^n,   \frk M_\emptyset^n , \left\{ (\frk M_k^n , \frk Z_k^n) \right\}_{k\in\BB N} \right)
	\rta \left( Z_\emptyset , \frk H_\emptyset ,  \left\{ (\rng{\frk H}_k  , \rng Z_k  )\right\}_{k\in\BB N} \right)\qquad \textrm{in law} .
	\eqe 
	As explained after the statement of Lemma~\ref{lem-Z-law-conv1}, {we have that} $ \left\{ (H_k ,d_k , \mu_k , \xi_k(0) )\right\}_{k\in\BB N} $ is a.s.\ determined by $\frk H_\emptyset$. By Lemma~\ref{lem-Z-law-conv1}, if we write $\rng{\frk H}_k = (\rng H_k , \rng d_k , \rng \mu_k ,\rng \xi_k , \rng \eta_k)$ then $(\rng H_k , \rng d_k , \rng \mu_k , \rng \xi(0) ) = (H_k,d_k,\mu_k , \xi(0) )$ as metric measure spaces for each $k\in\BB N$. 
	The metric and area measure on the Brownian disk a.s.\ determine its boundary length measure (see, e.g.,~\cite[Proposition 2]{legall-disk-snake}). 
	This boundary length measure together with the marked boundary point $\rng\xi(0)$ determines $\rng \xi$, so we get that a.s.
	\eqb \label{eqn-disk-agree}
	(\rng H_k , \rng d_k , \rng \mu_k , \rng \xi_k) = (H_k,d_k,\mu_k , \xi_k) ,\quad\forall k\in\BB N  , 
	\eqe 
	where here we mean equality as curve-decorated metric measure spaces. 
	We henceforth identify the curve decorated metric measure spaces appearing in~\eqref{eqn-disk-agree}. 
	
	By~\eqref{eqn-Z-law-conv2}, if we condition on $Z_\emptyset$ then the curve-decorated metric measure spaces $\rng{\frk H}_k$ are conditionally independent SLE$_6$-decorated Brownian disks with boundary lengths specified by the downward jumps of $Z_\emptyset$.
	In particular, under the conditional law given $Z_\emptyset$ and $\{( H_k ,  d_k ,  \mu_k ,  \xi_k)\}_{k\in\BB N}$,
	each of the curves $\rng\eta_k$ is an SLE$_6$ on $ H_k$ going from $x_k = \xi_k(0)$ to the marked boundary point $\wh x_k$ which is determined by $Z_\emptyset$ and $\{(H_k ,   d_k ,  \mu_k ,  \xi_k)\}_{k\in\BB N}$. 
	That is, the conditional laws of $\{\rng{\frk H}_k\}_{k\in\BB N}$ and $\{\frk H_k\}_{k\in\BB N}$ given $Z_\emptyset$ and $\{( H_k ,  d_k ,  \mu_k ,  \xi_k)\}_{k\in\BB N}$ agree.
	The measurability statement in Lemma~\ref{lem-domain-SLE} implies $Z_\emptyset$ and $\{( H_k ,  d_k ,  \mu_k ,  \xi_k)\}_{k\in\BB N}$ a.s.\ determine $\frk H_\emptyset$, so we get that the conditional laws of $\{\rng{\frk H}_k\}_{k\in\BB N}$ and $\{\frk H_k\}_{k\in\BB N}$ given $(Z_\emptyset, \frk H_\emptyset)$ agree.
	Since $\frk H_k$ a.s.\ determines $Z_k$ and $(\frk H_k , Z_k) \eqD (\rng{\frk H}_k ,\rng Z_k)$, we infer that the right sides of~\eqref{eqn-Z-law-conv-base} and~\eqref{eqn-Z-law-conv-ssl} have the same law. Since our initial choice of subsequence was arbitrary, this implies~\eqref{eqn-Z-law-conv-base}. 
	
	Due to the iterative constructions of the $\frk H_\bk$'s and the $\frk M_\bk^n$'s, we may now iterate the above argument to obtain~\eqref{eqn-component-conv}. 
	In particular, we prove by induction on $m$ that
	\eqb \label{eqn-component-conv-m}
	\left\{ \left(\frk M_{\bk}^n , Z_{\bk}^n \right) : \bk \in \bigcup_{r=0}^m\BB N^r \right\} \rta 
	\left\{ \left(\frk H_{\bk}  , Z_{\bk} \right) : \bk \in \bigcup_{r=0}^m \BB N^r  \right\} ,
	\eqe 
	in law. The base case $m=1$ is~\eqref{eqn-Z-law-conv-base}. 
	Now assume that~\eqref{eqn-component-conv-m} holds for some $m\in\BB N$. 
	Then the same reasoning leading to~\eqref{eqn-Z-law-conv2} together with~\eqref{eqn-component-conv-m} applied once inside each of the bubbles of $\frk M_\emptyset^n$ yields
	\eqb \label{eqn-Z-law-conv2-m}
	\left( Z_\emptyset^n, \left\{ (\frk M_\bk^n , Z_\bk^n )\right\}_{\bk \in\bigcup_{r=1}^{m+1} \BB N^r} \right) \rta \left( Z_\emptyset , \left\{ (\frk H_\bk  , Z_\bk  )\right\}_{\bk \in\bigcup_{r=1}^{m+1} \BB N^r} \right) .
	\eqe 
	One can then combine Lemma~\ref{lem-Z-law-conv1} and~\eqref{eqn-Z-law-conv2-m} via exactly the same argument used to prove~\eqref{eqn-Z-law-conv-base} to get~\eqref{eqn-component-conv-m} with $m+1$ in place of $m$. 
	This completes the induction, hence~\eqref{eqn-component-conv} holds.
\end{proof}

\subsection{Joint convergence with   the random walk encoding}
\label{sec-disk-walk}

By standard facts on lattice walks in the first quadrant (see e.g.\ \cite[Theorem~4]{dw-limit}), {we have that} $\acute{Z}^n \rta Z'$  in law with respect to the uniform topology.   
The main result in this section is the following. 

\begin{prop}\label{prop:joint}
	The convergence of $\acute Z^n$ to $Z'$ holds jointly with the one in  Proposition~\ref{prop-component-conv}.
\end{prop}
By the Skorokhod representation theorem, for each subsequence of $\BB N$ there exists a further subsequence $\cN$ and a coupling of $\left\{ \left(\frk H_{\bk}  , Z_{\bk} \right) : \bk \in \bigcup_{m=0}^\infty \BB N^m  \right\}$ 
and $\{\acute{\frk M}^n\}_{n\in\mcl N}$  such that almost surely
\begin{enumerate}
	\item $\lim_{\mcl N \ni n \to \infty}(\frk M_{\BB k}^n,Z_{\BB k}^n) =(\frk H_{\BB k},Z_{\BB k})$ for each $\BB k\in \bigcup_{m=0}^\infty \BB N^m$    
	with respect to the topology in Proposition~\ref{prop-component-conv}, 
	\item $\acute Z^n$ converges to a stochastic process  $\rng Z'$ in $C_0(\R,\R^2)$ with the same law as $Z'$.
\end{enumerate}
Throughout this section we work under such a coupling and assume $n\in \cN$. Then Proposition~\ref{prop:joint} is an immediate consequence of the following result concerning the coupling.
\begin{prop}\label{prop:joint-metric-mating}
	$Z'=\rng Z'$ almost surely.
\end{prop}
{We devote the rest of this subsection to the proof of Proposition~\ref{prop:joint-metric-mating}.}

For each multi-index $\bk\in \cC_\emptyset$, recall the times $s_\bk,t_\bk,\wh t_\bk$ in  Section~\ref{sec-mating}  for the  bubble $H_\bk$, which are defined so that $\eta'([s_\bk , t_\bk]) = H_\bk$ and $\eta'(\wh t_\bk) = \wh x_\bk$. If $\Map^n_\bk\neq \emptyset$, we will also need the discrete analog of these times in the setting of Section~\ref{sec-sle6-def-discrete}. 
Let $S_{\bk}^n,T^n_{\bk}\in \N_0$ be such that  $\acute{\lambda}^n([S_{\bk}^n,T^n_{\bk}]_\Z)=\cE(\Map_\bk^n )$ and let $\wh T_\bk^n \in [S_\bk^n , T_\bk^n]_{\BB Z}$ be such that  $\acute\lambda^n(T_\bk^n) =\wh{\BB e}_{\bk}^n$.
The rescaled version of these times are given by 
\eqb \label{eqn-rescaled-abc-n}
s_\bk^n := \mcon^{-1} n^{-1} S_\bk^n ,\quad 
\wh t_\bk^n := \mcon^{-1} n^{-1} \wh T_\bk^n ,\quad\op{and}\quad
t_\bk^n := \mcon^{-1} n^{-1} T_\bk^n  .
\eqe 

Recall the description of $S^n_\bk,\wh T^n_\bk,T^n_\bk$ in terms of $\acute \cZ^n$ at the end of Section~\ref{sec-sle6-def-discrete}, namely, Facts~\ref{item:hat}-\ref{item:future} above Lemma~\ref{lem-perc-decomp-markov}.
We now give the continuum analog of these facts and the convergence  of these discrete quantities to their continuum counterparts.
Fix $s,t \in [0,\mu(H)]$ such that  $s<t$. Given a sample of $Z'$, we say that  $t$ is an \emph{ancestor} of $s$ if $L'_u>L'_t$ and $R'_u>R'_t$ for all $s<u\le t$.  We call $s$ an \emph{ancestor-free time relative to $t$} if there are no ancestors of $s$ in $(s,t)$.  
Let $\ans(t)$ be the set of ancestor-free times relative to $t$, which is  known to be  a closed set with Hausdorff dimension $3/4$~\cite[Section 10.2]{wedges} (see also~\cite[Example 2.3]{ghm-kpz}). Let $\Cut(s) =\{ t'\in (s, \mu(H)):  \textrm{$t'$ is an ancestor of $s$}\}$.
The set $\Cut(s)$ is the same as the set of so-called \emph{$\pi/2$-cone times}\footnote{A time $t$ is a $\pi/2$-cone time for $Z'=(L',R')$ is there exists a $u<t$ such that $L'_{u'}>L'_{t}$ and $R'_{u'}>R'_{t}$ for all $u'\in(u,t)$. Note that this definition of a cone time corresponds to a cone time for the \emph{time reversal} of $Z'$ in some other literature.} for $Z'$ whose corresponding $\pi/2$-cone intervals contain $s$. Therefore $\Cut(s)$ is a closed set of Hausdorff dimension $1/4$ by the main result of~\cite{evans-cone} (see also the proof of~\cite[Lemma 8.5]{wedges}).

\begin{prop}\label{prop:mating}
	For each multi-index $\bk\in \bigcup_{m\in\N_0} \N^m$, the following hold.
	\begin{enumerate}[label=(\arabic{enumi})]
		\item 
		If $H_\bk$ is monochromatic red, then $\wh t_\bk=\inf\{t\ge s_\bk: L'_{s_\bk} - L'_t =\ell_\bk/2 \}$. \label{item-mating1}
		If $H_\bk$ is monochromatic blue, then $\wh t_\bk=\inf\{t\ge s_\bk: R'_{s_\bk} - R'_t=\ell_\bk/2 \}$. 
		If $H_\bk$ is dichromatic, then $\wh t_\bk=t_\bk$.
		\item 	If $H_\bk$ is monochromatic, then the set  of open intervals  $\{ (s_{\bk'}, t_{\bk'}) : \bk' \in\cC_\bk \label{item-mating2}
		\; \textrm{and $H_{\bk'}$ is dichromatic}\}$ equals the set of connected components of 
		$(\wh t_\bk, t_\bk) \setminus \Cut(\wh t_\bk)$.
		\item	The set  of open intervals  $\{ (s_{\bk'}, t_{\bk'}) : \bk' \in\cC_\bk\; \textrm{and $H_{\bk'}$ is monochromatic}\}$ \label{item-mating3}
		equals the set of connected components of  $(s_\bk,\wh t_\bk) \setminus \ans(\wh t_\bk)$.
		\item $(\acute{Z}^n, s^n_\bk, t^n_\bk, \wh t^n_\bk)$ converges in law to $(Z', s_\bk,t_\bk,\wh t_\bk)$.\label{item-mating4}
	\end{enumerate} 
\end{prop} 
\begin{proof}
	We will first argue that \ref{item-mating1}-\ref{item-mating3} hold for $\bk=\emptyset$ and $\bk\in\cC_{\emptyset}$, then we argue that \ref{item-mating4} holds for $\bk=\emptyset$ and $\bk\in\cC_{\emptyset}$, and then we argue that \ref{item-mating1}-\ref{item-mating4} hold for $\bk\in\N^2$ with $H_{\bk^-}$ dichromatic. This will allow us to conclude by iteration.
	
	The statements concerning $\wh t_\emptyset$ and $\{s_\bk,t_\bk\}_{\bk\in  \cC_{\emptyset} }$ in Assertions \ref{item-mating1}-\ref{item-mating3} are consequences of Theorem~\ref{thm:mating}, which are explained in detail in \cite[Section~6.9]{bhs-site-perc}, based on Sections~6.1-6.8 there.  
	
	Assertion~\ref{item-mating4} for $\wh t^n_\emptyset$  and $\{s^n_\bk,t^n_\bk\}_{\bk\in \cC_\emptyset}$ follows from  \cite[Lemma~9.25]{bhs-site-perc}:
	indeed, the convergence of $\wh t^n_\emptyset$   follows from \cite[Lemma~9.25, (i)]{bhs-site-perc} and Fact~\ref{item:hat}  above Lemma~\ref{lem-perc-decomp-markov}.
	The convergence of  $\{s^n_\bk,t^n_\bk: \bk\in \cC_\emptyset, H_\bk\textrm{ is monochromatic} \}$ follows from \cite[Lemma~9.25, (i)]{bhs-site-perc} and Fact~\ref{item:past} above Lemma~\ref{lem-perc-decomp-markov}. 
	The convergence of  $\{s^n_\bk,t^n_\bk: \bk\in \cC_\emptyset, H_\bk \textrm{ is dichromatic} \}$ follows from \cite[Lemma~9.25, (ii)]{bhs-site-perc} and Fact~\ref{item:future} above Lemma~\ref{lem-perc-decomp-markov}.
	Here Facts~\ref{item:hat}-\ref{item:future} give the  $\acute Z^n$-description of  special points corresponding to the percolation interface $\lambda^n_\emptyset$. Section~6.9 in \cite{bhs-site-perc} describes the same quantities in the continuum in terms of $Z'$ and \cite[Lemma~9.25]{bhs-site-perc} show that the random walk quantities converge to their Brownian motion counterparts. 
	
	For $\bk\in\cC_\emptyset$  such that $H_\bk$ is monochromatic, since $(\acute{Z}^n, s^n_\bk, t^n_\bk)$ converges in law to $(Z', s_\bk,t_\bk)$, we may repeat the argument for $\wh t^n_\emptyset$ to conclude  that $(\acute{Z}^n, \wh t^n_\bk)$ converges in law to $(Z',\wh t_\bk)$.
	
	Now for $\bk\in\N^2$ with $H_{\bk^{-}}$ being dichromatic, since $\wh t^n_{\bk^-}=t^n_{\bk^-}$ converges to $\wh t_{\bk^-}=t_{\bk^-}$,
	Assertions \ref{item-mating1}-\ref{item-mating3} for $s_\bk,t_\bk$ follows from \cite[Section~6.9]{bhs-site-perc}.
	Moreover, Assertion~\ref{item-mating4} for $s^n_\bk,t^n_\bk$ follows from \cite[Lemma~9.25 (i)]{bhs-site-perc} and Fact~\ref{item:past} above Lemma~\ref{lem-perc-decomp-markov}.

	Since we have proved Proposition~\ref{prop:mating} for multi-indices in 
	\begin{equation}\label{eq:intial}
	\initial \defeq \{\emptyset\} \cup \N  \cup\{\bk\in \N^2: H_{\bk^-}\textrm{ is dichromatic}  \}
	\end{equation}
	the remaining cases follow from iteration.	
\end{proof}
\begin{remark}\label{rmk:initial}
	In the proof of Proposition~\ref{prop:mating} and other arguments in this section where the iteration is used, it is important to consider  the cases indexed by $\initial$ instead of just $\cC_\emptyset$. Otherwise it would only  cover the  cases where the parent bubble is monochromatic. As seen in the proof of Proposition~\ref{prop:mating}, when the parent bubble is dichromatic, it requires a separate argument.
\end{remark}
\begin{lem}\label{lem:endpoint1}
	For all $\bk\in \bigcup_{m\in\N_0}\N_m$, we have 
	\((\rng Z', s_\bk, t_\bk,\wh t_\bk )\overset{d}{=}(Z',s_\bk, t_\bk,\wh t_\bk)\). Moreover, $(s^n_\bk,t^n_\bk,\wh t^n_\bk) \rta (s_\bk, t_\bk,\wh t_\bk)$ in probability   as $n$ tends to $\infty$.
\end{lem}
\begin{proof}
	For each multi-index $\bk$, let $\rng s_\bk,\rng t_\bk$ be such that  $(\rng Z', \rng s_\bk, \rng t_\bk)\overset{d}{=}(Z',s_\bk,t_\bk)$.  
	By Assertion~\ref{item-mating4} of Proposition~\ref{prop:mating}, we see that 
	$(\acute{Z}^n, s^n_\bk, t^n_\bk)$ converges in probability to $(\rng Z', \rng s_\bk, \rng t_\bk)$.\footnote{Here and in several places below, we use the general fact that if $f$ is a measurable function and $\{X_n\}_{n\in\BB N}$ and $X$ are random variables such that $(X_n, f(X_n))\rta(X,f(X))$ in law and $X_n\rta X$ a.s, then $f(X_n)\rta f(X)$ in probability.\label{fn:in-prob}}
	To show \((\rng Z', s_\bk, t_\bk)\overset{d}{=}(Z',s_\bk, t_\bk)\) and $(s^n_\bk, t^n_\bk) \rta (s_\bk,t_\bk)$ in probability, it suffices to show that $\rng s_\bk=s_\bk$ and $\rng t_\bk=t_\bk$ a.s. 
	
	We first consider the case when $\bk \in \initial$.
	Since $\eta'$ is parametrized by $\mu$-mass and $\mu(\eta_\bk) = 0$ for all multi-indices $\bk$, we have  
	\begin{equation}\label{eq:sk}
	s_\bk-{s_{\bk^-}} =\sum_{\bk' \preceq_{\bk^-} \bk}\mu(H_{\bk'}) \qquad \textrm{for each multi-index }\bk,
	\end{equation}
	where here the sum ranges over all $\bk'\in \cC_{\bk^-}$ such that $\bk' \prec_{\bk^-} \bk $ (recall $\prec_{\bk^-}$ from Section~\ref{sec-mating}). 
	
	Suppose $\bk\in\cC_\emptyset$ and $H_\bk$ is monochromatic. For any $\bk'\prec_\emptyset \bk$,  by Proposition~\ref{prop-component-conv}, for large enough $n$, the bubble $\Map^n_{\bk'}$ is contained in $\acute{\eta}^n([0,s^n_\bk])$. By letting $n\rta \infty$, we have $\sum_{\bk' \preceq_\emptyset \bk}\mu(H_{\bk'})\le \rng s_\bk$  a.s.  Since $s_\bk\overset{d}{=} \rng s_\bk$, we must have $s_\bk=\rng s_\bk$ a.s.
	
	Suppose $\bk\in \cC_\emptyset$ and $H_\bk$ is dichromatic. Then $\mu(H_\emptyset)- t_\bk=\sum_{\bk \preceq_\emptyset \bk'} \mu(H_{\bk'})$, where the sum ranges over all $\bk'\in \cC_\emptyset$ such that $\bk\prec_\emptyset \bk'$.
	Similarly as in the monochromatic case we have $\mu(H_\emptyset)- t_\bk\le \mu(H_\emptyset)- \rng t_\bk$ a.s.\ and hence $\rng t_\bk=t_\bk$ a.s. 
	
	For each multi-index $\bk$, we have $ t_\bk -s_\bk=\mu(H_\bk)$ and $t^n_\bk-s^n_\bk=\mu^n(\Map^n)$.
	By letting $n\rta\infty$ and using Proposition~\ref{prop-component-conv} and the two preceding paragraphs, we get that $\rng s_\bk=s_\bk$ and $\rng t_\bk=t_\bk$ a.s.\ for all $\bk \in \cC_\emptyset$.
	
	We have $\mu(H)-\wh t_\emptyset=\sum_{\bk}\mu(H_{\bk})$ where $\bk$ ranges over all $\bk\in \cC_\emptyset$ such that $H_\bk$ is dichromatic.
	A similar argument as above then shows that \((\rng Z', \wh t_\emptyset )\overset{d}{=}(Z',\wh t_\emptyset)\) and $\wh t^n_\emptyset\rta \wh t_\emptyset$ in probability. 
	For $\bk\in\cC_\emptyset$  such that $H_\bk$ is monochromatic, since \((\rng Z', s_\bk, t_\bk)\overset{d}{=}(Z',s_\bk, t_\bk)\) and $(s^n_\bk,t^n_\bk) \rta (s_\bk, t_\bk)$ in probability, we may repeat the argument for $\wh t_\emptyset$ to conclude  that \((\rng Z', \wh t_\bk )\overset{d}{=}(Z',\wh t_\bk)\) and $\wh t^n_\bk\rta \wh t_\bk$ in probability. 
	
	Now suppose $\bk\in\N^2$ and $H_{\bk^-}$ is dichromatic. Since $s^n_{\bk^-}$  converges to $s_{\bk^-}$  in probability, we can use \eqref{eq:sk} and 
	the same argument as in the case where $\bk\in\cC_\emptyset$ and $H_\bk$ is monochromatic to conclude.
	This proves Lemma~\ref{lem:endpoint1} for $\bk\in\initial$.   
		
	Lemma~\ref{lem:endpoint1} for general multi-indices  now follows by iterating the argument above. 
\end{proof}

For each $\bk\in \bigcup_{m=0}^\infty \N^m$, let $\ell_{\bk}^{\mathrm l}$ and $\ell_{\bk}^{\mathrm r}$ be the boundary lengths of the clockwise and counterclockwise arc on $\bdy H_\bk$ from $x_\bk$ to $\wh x_\bk$, respectively. Equivalently, $(\ell^{\mathrm l}_\bk,\ell^{\mathrm{r}}_\bk )= Z_{\bk}(0)$. In particular, with probability 1,   $(\ell^{\mathrm l}_\bk,\ell^{\mathrm{r}}_\bk )=(\ell_\bk/2,\ell_\bk/2)$ if and only if $H_\bk$ is monochromatic. 

\begin{lem}\label{lem:endpoint2}
	\((\rng Z', \ell^{\mathrm l}_\bk,\ell^{\mathrm{r}}_\bk)\overset{d}{=}(Z', \ell^{\mathrm l}_\bk,\ell^{\mathrm{r}}_\bk)\) and \(\rng Z'(t_\bk) -\rng Z'(s_\bk)=Z'(t_\bk)-Z'(s_\bk)\)
	for each  $\bk \in \cC_\emptyset$. 
\end{lem} 
\begin{proof}
	Let $\ell^{\op l,n}_\bk$ and $\ell^{\op r, n}_\bk$  be such that 
	$(\Map^n_\bk,\Be^n_\bk ,\omega^n_\bk) \in \DP(\bcon n^{1/2}  \ell^{\op l,n}_\bk, \bcon n^{1/2}  \ell^{\op r, n}_\bk)$. 
	By  Proposition~\ref{prop:mating} and Lemma~\ref{lem:endpoint1}, {we have that} $(\acute Z^n , \ell^{\mathrm l,n}_\bk,\ell^{\mathrm{r},n}_\bk) $ converges to $(Z', \ell^{\mathrm l}_\bk,\ell^{\mathrm{r}}_\bk)  $ in law. On the other hand, in our coupling $(\acute Z^n , \ell^{\mathrm l,n}_\bk,\ell^{\mathrm{r},n}_\bk)  $ converge to $(\rng Z', \ell^{\mathrm l}_\bk,\ell^{\mathrm{r}}_\bk)$ in probability.
	This gives \((\rng Z', \ell^{\mathrm l}_\bk,\ell^{\mathrm{r}}_\bk)\overset{d}{=}(Z', \ell^{\mathrm l}_\bk,\ell^{\mathrm{r}}_\bk)\).
	
	If $H_\bk$ is monochromatic, then $Z'(t_\bk)-Z'(s_\bk)=(-\ell_\bk,0)$ or $(0,-\ell_\bk)$ depending on whether the the color is red or blue. If $H_\bk$ is dichromatic, then 
	\(Z'(t_\bk)-Z'(s_\bk)=-(\ell_\bk^{\op r}, \ell_\bk^{\op l})\). In both cases since \((\rng Z', \ell^{\mathrm l}_\bk,\ell^{\mathrm{r}}_\bk)\overset{d}{=}(Z', \ell^{\mathrm l}_\bk,\ell^{\mathrm{r}}_\bk)\),
	we have \(\rng Z'(t_\bk) -\rng Z'(s_\bk)=Z'(t_\bk)-Z'(s_\bk)\)  almost surely.
\end{proof}

The next  two lemmas are the main steps towards the proof of Proposition~\ref{prop:joint-metric-mating}.
\begin{lem}\label{lem:Levy-joint}
	$(Z',Z_\emptyset) \overset{d}{=}(\rng Z',Z_\emptyset)$.
\end{lem}
\begin{proof}	
	Recall the pair of independent $\frac32$-stable L\'evy processes $Z^\infty=(L^\infty,R^\infty)$ in Lemma~\ref{lem:Levy}.   It is well known that for each $t>0$  the process $L^\infty|_{[0,t]}$ is almost surely  determined by the ordered downward jumps, i.e., the lengths of these jumps as well as the order in which the jumps appear.  By Lemma~\ref{lem:Levy}, the same almost sure statement holds for $Z_\emptyset$. Let  $\rng Z_\emptyset$ be such that $(Z',Z_\emptyset)\overset{d}{=}(\rng Z',\rng Z_\emptyset)$.
	It is clear from Lemma~\ref{lem:endpoint2} that $Z_\emptyset$ and $\rng Z_\emptyset$ have the same set of jumps.  (Note that a.s.\ there are no two jumps for $Z_\emptyset$ of the same size, and similarly for $\rng Z_\emptyset$). Therefore we are left to show that the order in which the jumps appear also agree. 
	
	Suppose $\bk,\bk' \in\cC_\emptyset$.
	We need to show that the jumps of $Z_\emptyset$ and $\rng Z_\emptyset$ of size $\ell_{\BB k}$ and $\ell_{\BB k'}$ corresponding to the bubbles $H_{\BB k}$ and $H_{\BB k'}$ occur in the same order.  
	If $H_\bk$ and $H_{\bk'}$ are both monochromatic (resp.\ both  dichromatic), then since the order in which $\eta'$ fills in the monochromatic (resp.\ dichromatic) bubbles cut out by $\eta_\emptyset$ is the same (resp.\ the opposite) as the order in which $\eta_\emptyset$ cuts these bubbles off, the jump of length $\ell_{\bk}$ comes before the jump of length $\ell_{\bk'}$ if and only if $s_\bk<s_{\bk'} $ (resp.\ $s_\bk > s_{\bk'}$). By Lemma~\ref{lem:endpoint1}, the same holds for $\rng Z_\emptyset$.
	
	Now suppose $H_{\bk}$ is dichromatic and $H_{\bk'}$ is monochromatic. 
	Let 
	\[
	\sigma_{\bk} =\sup\{t\le s_{\bk}:  R'_t\le R'_{s_\bk} \}.
	\]
	Then $\sigma_{\bk}$ is a.s.\ the first time such that $\eta' $ visits $x_{\bk}$; equivalently, $\sigma_{\bk}$ is the time that the bubble $H_{\bk}$ is enclosed, i.e., separated from $\hat x_\emptyset$ by $\eta'$ (recall that $H_\emptyset$ is monochromatic red, so $H_\bk$ is a bubble cut out by $\eta_\emptyset$ which intersects the right boundary of $H_\bk$). See the left part of Figure \ref{fig-Z-conv}. 
	Recall the definition of the ordering $\prec_\emptyset$ on bubbles cut out by $\eta_\emptyset$ from Section~\ref{sec-mating}. 
	The jump of length $\ell_{\bk}$ for $Z_\emptyset$ comes before the jump of length $\ell_{\bk'}$ if and only if  $H_{\bk'} \subset \eta' ([\sigma_{\bk},s_{\bk}])$, which is  further equivalent to $\sigma_{\bk} < s_{\bk'}$. 
	Let $\rng\sigma_{\bk}$ be such that $(\rng Z',\rng \sigma_{\bk})\overset{d}{=} (Z',\sigma_{\bk})$.   
	
	In the discrete setting, suppose $\Map^n_\bk$ (resp.\ $\Map^n_{\bk'}$) is visited by $\acute{\lambda}^n$ after (resp.\ before) $\wh e^n_\emptyset$. 
	Equivalently, $\type_{\bk}=\op{di}$ and $\type_{\bk'}=\op{mono}$.
	By Lemma~\ref{lem:future2}, {we have that} $\Delta\acute{\cZ}^n_{S^n_\bk}$ is a $c$-step, corresponding to the last edge on $\Map^n_\bk$ visited by $\acute{\lambda}^n$.
	Let $\sigma^n_{\bk'}$ be  such that $\Delta\acute{\cZ}^n_{\mcon n\sigma^n_{\bk'}}$ is  the matching $b$-step of $\Delta\acute{\cZ}^n_{S^n_\bk}$.
	Then by Lemma~\ref{lem:future2}, {we have that} $\Map^n_{\bk}$ comes before $\Map^n_{\bk'}$ in the peeling process of  ${\lambda}^n_\emptyset$
	if and only if $\sigma^n_{\bk'}< s_n^k$.
	Then as in the proof of Lemma~\ref{lem:endpoint1},  {we have that} $\lim_{n\to \infty} \sigma^n_{\bk'}=\rng \sigma_{\bk'}$ in probability.  Therefore $H_\bk$ is disconnected by $\eta_\emptyset$ from $\wh x_\emptyset$ before  $H_{\bk'}$
	(equivalently, the jump $\ell_{\bk}$ comes before the jump $\ell_{\bk'}$)
	if and only if $\rng\sigma_{\bk}< s_{\bk'}$. Combined with the previous paragraph, we see that  $\rng \sigma_{\bk}<s_{\bk'}$ if and only if $\sigma_{\bk}<s_{\bk'}$.
	Since there exists a sequence $\{\bk_m\} $ in $\cC_\emptyset$ such that $s_{\bk_m}\downarrow \sigma_{\bk'}$, we have $\rng\sigma_{\bk'}=\sigma_{\bk'}$ almost surely. 
	Therefore if $H_\bk$ is dichromatic and $H_{\bk'}$ is monochromatic, the jump of length $\ell_{\bk}$ comes before the jump of length $\ell_{\bk'}$ for $Z_\emptyset$ if and only if the same holds for $\rng Z_\emptyset$, since this event occurs exactly when $\sigma_{\bk} < s_{\bk'}$ (equivalently, when $\rng\sigma_{\bk} < s_{\bk'}$). 
	This gives $(Z',Z_\emptyset)\overset{d}{=}(\rng Z',Z_\emptyset)$.
\end{proof}

\begin{figure}
	\begin{center}
		\includegraphics[scale=1]{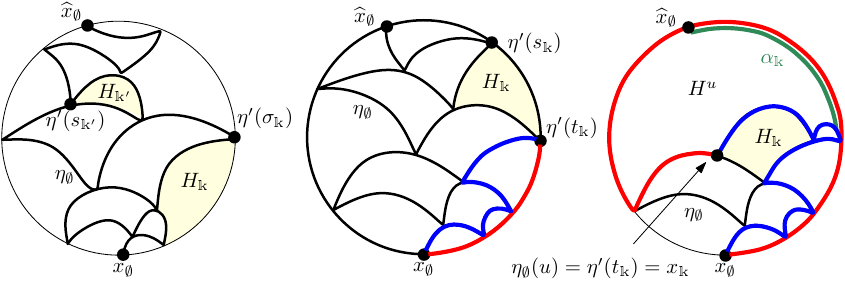}
		\caption{\textbf{Left:} Illustration of the proof of Lemma \ref{lem:Levy-joint} in the case where $H_{\bk}$ is dichromatic and $H_{\bk'}$ is monochromatic. 	
			\textbf{Middle:} 	 Illustration of the proof of Lemma~\ref{lem:Z'} in the case where $H_{\bk}$ is dichromatic.  \textbf{Right:} Illustration of the proof of Lemma~\ref{lem:Z'} in the case where $H_{\bk}$ is monochromatic. In both of the middle and right figures, the red (resp.\ blue) curve indicates the left (resp.\ right) boundary at time $t_{\bk}$. The green curve on the right figure has length $\alpha_\bk$. 
			}
		\label{fig-Z-conv}
	\end{center}
\end{figure}
\begin{lem}\label{lem:Z'}
	\(\rng Z'(s_\bk)=Z'(s_\bk) \) and \(\rng Z'(t_\bk)=Z'(t_\bk) \)  almost surely for all $\bk\in \cC_\emptyset$.
\end{lem}
\begin{proof}
	Suppose $H_\bk$ is dichromatic. 
	Since  the right boundary of the unexplored region at time $t_\bk$ is the union of the right boundaries of the dichromatic bubbles which come after $H_\bk$,
	we have $R'(t_\bk)=\sum_{\bk  \preceq_\emptyset \bk'}\ell^{\mathrm r}_{\bk'}$ where here the sum ranges over all $\bk'\in \cC_\emptyset$ 
	such that $\bk\prec_\emptyset \bk'$. Note that the ordering $\preceq_\emptyset$ on $\cC_\emptyset$ is determined by $\{s_{\bk'},t_{\bk'},\wh t_{\bk'}\}_{\bk'\in \cC_\emptyset}$. Moreover, both $\{s_{\bk'},t_{\bk'},\wh t_{\bk'}\}_{\bk'\in \cC_\emptyset}$  and $\{\ell^{\mathrm r}_{\bk'}, \ell^{\mathrm l}_{\bk'} \}_{\bk'\in\cC_\emptyset}$ are determined by $Z'$.
	By Lemmas~\ref{lem:endpoint1} and~\ref{lem:endpoint2},   $\{s_{\bk'},t_{\bk'},\wh t_{\bk'}\}_{\bk'\in \cC_\emptyset}$  and $\{\ell^{\mathrm r}_{\bk'}, \ell^{\mathrm l}_{\bk'} \}_{\bk'\in\cC_\emptyset}$ are also determined by $\rng Z'$ via the same measurable functions. Therefore
	\begin{equation}\label{eq:R'}
	\rng R'(t_\bk)=\sum_{\bk  \preceq_\emptyset \bk'}\ell^{\mathrm r}_{\bk'}=R'(t_\bk).
	\end{equation}
	See the middle part of Figure \ref{fig-Z-conv}.
	The same statement holds for $\rng L'(t_\bk)$ and $L'(t_\bk)$ with  $\ell^{\mathrm r}_{\bk'}$ replaced by $\ell^{\mathrm l}_{\bk'}$. Therefore $\rng Z'(t_\bk)=Z'(t_\bk)$.
	By the second equality in  Lemma~\ref{lem:endpoint2}, we conclude the proof  in the dichromatic case.

	Suppose $H_\bk$ is monochromatic.  
	Let $u$ be the last time such that $\eta_\emptyset(u)=x_\bk$ 
	and let $H^u$ be the connected component of $H\setminus \eta_\emptyset([0,u])$ containing $\wh x_\emptyset$ on its boundary. See the right part of Figure \ref{fig-Z-conv}.
	Then both $\rng L'(t_\bk)$ and $L'(t_\bk)$ are equal to $L_\emptyset(u)+\ell^{\mathrm r}_\emptyset$. Therefore $\rng L'(t_\bk)=L'(t_\bk)$.
	Let $\alpha_\bk$ be the boundary length of the intersection of  $\bdy H^u$ and the right  arc on $(H_\emptyset, x_\emptyset,\wh x_\emptyset)$. 
	Then $\ell^{\mathrm r}_\emptyset-\alpha_\bk=\sum_{\bk'} \ell^{\op l}_{\bk'}$ where $\bk'$ ranges over $\bk'\in \cC_\emptyset$ such that $H_{\bk'}$ is dichromatic and the jump of length $\ell_{\bk'}$ comes before $\ell_\bk$ in $Z_\emptyset$. By Lemmas~\ref{lem:endpoint2} and~\ref{lem:Levy-joint}, we have    $(Z',\alpha_\bk)\overset{d}{=}(\rng Z',\alpha_\bk)$.
	Note that $R'(t_\bk)=R_\emptyset(u)-\alpha_\bk+ \sum_{\bk'} \ell^{\op r}_{\bk'}$. 
	By a similar argument as for \eqref{eq:R'}, Lemmas~\ref{lem:endpoint2} and~\ref{lem:Levy-joint} yield
	$\rng R'(t_\bk)=R_\emptyset(u)-\alpha_\bk+ \sum_{\bk'} \ell^{\op r}_{\bk'}$ where $\bk'$ ranges over $\bk'\in \cC_\emptyset$ such that $H_{\bk'}$ is dichromatic and the jump of length $\ell_{\bk'}$ comes before the jump of length $\ell_\bk$ for $Z_\emptyset$.  Therefore $\rng R'(t_\bk)=R'(t_\bk)$ hence $\rng Z'(t_\bk)=Z'(t_\bk)$.
	Now  the second equality in  Lemma~\ref{lem:endpoint2} concludes the proof.
\end{proof}

\begin{proof}[Proof of Proposition~\ref{prop:joint-metric-mating}]
	When $\bk\in \cC_\emptyset$ and $H_\bk$ is monochromatic, having established Lemma~\ref{lem:Z'}, the argument in Lemma~\ref{lem:Levy-joint} gives that  $(Z',Z_\bk)\overset{d}{=}(\rng Z',Z_\bk)$.
	When $\bk\in \cC_\emptyset$ and $H_\bk$ is dichromatic, we have the same statement with an even simpler argument. In this case the order of jumps for $Z_\bk$ is given by the 
	ordering $\prec_\bk$, which is determined by $\{s_{\bk'}\}_{\bk'\in\cC_\bk}$. Therefore by Lemma~\ref{lem:endpoint1}, {we have that} $(Z',Z_\bk)\overset{d}{=}(\rng Z',Z_\bk)$.
	Now the argument in Lemma~\ref{lem:Z'} gives that Lemma~\ref{lem:Z'}  still holds with $\cC_\emptyset$ replaced by $\cC_\bk$.  
	
	By iteration, both Lemmas~\ref{lem:Levy-joint} and~\ref{lem:Z'} hold  for all  multi-indices.
	Since $\{s_\bk,t_\bk: \bk \in \bigcup_{m=0}^\infty \N^m \}$ is dense in $[0,\mu(H_\emptyset)]$, we must have $Z'=\rng Z'$ a.s.
\end{proof}

\section{Scaling limit of the space-filling exploration}
\label{sec-sle6-conv}

Throughout this section we work in the setting of Theorem~\ref{thm:main-precise} and Proposition~\ref{prop-component-conv}.  By Proposition~\ref{prop:joint} and the Skorokhod representation theorem,  we can work under a  coupling of $\{\acute{\frk M}^n \}_{n\in \N}$ and  $\left\{ \left(\frk H_{\bk}  , Z_{\bk} \right) : \bk \in \bigcup_{m=0}^\infty \BB N^m  \right\}$  where  almost surely the convergence in Proposition~\ref{prop-component-conv} holds and $\acute{Z}^n\rta Z'$. Henceforth fix such a coupling.

By Proposition~\ref{prop-ghpu-embed}, for each $\bk\in \bigcup_{m=0}^\infty \BB N^m$  there a.s.\ exists a compact metric space $(W_\bk ,D_\bk )$ and isometric embeddings
\begin{equation}
\label{eq:embedding}
\iota_\bk : H_\bk \rta W_\bk \quad \op{and}\quad
\iota_\bk^n : \Map_\bk^n \rta W_\bk ,\: \forall n \in\N
\end{equation} 
such that $\iota_\bk^n (\frk M_\bk^n) \rta \iota_\bk(H_\bk)$ a.s.\ in the $D_\bk$-HPU sense (Definition~\ref{def-hpu}). 
Write $\iota=\iota_\emptyset$, $\iota^n = \iota_\emptyset^n$, and $(W,D) := (W_\emptyset , D_\emptyset)$. 
We henceforth identify $\frk H'$ and $\acute{\frk M}^n$ with their images under $\iota $ and $\iota^n$, respectively. Note that this has the effect of identifying $\frk H_\bk$ and $\frk M_\bk^n$ for $n\in\BB N$ with their images under $\iota$ and $\iota^n$, respectively, so that every $H_\bk$ and $\Map_\bk^n$ is a subset of $W$. 
The goal of this section is to prove the following proposition, which implies Theorem~\ref{thm:main-precise}. 

\begin{prop} \label{prop-peano-conv}
	In the above setting,  a.s.\ $\lim_{n\to\infty}\acute\eta^n =\eta'$ in the $D$-uniform topology. 
\end{prop}
To prove Proposition~\ref{prop-peano-conv}, we will define for $M,K \in \BB N$ a curve $\eta_{M,K}'$ by concatenating, in a certain manner, the (finite) collection of SLE$_6$ curves $\eta_\bk$ corresponding to $\bk \in \bigcup_{m=0}^M [1,K]_{\BB Z}^m$, parametrized by disconnected area (see Section~\ref{sec-sle6-area} for this parametrization). We also define analogous discrete curves $\acute\eta_{M,K}^n$.
We will show that if $M$ and $K$ are chosen sufficiently large, independently of $n$, then all of the connected components of $H\setminus \eta_{M,K}'$ and $\Map^n\setminus \acute\eta_{M,K}^n$ for $n\in\BB N$ are uniformly close. This will imply that $\eta_{M,K}'$ (resp.\ $\acute\eta_{M,K}^n$) is uniformly close to $\eta'$ (resp.\ $\acute\eta^n$). We can also deduce from the results of Section~\ref{sec-sle6-area} that $\acute\eta_{M,K}^n $ is close to $\eta_{M,K}'$ when $n$ is large, which will give the desired GHPU convergence. 
Before proceeding to this argument we first record some basic facts about the above-defined embeddings in Section~\ref{sec-embedding}.

\subsection{Basic properties  of the embeddings}\label{sec-embedding}
Recall \eqref{eq:embedding}. For each multi-index $\bk\not=\emptyset$, let 
\eqb \label{eqn-embedded-space-cont}
\wt{\frk H}_\bk = \left( \wt H_\bk , \wt d_\bk , \wt\mu_\bk  , \wt \xi_{\bk } ,  \wt\eta_\bk \right) := \iota_\bk(\frk H_\bk) 
\quad\textrm{and}\quad
\wt{\frk M}_\bk^n = \left( \wt \Map_\bk^n , \wt d_\bk^n , \wt\mu_\bk^n  , \wt \xi_{\bk }^n ,  \wt\eta_\bk^n \right) := \iota_\bk^n(\frk M_\bk^n)  
\eqe 
so that a.s.\ $\lim_{n\to\infty}\wt{\frk M}_\bk^n= \wt{\frk H}_\bk$ in the $D_\bk$-HPU sense. Then the curves $\wt\xi_\bk^n$ and $\wt\eta_\bk^n$ converge to $\wt\xi_\bk$ and $\wt\eta_\bk$, respectively, in the space $(W_\bk , D_\bk )$.  In this section we show that the analogous convergence holds for the curves which are all embedded into the same space $(W,D)$. 
\begin{lem} \label{lem-path-conv} 
	$\xi_\bk^n \rta \xi_\bk$ and $\eta_\bk^n \rta \eta_\bk$ a.s.\ as $n\rta\infty$ in $D$-uniform distance for each  $\bk \in \bigcup_{m=0}^\infty \BB N^m$.
\end{lem}

To prove Lemma~\ref{lem-path-conv}, for $n\in \N$ let \(f_\bk := \iota_\bk^{-1} : \wt H_\bk \rta W\) and \(f_\bk^n := (\iota_\bk^n)^{-1} : \wt \Map_\bk^n \rta W\), 
which take us from the embedding of $H_\bk$ (resp.\ $\Map_\bk^n$) into $W_\bk$ to its embedding into $W$. 
Let $g_\emptyset$ (resp.\ $g_\emptyset^n$) be the identity map on $H$ (resp.\ $\Map^n$) and for $\bk\not=\emptyset$, let
\eqbn
g_\bk := \iota_{\bk^-} \circ \iota_\bk^{-1} : \wt H_\bk \rta W_{\bk^-}
\quad \op{and} \quad
g_\bk^n := \iota_{\bk^-}^n \circ (\iota_\bk^n)^{-1} : \wt \Map_\bk^n \rta W_{\bk^-} 
\eqen
where  $\iota_{\bk^-}$ (resp.\ $\iota^n_{\bk^-}$) is considered as a map from $H_\bk$ (resp.\ $\Map^n_\bk$) to $W_{\bk^-}$ under the natural inclusion  $H_\bk \rta H_{\bk^-}$ (resp.\ $\Map_\bk^n \rta \Map_{\bk^-}^n$).
Then $f_\bk$ (resp.\ $g_\bk $) is 1-Lipschitz from $(\wt H_\bk , \wt d_\bk)$ to $(W,D)$ (resp.\ $(W_{\bk^-} , D_{\bk^-})$ and its image is $H_\bk$ (resp.\ $\iota_{\bk^-}(H_\bk)$). Furthermore, $f_\bk$ pushes forward $\wt\mu_\bk$ to $\mu_\bk$, $\wt\xi_\bk$ to $\xi_\bk$, and $\wt\eta_\bk$ to $\eta_\bk$; and
\eqb \label{eqn-map-decomp}
f_\bk = g_{\emptyset} \circ g_{\bk^{-(m-1)}} \circ \dots \circ g_{\bk^-}  \circ g_\bk .
\eqe 
Analogous statements hold for $f_\bk^n$ and $g_\bk^n$.  We start by proving a  limit result for  $f_\bk^n$ and $g_\bk^n$.

\begin{lem} \label{lem-map-limit}
	Almost surely, for each subsequence of $\mcl N$ of $\N$ there is a further subsequence $\mcl N'$ such that for each $\bk\in \bigcup_{m=1}^\infty\BB N^m$ the maps $f_\bk^n$ converge to $f_\bk$ and the maps $g_\bk^n$ converge to $g_\bk$ in the following sense as $\mcl N'\ni n\rta\infty$. 
	For each subsequence $\mcl N''$ of $\mcl N'$ and each sequence of points $x^n \in \wt \Map_\bk^n$ for $n\in\mcl N''$ such that $x^n\rta x\in \wt H_\bk$, one has $D(f_\bk^n(x^n) , f_\bk(x)) \rta 0$ and $D_{\bk^-}(g_\bk^n(x^n) , g_\bk(x)) \rta 0$ as $\mcl N'' \ni n \rta\infty$.
\end{lem} 
\begin{proof}
	By~\cite[Lemma 2.1]{gwynne-miller-uihpq} and a diagonalization argument, it is a.s.\ the case that for each subsequence of $\N$ there exists a further subsequence $\mcl N'$ and maps $\rng f_\bk : \wt H_\bk \rta W$ and $\rng g_\bk : \wt H_\bk \rta W_{\bk^-}$ for each $\bk\in \bigcup_{m=1}^\infty\BB N^m$ such that $f_\bk^n \rta \rng f_\bk$ and $g_\bk^n \rta \rng g_\bk$ in the sense described in the statement of the lemma as $\mcl N'\ni n \rta\infty$. 
	
	Repeating verbatim the proof of~\cite[Lemma 7.5, assertions 2 and 3]{gwynne-miller-perc} shows that each $\rng g_\bk$ is an isometry from $(\wt H_\bk \setminus \bdy \wt H_\bk , \wt d_\bk) $ to $\rng g_\bk(\wt H_\bk \setminus \bdy \wt H_\bk)$, equipped with the internal metric of $\wt d_{\bk^-}$; and pushes forward $\wt\mu_\bk$ to $\wt\mu_{\bk^-}|_{ \wt H_\bk   }$. Repeating verbatim the proof of~\cite[Lemma 7.6, assertion 1]{gwynne-miller-perc} shows that in fact $\rng f_\bk(\wt H_\bk \setminus \bdy \wt H_\bk) $ is a connected component of $\wt H_{\bk^-} \setminus \wt\eta_{\bk^-}$, so must be equal to $\iota_{\bk^-}(H_\bk)$. 
	From this and the uniqueness statement~\cite[Proposition 7.3]{gwynne-miller-char} for isometries of Brownian surfaces, we infer that $\rng g_\bk = g_\bk$. 
	By~\eqref{eqn-map-decomp}, the analogous relation for $f_\bk^n$ and $g_\bk^n$, and the convergence statement for the $g_\bk^n$'s, we find that also $f_\bk^n = f_\bk$.
\end{proof}

\begin{proof}[Proof of Lemma~\ref{lem-path-conv}]
	Fix $\bk \in \bigcup_{m=0}^\infty \BB N^m$. 
	By our choice of coupling and isometric embeddings, one has $\wt\xi_\bk^n \rta \wt\xi_\bk$ and $\wt\eta_\bk^n \rta\wt\eta_\bk$ uniformly with respect to the metric $D_\bk$. 
	The maps $f_\bk^n : (\wt \Map_\bk^n , \wt d_\bk^n ) \rta (W,D)$ are 1-Lipschitz, so the curves $\xi_\bk^n = f_\bk^n \circ \wt\xi_\bk^n$ and $\eta_\bk^n = f_\bk^n \circ \wt\eta_\bk^n$ are equicontinuous in $(W,D)$. By the Arz\'ela-Ascoli theorem, it is a.s.\ the case that for every subsequence of $\N$, there exists a further subsequence $\mcl N_*$, a curve $\xi_{\bk,*}  : [0, \nu(\bdy H_\bk)] \rta H$, and a curve $\eta_{\bk,*} : [0,\infty) \rta H$ which is constant at $\infty$ such that $\xi_\bk^n \rta \xi_{\bk,*} $ and $\eta_\bk^n \rta \eta_{\bk,*}^n$ in the $D$-uniform distance as $\mcl N_*\ni n\rta\infty$. By Lemma~\ref{lem-map-limit}, after possibly passing to a further subsequence we can arrange that the maps $f_\bk^n$ converge to $f_\bk$ in the sense of Lemma~\ref{lem-map-limit} along the subsequence $\mcl N_*$. This implies in particular that for each $s\in \BB R$,  
	\eqbn
	\xi_\bk^n(s) = f_\bk^n(\wt\xi_\bk^n(s)) \rta f_\bk^n(\wt\xi_\bk(s)) = \xi_\bk(s)  
	\eqen 
	along $\mcl N_*$, and similarly for each $u\geq 0$ one has $\eta_\bk^n(u) \rta \eta_\bk(u)$. 
	Hence $\xi_{\bk,*} = \xi_\bk$ and $\eta_{\bk,*} = \eta_\bk$. 
\end{proof}

\subsection{Parameterizing by disconnected area}
\label{sec-sle6-area}
In this brief subsection we prove variants of the scaling limit results of the previous subsections where we parametrize the paths $\eta_{\BB k}^n$ and $\eta_{\BB k}$ by the accumulated areas of the monochromatic bubbles which they disconnect from the target point (so that the curves are constant on some intervals of time) instead of by quantum natural time. The reason for considering this choice of parametrization is that it is more closely connected to the peano curves $\acute\eta^n$ and $\eta'$.

For $\bk \in \bigcup_{m=0}^\infty \BB N^m$, recall from Section~\ref{sec-mating} that $[s_\bk , t_\bk]$ is the interval of times during which $\eta'$ is filling in $H_\bk$ and $\wh t_\bk$ is the time such that $\eta'(\wh t_\bk)=\wh x_\bk$. Note that $\wh t_\bk=t_\bk$ if and only if $H_\bk$ is dichromatic. 
For $t\in [s_\bk  , \wh t_\bk]$, if $\eta'(t)\in \eta_\bk(\R)$, let $u_\bk(t)=\sup\{u:\eta_\bk(u)=\eta'(t)\}$. If $\eta'(t)\notin \eta_\bk(\R)$, then $\eta'(t)$ must belong to a monochromatic bubble $H_{\bk'}$ for some $\bk'\in\cC_\bk$. Let $u_\bk(t)=u_\bk(s_{\bk'})$.
Let $ \eta_\bk^\mas: [s_\bk,\wh t_\bk]\to H_\bk$ be defined by 
\begin{equation}\label{eq:repara}
\eta_\bk^\mas(t) := \eta_\bk(u_\bk(t)) .
\end{equation}
Recall that the set of ancestor-free times $\ans(\wh t_\bk)$ is the same as the set of times $t$ for which $\eta'(t) \in \eta_\bk$, equivalently the times when $\eta'$ is not filling in a connected component of $H_\bk\setminus \eta_\bk$ (see \ref{item-mating3} in Proposition~\ref{prop:mating}). 
Therefore, $u_\bk$ is a nondecreasing continuous function on $[s_\bk,\wh t_\bk]$, which stays constant on each connected component of  $(s_\bk,\wh t_\bk) \setminus \ans(\wh t_\bk)$.

In the discrete setting, for $n\in\BB N$ and $\bk \in \bigcup_{m=0}^\infty \BB N^m$,  recall the (unscaled) time $S^n_\bk,\wh T^n_\bk$ and $T^n_\bk$ in Section~\ref{sec-disk-walk}.  
For $i\in [S_\bk^n , \wh T_\bk^n]$, let $j_\bk^n(i)=\inf\{j: \lambda_\bk^n(j) \in \acute{\lambda}^n([i,\infty)_\Z)\}$ and let   $\lambda_\bk^{\mas,n}(i) := \lambda_\bk^n(j_\bk^n(i))$.
Note that
\eqbn
\lambda_\bk^{\mas,n}(i)=\acute\lambda^n(i) ,\quad \forall i \in [S_\bk^n ,\wh T_\bk^n]_{\BB Z} \: \op{with} \: \acute\lambda^n(i) \in \lambda_\bk^n.
\eqen 
Extend the curve $\lambda_\bk^{\mas,n}$ from $[S_\bk^n , \wh T_\bk^n]_{\BB Z}$ to $[S_\bk^n ,\wh T_\bk^n]$ as in Section~\ref{sec-ghpu}. For $t \in [s_\bk^n ,\wh t_\bk^n]$, let 
\eqbn
\eta_\bk^{\mas,n}(t) :=  \lambda_\bk^{\mas,n}(\mcon n t)\qquad \textrm{and}\qquad u_\bk^n(t) := \tcon^{-1} n^{-3/4} j_\bk^n( \mcon n t) ,
\eqen
where here we extend $j_\bk^n$ to $[S_\bk^n , T_\bk^n]$ in such a way that $\eta_\bk^{\mas,n}=\eta_\bk^{n}\circ u_\bk^n$.
\begin{lem} \label{lem-time-change-conv}
	For each $\bk\in \bigcup_{m=0}^\infty\N^,$, {we have that} $\eta_\bk^{\mas,n} \rta \eta_{ \bk}^\mas $ uniformly in $(W,D)$ in probability.
\end{lem}
\begin{proof} 
	By \cite[Lemma~9.25]{bhs-site-perc}, {we have that} $(\acute{Z}^n, u_\bk^n)$ converges in law to $(Z', u_\bk)$.
	Therefore $u_\bk^n \rta u_\bk$ in probability  in our coupling.
	Now Lemma~\ref{lem-time-change-conv} follows from Lemma~\ref{lem-path-conv} and \eqref{eq:repara}.
\end{proof}

\subsection{Proof of Theorem~\ref{thm:main-precise}}
\label{sec-main-proof}

Now we carry out the strategy outlined at the beginning of Section~\ref{sec-sle6-conv}. Fix $M,K \in \BB N$. We will define curves $\eta_{M,K}'$ and $\acute\eta_{M,K}^n$ for $M,K \in \BB N$ in terms of the chordal curves $\eta^\mas_\bk$ and $\eta_\bk^{\mas, n}$, respectively. Let us start by defining $\eta_{M,K}'$. 
Let $I_{M,K}=\bigcup_{m=0}^M \bigcup_{\bk \in [1,K]^m_{\BB Z}} [s_\bk , \wh t_\bk]$.
For each $t\in I_{M,K}$,  let  $\bk'$ be such that $[s_{\bk'},\wh t_{\bk'}]$
is the shortest interval containing $t$ in  $\{[s_{\bk},\wh t_{\bk}]\}_{\bk \in [1,K]^m_{\BB Z}}$  and 
set $\eta_{M,K}'(t) := \eta_{\bk'}^\mas(t)$.
For each $t\in [0,\mu(H)] \setminus I_{M,K}$, 
set $\eta'_{M,K}(t)  := \eta_{M,K}'(\ol t)$, where $\ol t$ is the next time after $t$ which belongs to $I_{M,K}$.
Roughly speaking, $\eta_{M,K}'$ is obtained by concatenating finitely many of the chordal curves $\eta_\bk$. 

The curve $\acute{\eta}_{M,K}^n$ is defined similarly: let $I^n_{M,K}=\bigcup_{m=0}^M \bigcup_{\bk \in [1,K]^m_{\BB Z}} [s^n_\bk , \wh t^n_\bk]$.
For each $t\in I^n_{M,K}$,  let  $\bk'$ be such that $[s^n_{\bk'},\wh t^n_{\bk'}]$
is the shortest interval containing $t$ in  $\{[s^n_{\bk'},\wh t^n_{\bk'}]\}_{\bk \in [1,K]^m_{\BB Z}}$  and 
set $\acute{\eta}_{M,K}^n(t) := \eta_\bk^{\mas,n}(t)$.
For each $t\in [0,\mu^n(H)] \setminus I^n_{M,K}$, 
set $\acute{\eta}_{M,K}^n(t)  := \acute{\eta}_{M,K}^n(\ol t)$, where $\ol t$ is the next time after $t$ which belongs to $I^n_{M,K}$.

We note that $\eta_{M,K}'$ is  left continuous with right limits, but not  right continuous. 
Indeed, for $\bk \in  [1,K]_{\BB Z}^m$  such that $m\in [0,M]_\Z$ and $H_\bk$ is monochromatic, the curve $\eta_{M,K}|_{[\wh t_\bk ,t_\bk]}$ 
will ``jump" across strings of dichromatic bubbles $H_{\bk'}$ for $\bk' \in \mcl C_\bk \setminus [1,K]_{\BB Z}^{m+1}$ which are traced in order by $\eta'$. Note that this discontinuity always occurs at one of the times $\wh t_\bk$ for $\bk\in \bigcup_{m=0}^M [1,K]_{\BB Z}^m$. 
However, as we will eventually show  in  Lemma~\ref{lem-bead-diam}, these discontinuities are typically small when $M$ and $K$ are large.

We introduce the following notations for the next lemma. Given a metric space $(X,d)$ and a subset $A\subset X$, the $d$-diameter of $A$ is defined by $\op{diam}(A;d):=\inf\{d(x,y): x,y\in A \}$.
Suppose $(X,d)=(\Map^n,d^n)$ for some $n\in\N$. If $A$ is a subgraph of $\Map^n$, we write $\op{diam}(\cV(A);d)$ as $\op{diam}(A;d)$.
If $A\subset  \cE(\Map^n)$, then we write $\op{diam}(B;d)$ as $\op{diam}(A;d)$ where $B$ is the set of vertices which are endpoints of edges in $B$.
We write $\Map^n-E$ as $\Map^n-\acute\eta_{M,K}^n$  where $E$ is the collection of edges that are contained in the range of $\acute{\eta}^n_{M,K}$.

\begin{lem} \label{lem-component-diam}
	Almost surely, for each $\ep > 0$ there exists $M,K \in \BB N$ only depending on $\ep$ such that 
	\begin{align}
	\label{eqn-component-diam}
	&\op{diam}\left( U ; d \right) \leq \ep ,&&\text{$\forall$ connected component $U$ of $H\setminus \eta_{M,K}'$};\\
	\label{eqn-component-diam-n}
	&\op{diam}\left(  U^n  ; d^n \right) \leq \ep , &&\text{$\forall$ 2-connected component $U^n$ of $\Map^n-\acute\eta_{M,K}^n$ for all $n\in\N$}.
	\end{align}
\end{lem}
Lemma~\ref{lem-component-diam} together with Lemma~\ref{lem-bead-diam} below will allow us to show that when $M$ and $K$ are large, {we have that} 
$\eta_{M,K}'$ is close to $\eta'$ and $\acute\eta_{M,K}^n$ is close to $\acute\eta^n$ uniformly in $n$. The bound \eqref{eqn-component-diam} is an immediate consequence of the following lemma.
\begin{lem} \label{lem-sle6-decomp}
	Almost surely, for each $\ep > 0$ there exists random $M , K\in\BB N$ such that 
	each connected component of $H\setminus \eta_{M,K}'$ has $\mu$-mass at most $\ep \mu(H)$. 
\end{lem} 
\begin{proof}
	For $\ep\in (0,1)$, let $E(\ep)$ be the event that for each $M,K\in \BB N$, there exists a connected component of $H\setminus \eta_{M,K}'$ with $\mu$-mass larger than $\ep \mu(H)$.  
	By way of contradiction, we assume that there exists an $\ep$ with $\P[E(\ep)]>0$. 
	Fix  $M\in \N$. For each $m\in [0,M]_\Z$ and $\bk\in \N^m$ such that $\mu(H_\bk)>\ep \mu(H)$, let  $K_\bk\in \N$ be such that
	$\sum_{j=K_\bk}^{\infty}\mu(H_{(\bk,j)})<M^{-1}\ep\mu(H_\bk)$ (such a $K_\bk$ can be found since the union of the bubbles $H_{\BB k'}$ for $\BB k' \in \mcl C_{\BB k}$ is all of $H_{\BB k}$ up to a set of $\mu$-mass zero).
	Let $K_M'$ be the maximum over all such $K_\bk$'s, which is finite since there are only finitely many $\BB k \in \bigcup_{m=1}^M \BB N^m$ with $\mu(H_\bk)  > \ep \mu(H)$. 
	On the  event $E(\ep)$, there exists a connected component $U$ of $H\setminus \eta'_{M,K'}$ such that $\mu(U)>\ep \mu(H)$. The component $U$ must be a bubble $H_{\bk_M}$ for certain multi-index $\bk_M$. By our choice of $K_M'$, we must have  $\bk_M \in \bigcup_{m>M}\N^{m}$.  
	Let $z\in H$ be a point sampled according to $\mu$. Then $\liminf_{M\to\infty}\P[z\in H_{\bk_M}]\ge \ep\P[E(\ep)]>0$. 
	This contradicts Lemma~\ref{lem:area}.
\end{proof}
\begin{proof}[Proof of \eqref{eqn-component-diam} in Lemma~\ref{lem-component-diam}]
	Since $\eta'$ is continuous, almost surely there exists $\delta =\delta(\ep) >0$ such that $d(\eta'(t_1) ,\eta'(t_2)) \leq \ep$
	for each $t_1,t_2 \in [0,\mu(H)]$ with $|t_1-t_2| \leq \delta$.
	By Lemma~\ref{lem-sle6-decomp}, there a.s.\ exists $M,K \in \BB N$ such that each connected component of $H\setminus \eta_{M,K}'$ has $\mu$-mass at most $\delta$. 
	Thus each such connected component can be  filled in by $\eta'$ in a single interval of length at most $\delta$
	Therefore~\eqref{eqn-component-diam} holds for our choices of $M,K$.
\end{proof}
The bound \eqref{eqn-component-diam-n} does not directly follows from  \eqref{eqn-component-diam} and the scaling limit results that we have already established,
since we only have convergence of each $\frk M_\bk^n$ to $H_\bk$ individually, not uniformly over all $\bk \in \bigcup_{m=0}^\infty \BB N^m$, and there are infinitely many connected components of $H\setminus \eta_{M,K}'$.  
To prove \eqref{eqn-component-diam-n},
we first record a simple lemma reducing estimates for the diameter of the set itself to estimates for the diameter of its boundary, which is more amenable to analysis via the boundary length process (see Lemma~\ref{lem-bubble-2interval}).  
\begin{lem} \label{lem-pinch}
	Given  a subgraph $S$ of $\Map^n$, let $\bdy S$ be the graph such that $\cV(\bdy S)$ is the union of $\cV(S)\cap \cV(\bdy \Map^n)$ and $\{v\in\cV(S): \;v \textrm{ is adjacent to a vertex in }\cV(\Map^n)\setminus\cV(S)\}$; and $\cE(\bdy S)$ is the set of 
	edges in $\cE(S)$ with both endpoints in $\cV(\bdy S)$.	
	Almost surely, for each $\ep   \in (0,1)$ there exists a random $\delta >0$ such that  for each $n\in\N$ and each subgraph $S$ of $\Map^n$ with  $\op{diam}\left( \bdy S ; d^n \right) \leq \delta$, we have 
	$\op{diam}\left( S ; d^n \right) \leq \ep$. 
\end{lem}
\begin{proof}
	The proof is essentially the same as that of~\cite[Lemma 5.7]{gwynne-miller-perc}, but we give the details since the statement in our setting is slightly different.
	
	Suppose by way of contradiction that the statement of the lemma is false. It is clear that for each $\ep  >0$, we can find a $\delta>0$ which works for any fixed finite collection of $n\in\N$, so there must exist $\ep > 0$ and a subsequence $\{n_q\}_{q\in\BB N} \subset \N$ such that for each $q\in\BB N$, there exists a subgraph $S^{n_q} \subset \Map^{n_q}$ with $\op{diam}(S^{n_q} , d^{n_q}) \geq \ep$ and $\op{diam} (\bdy S^{n_q} , d^{n_q}) \leq 1/q$. 
	
	For $q\in\BB N$, choose $y^{n_q} \in \bdy S^{n_q}$. It is clear that a.s.\ $\liminf_{q \rta\infty} \op{diam}( \bdy \Map^{n_q} ; d^{n_q}) > 0$, so there exists $\zeta >0$ such that for large enough $q \in \BB N$, the $d^{n_q}$-diameter of $\Map^{n_q} \setminus S^{n_q}$ is at least $\zeta$. 
	
	For $\delta \in (0, (\ep \wedge \zeta)/100 )$, define 
	\eqbn
	V_\delta^{n_q} := S^{n_q} \setminus B_{4\delta}(y^{n_q}; d^{n_q}) \quad \op{and} \quad
	U_\delta^{n_q} := \Map^{n_q} \setminus \left( S^{n_q} \cup B_{4\delta}(y^{n_q}; d^{n_q})  \right) .
	\eqen
	For large enough $q\in \BB N$, the set $V_\delta^{n_q}$ (resp.\ $U_\delta^{n_q}$) has $d^{n_q}$-diameter at least $\ep/2$ (resp.\ $\zeta/2$).
	Furthermore, since $\op{diam}(\bdy S^{n_q} , d^{n_q} ) \leq 1/k$, it follows that for large enough $q \in \BB N$ the sets $\bdy S^{n_q} \subset B_{2\delta}(y^{n_q} ; d^{n_q})$ so the sets $V_\delta^{n_q}$ and $U_\delta^{n_q}$ lie at $d^{n_q}$-distance at least $\delta$ from each other. 
	
	By possibly passing to a further subsequence, we can find $y\in H$ closed sets $U_\delta , V_\delta \subset H$ for each rational $\delta \in (0, (\ep\wedge \zeta) /100)$ such that as $q \rta\infty$, a.s.\ $y^{n_q} \rta y$ and $U_\delta^{n_q} \rta U_\delta$ and $V_\delta^{ n_q} \rta V_\delta$ in the $D$-Hausdorff metric. 
	Then $U_\delta$ and $V_\delta$ lie at $d$-distance at least $\delta$ from each other and have $d$-diameters at least $\ep/2$ and $\zeta/2$, respectively. 
	Furthermore, we have $H = U_\delta\cup V_\delta \cup B_{4\delta}(y ; d)$. Sending $\delta\rta 0$, we see that removing $y$ from $H$ disconnects $H$ into two components. 
	But, $H$ a.s.\ has the topology of a disk~\cite{bet-disk-tight}, so we obtain a contradiction. 
\end{proof}

The following lemma together with the equicontinuity of the curves $\eta_\bk^n$ for $\bk \in \bigcup_{m=0}^M [1,K]_{\BB Z}^m$ will allow us to bound the diameters of the boundaries (with respect to the   $d^n$-metric on $\Map^n$) of all but finitely many 2-connected components of $\Map^n- \acute\eta_{M,K}^n$ simultaneously.

\begin{lem} \label{lem-bubble-2interval}
	Almost surely, for each $M,K \in \BB N$ and each $\delta >0$ there exists a random $r = r(M,K,\delta)\in (0,\delta)$ such that the following is true. 
	For  $n\in\N$ and $\bk\in \bigcup_{m=0}^M [ 1,K]_{\BB Z}^m $, if $U$ is a 2-connected component of $\Map_\bk^n- \lambda_\bk^n$ with  boundary length at most $\bcon  r\sqrt n$,
	then there exist no times $u_1^n,u_2^n,u_3^n$ such that $\lambda_\bk^n(\mcon nu_i^n)$ has an endpoint on $\cV(U)$ for $i=1,2,3$ and $u_1^n<u_2^n-\delta<u_3^n-2\delta$.
\end{lem}
\begin{proof}
	To simply the notion, we only prove the case when $\bk=\emptyset$ so that $(\Map^n_\bk,\lambda_\bk, \eta^n_\bk)=(\Map^n,\lambda^n, \eta^n)$. 
	For other $\bk$, since $M,K$ are finite, we can iterate the argument for $k=\emptyset$ to conclude.	
	The idea of the proof is that if the statements of the lemma were false, then $\eta^n$ would have to have an ``approximate triple point", but we know the limiting SLE$_6$ curve $\eta$ has no triple points by~\cite[Remark 5.4]{miller-wu-dim}. However, we work with boundary length processes instead of curves since we do not know a priori that $\eta^n$ does not have any complementary 2-connected components with small boundary length but large diameter. To be more precise,
	let $\sigma$ be the time when $\eta$ hits its target point $\wh x$. 
	We know from~\cite[Lemma 7.8]{gwynne-miller-perc} that there a.s.\ do not exist times $u_1 < u_2<u_3 \leq \sigma$ for which the left boundary length process satisfies
	\eqb \label{eqn-sle-bdy-triple}
	L(u_1) = L(u_2) = L(u_3) = \inf_{u\in [u_1,u_3]} L(u)\quad\textrm{or}\quad R(u_1) = R(u_2) = R(u_3) = \inf_{u\in [u_1,u_3]} R(u)
	\eqe  
	
	Assume by way of contradiction that the statement of the lemma is false. 
	Then we can find $\delta>0$ such that for each $q\in\N$ there exists $n_q \in \N$ and a 2-connected component $U^{n_q}$ of $\Map^{n_q}- \eta^{n_q}$ with boundary length at most $\bcon\sqrt n/q$  and $u_1^{n_q},u_2^{n_q},u_3^{n_q}$ such that $\lambda^n(\mcon nu_i^{n_q})$ has an endpoint on $\cV(U)$ for $i=1,2,3$ and $u_1^{n_q}<u_2^{n_q}-\delta<u_3^{n_q}-2\delta$. Define the time set $I_U^{n_q}$ as in~\eqref{eq:IU} but with $\lambda^n$ in place of $\lambda$. Without loss of generality, we may assume that $\mcon n u_1^{n_q}=\min I_U^{n_q}$ and  $\mcon n u_3^{n_q}=\max I_U^{n_q}$.
	By left/right symmetry, after possibly passing to a further subsequence we can assume without loss of generality that each $U^{n_q}$ lies to the left of $\eta^{n_q}$. 
	By~\eqref{eqn-I_U-property} and since $|L^{n_q}(s) - L^{n_q}(t)| \leq \bcon^{-1} n^{-1/2} \#\mcl E(\bdy U^{n_q}) \leq 1/q$ for each $s,t \in U^{n_q}$,
	\begin{equation}\label{eqn-sle-bdy-triple-approx}
	L^{n_q}(u^{n_q}_3) =\min\{L^{n_q}(t): t\in [u^{n^q}_1, u^{n^q}_3]\}\ge \max \{L^{n_q}(u^{n_q}_1),L^n(u^{n_q}_2)\} -1/q.	
	\end{equation}
	
	By compactness and \eqref{eqn-sle-bdy-triple-approx},  we can possibly passing along a further  subsequence  and assume that 
	there exits times $u_1 ,u_2,u_3 \in [0,\sigma]$ with $u_2 \in [u_1+\delta,u_3-\delta]$  and a number $\ell\in \BB R$ such that as $q \rta\infty$, 
	\eqbn
	u_i^{n_q} \rta u_i \quad \op{and} \quad
	L^{n_q}(u_i^{n_q}) \rta \ell ,\quad
	\forall i \in \{1,2,3\}.
	\eqen
	By the Skorokhod convergence $L^n\rta L$, a.s.\ for any $t\in (u_1,u_3)$  there exist $t^{n_q}\in (u^{n_q}_1, u^{n_q}_3)$ such that $L^n(t^{n_q})\rta L(t)$.
	Now \eqref{eqn-sle-bdy-triple-approx} implies $\ell\le L(t)$, for all such $t$. Hence \(\ell= \inf\{L(t): t\in [u_1,u_3] \}\), which contradicts~\eqref{eqn-sle-bdy-triple}.
\end{proof}

\begin{proof}[Proof of \eqref{eqn-component-diam-n} of Lemma~\ref{lem-component-diam}]
	Fix $\ep > 0$. By Lemma~\ref{lem-pinch}, there a.s.\ exists $\delta_0 = \delta_0(\ep) \in (0,\ep)$ such that for each $n\in\N$ and each subgraph $S$ of $\Map^n$ with $\op{diam}\left( \bdy S ; d^n \right) \leq \delta_0$, it holds that $\op{diam}\left( S ; d^n \right) \leq \ep$. As we proved earlier, there a.s.\ exist $M,K$ such that \eqref{eqn-component-diam} holds with $\delta_0/2$ in place of $\ep$.
	
	By Lemma~\ref{lem-path-conv}, the paths $\xi_\bk^n$ and $\eta_\bk^n$ for fixed $\bk\in \bigcup_{m=0}^\infty \BB N^m$ are equicontinuous as an $n$-indexed sequence. Hence we can a.s.\ find $\delta_2 = \delta_2(\delta_0,M,K) >0$ such that for each $n\in\N$ and each $\bk \in \bigcup_{m=0}^M [1,K]_{\BB Z}^m$, it holds for each $s,s' \in \BB R$ with $|s-s'| \leq \delta_2$ that $d_\bk^n(\xi_\bk^n(s) , \xi_\bk^n(s')) \leq \delta_0/3$ 
	and $d_\bk^n(\eta_\bk^n(s ) , \eta_\bk^n(s')) \leq \delta_0/3$. Set $\delta$ in Lemma~\ref{lem-bubble-2interval} to be $\delta_2$ and let $r=r(M,K,\delta_2)\in (0,\delta_2)$ be as in that lemma.  
	
	Each 2-connected component of $\Map^n-\acute\eta_{M,K}^n$ is equal to $\Map_{\bk'}^n$ for some $\bk' \in  \bigcup_{m=0}^M \bigcup_{\bk\in [1,K]_{\BB Z}^m} \mcl C_\bk$.
	Suppose $U$ is such a component with boundary length at most $\bcon r\sqrt{n}$  and choose $\bk \in \bigcup_{m=0}^M   [1,K]_{\BB Z}^m $ such that $U = \Map_{\bk'}^n$ for $\bk' \in \mcl C_\bk$.
	By our choice of $r$ from Lemma~\ref{lem-bubble-2interval},
	$\cV(\bdy U)$ is contained in the union of the endpoints of the edges in $\lambda^n_\bk([a_\bk^n  , a_\bk^n + \delta_2]) \cup \lambda^n([b_\bk^n - \delta_2 , b_\bk^n])$, where $a_\bk^n$ (resp.\ $b_\bk^n$) is the first (resp.\ last) time that $\lambda_\bk^n$ hits $\bdy U$, and an arc of $\bdy \Map^n$ with at most $\bcon r\sqrt n \leq \bcon \delta_2 \sqrt n$ edges. 
	By our choice of $\delta_2$ and the triangle inequality, {we have that} $\op{diam}(\bdy U;d_\bk^n)\leq \delta_0$. By our choice of $\delta_0$, this implies that $\op{diam}(U ; d^n) \leq \ep$. 
	
	By the Skorokhod convergence $Z_\bk^n \rta Z_\bk$ and since the components $\Map_{\bk'}$ are ordered in decreasing order by their boundary lengths, we can a.s.\ find $k_* \in \BB N$ such that for each $\bk \in \bigcup_{m=0}^M [1,K]_{\BB Z}^m$ and each $\bk' \in \mcl C_\bk$ whose last coordinate is at least $k_*$, the boundary length of $\Map_{\bk'}^n$ is at most $\bcon r\sqrt n$. By the preceding paragraph, it is a.s.\ the case that for large enough $n\in\BB N$, each such $\Map_{\bk'}^n$ has $d^n$-diameter at most $\ep$. 
	
	It remains to consider the finitely many 2-connected components of the form $\Map_{\bk'}^n$ for $\bk' \in  \bigcup_{m=0}^M \bigcup_{\bk\in [1,K]_{\BB Z}^m} \mcl C_\bk$ with the last coordinate of $\bk'$ less than or equal to $k_*$. Let $\mcl K$ be the set of such multi-indices $\bk'$. For each $\bk' \in \mcl K$, the rescaled boundary path $\xi_{\bk'}^n$ converges $D$-uniformly to $\xi_{\bk'}$, which is the boundary path of a connected component of $H\setminus \eta_{M,K}'$, so has $d$-diameter at most $\delta_0/2$.
	Therefore, for large enough $n\in\N$ it holds that
	\eqb \label{eqn-bad-component-bdy-diam}
	\op{diam}\left( \bdy \Map_{\bk'}^n ; d^n \right) \leq \delta_0 ,\quad \forall \bk' \in \mcl K .
	\eqe   
	By our choice of $\delta_0$, we see that~\eqref{eqn-bad-component-bdy-diam} implies~\eqref{eqn-component-diam-n}.
\end{proof}
Lemma~\ref{lem-component-diam} is not quite good enough for our purposes since the curves $\eta'$ (resp.\ $\acute\eta^n$) can trace arbitrarily long strings of small dichromatic bubbles without interacting with $\eta'_{M,K}$ (resp.\ $\acute\eta^n_{M,K}$). See Figure~\ref{fig-sle6-decomp}, right. So, we need to bound the maximal diameter of a string of dichromatic bubbles which do not intersect our approximating curves.

\begin{lem} \label{lem-bead-diam}
	Almost surely, for each $\ep > 0$ there exists a random $M_0\in \BB N$ such that for each $M\ge M_0$, there exist random $K,N\in\N$ (allowed to depend on $M$) satisfying the following conditions.
	For $m\in[0,M]_\Z$ and $\bk \in  [1,K]_{\BB Z}^m$, if  $H_\bk$ is monochromatic, and $I$ is a connected component of $[\wh t_\bk,t_\bk]\setminus(\bigcup_{\bk'\in [1,K]_{\BB Z}^{m+1}} [s_{\bk'},t_{\bk'}] )$, then $\op{diam}(\eta'(I);d)\le \ep$.
	Furthermore,
	for $m\in[0,M]_\Z$ and $\bk \in  [1,K]_{\BB Z}^m$, if  $\Map^n_\bk$ is visited by $\acute\lambda^n$ before $\wh \Be_{\bk^-}$, and $I^n$ is a connected component of 
	$[\wh T^n_\bk,T^n_\bk]_\Z\setminus(\bigcup_{\bk'\in [1,K]_{\BB Z}^{m+1}} [S^n_{\bk'},T^n_{\bk'}]_\Z)$, then $\op{diam}(\acute{\lambda}^n(I^n);d^n)\le \ep$ for $n\ge N$.
	(Here for a subset $A$ of $\Z$, the connected component of $A$ is with respect the adjacency relation $i\sim j$ if and only if $|i-j|=1$.)
\end{lem}
\begin{proof}
	The idea of the proof is that each connected component $\eta'(I)$ (resp.\ $\acute\eta^n(I^n)$) as in the statement of the lemma is close to the arc of $\bdy H_\bk$ (resp.\ $\bdy M_\bk^n$) which it intersects. The length of this boundary arc is small if $K$ is large, and its diameter can be bounded using equicontinuity of the boundary paths.
	
	Fix $\ep > 0$. 
	It is clear that
	increasing $M$ or $K$ cannot increase the maximal diameter of connected components of $H\setminus \eta_{M,K}'$; and the analogous statement holds for $\acute\eta_{M,K}^n$. 
	By  Lemma~\ref{lem-component-diam},  we can a.s.\ find $M_0 , K_0,N_0 \in \BB N$ such that as long as  $M\geq M_0,K\geq K_0$ and $n\ge N_0$,  we have 
	\eqb
	\label{eqn-component-diam2}
	\op{diam}\left( U ; d \right) \leq \frac{\ep}{4} ,\quad \text{$\forall$ connected component $U$ of $H\setminus \eta_{M,K}'$}; \quad \text{and}
	\eqe
	\eqb
	\label{eqn-component-diam-n2}
	\op{diam}\left(  U^n  ; d^n \right) \leq \frac{\ep}{4} ,\quad  \text{$\forall$ 2-connected component $U^n$ of $\Map^n- \acute\eta_{M,K}^n$}. 
	\eqe
	We will now fix $M\ge M_0$ and find $K$  satisfying  the conditions in Lemma~\ref{lem-bead-diam}.
	
	By Lemma~\ref{lem-path-conv}, there a.s.\ exists $\delta>0$ such that for each $\bk \in \bigcup_{m=0}^{M} [1,K_0]_{\BB Z}^m$ 
	and each $s_1,s_2 \in \BB R$ with $|s_1-s_2| \leq \delta$, the boundary paths satisfy $d (\xi_\bk (s_1) , \xi_\bk (s_2)) \leq \ep/4$ and $d^n(\xi_\bk^n(s_1) , \xi_\bk^n(s_2)) \leq \ep/2$ for each $n\in\N$.

	Every non-trivial arc of each $\bdy H_{\BB k}$ shares a non-trivial boundary arc with a connected component of $H_{\BB k}\setminus \eta_{\BB k}$. Hence we can a.s.\ find $K \geq K_0$ such that for each $\BB k\in \bigcup_{m=0}^{M_0} [1,K_0]_{\BB Z}^m$ and each arc $J$ of $\bdy H_{\BB k}$ with boundary length at least $\delta/2$, there exists $\BB k' \in \mcl C_{\BB k} \cap [1,K]_{\BB Z}$ such that $H_{\BB k'}$ shares a non-trivial arc with $J$. 
	
	From the Skorokhod convergence $Z_{\BB k}^n \rta Z_{\BB k}$ and the encoding of boundary intersections of $\eta_{\BB k}$ (resp.\ $\eta_{\BB k}^n$) in terms of the running infima of the coordinates of $Z_{\BB k}$ (resp.\ $Z_{\BB k}^n$) --- see Lemma~\ref{lem:boundary} --- we find that for large enough $n\in\N$, it holds for each $\BB k\in \bigcup_{m=0}^{M_0} [1,K_0]_{\BB Z}^m$ and each arc $J^n$ of $\bdy \Map_{\BB k}^n$ with rescaled boundary length at least $\delta $ that there exists $\BB k' \in \mcl C_{\BB k} \cap [1,K]_{\BB Z}$ such that $\Map_{\BB k'}^n$ shares a non-trivial arc with $J^n$. 
	
	If $I$ is a connected component of $[\wh t_\bk,t_\bk]\setminus(\bigcup_{\bk'\in [1,K]_{\BB Z}^{m+1}} [s_{\bk'},t_{\bk'}] )$ for $\BB k\in \bigcup_{m=0}^M [1,K]_{\BB Z}^m$, then $\eta'(I) \cap \bdy H_{\BB k}$ is a connected arc of $\bdy H_{\BB k}$ which does not share a non-trivial arc with $\bdy H_{\BB k'}$ for any $\BB k' \in \mcl C_{\BB k} \cap [1,K]_{\BB Z}$, so has boundary length at most $\delta$ by our choice of $K$. By our choice of $\delta$, the arc $\eta'(I) \cap \bdy H_{\BB k}$ has $d$-diameter at most $\ep/4$. The set $\eta'(I)$ is the union of closures of connected components of $H_{\BB k} \setminus \eta_{M,K}'$ whose boundaries intersect $\eta'(I) \cap \bdy H_{\BB k}$. By~\eqref{eqn-component-diam2}, each such component has $d$-diameter at most $\ep/4$, so $\op{diam}(\eta'(I)  ) \leq \ep$. This gives the continuum part of the lemma. The discrete part statement is obtained similarly, using~\eqref{eqn-component-diam-n2} and the last sentence of the preceding paragraph.
\end{proof}

\begin{proof}[Proof of Proposition~\ref{prop-peano-conv}]
	Fix $\ep >0$ and choose $M,K \in \BB N$ for which the conclusions of Lemmas~\ref{lem-component-diam} and~\ref{lem-bead-diam}  are both satisfied.
	We will now show that $\eta_{M,K}'$ (resp.\ $\acute\eta_{M,K}^n$) is a good approximation for $\eta'$ (resp.\ $\acute\eta^n$). 
	By the definition of $\eta_{M,K}'$, we have $\eta_{M,K}'(t) = \eta'(t)$ for each $t$ such that $\eta'(t) \in \eta'_{M,K}$. For each time $t$ such that $\eta_{M,K}'(t)$ does not lie in $\eta_{M,K}'$, either
	$\eta'(t)$ is contained in 
	a connected component of $H\setminus \eta_{M,K}'$ with $\eta_{M,K}'(t)$ on its boundary; 
	or $t$ is contained in  a connected component of $[\wh t_\bk,t_\bk]\setminus(\bigcup_{\bk'\in [1,K]_{\BB Z}^{m+1}} [s_{\bk'},t_{\bk'}] )$, 
	where $m\in[0,M]_\Z$, $\bk \in  [1,K]_{\BB Z}^m$, and   $H_\bk$ is monochromatic. In both cases, \(d\left( \eta_{M,K}'(t) , \eta'(t) \right) \leq  \ep\). To sum up, 
	$\BB d_{D}^{\op U}(\eta'_{M,K}, \eta')\le \ep$.
	Similarly, there exists $N \in \BB N$ such that $\BB d_{D}^{\op U}(\acute{\eta}^n_{M,K}, \acute{\eta}^n)\le \ep$
	 for  $n\ge N$.
	
	We claim that \(\limsup_{n \rta\infty} \BB d^{\op{Sk}}_D   \left(  \eta'_{M,K} , \acute\eta^n_{M,K} \right) \leq 2 \ep\), where $\BB d^{\op{Sk}}_D$ denotes the $D$-Skorokhod distance. Combined with the previous paragraph, we have  \(\limsup_{n \rta\infty} \BB d^{\op{Sk}}_D\left(\eta' , \acute\eta^n \right) \leq 4 \ep .\)
	Since $\ep  >0$ can be made arbitrarily small, we see that $\BB d_{D}^{\op {Sk}}(\acute\eta^n,\eta')\rta 0$.
	Since $\acute\eta^n$ and $\eta'$ are continuous, we will have $\BB d_{D}^{\op {U}}(\acute\eta^n,\eta')\rta 0$, as desired.
	
	Since the interval endpoints satisfy $s_\bk^n \rta s_\bk$ and $\wh t_\bk^n \rta \wh t_\bk$ (Lemma~\ref{lem:endpoint1}) and the $s_\bk$'s and $\wh t_\bk$'s for $\bk \in \bigcup_{m=0}^M [1,K]_{\BB Z}^m$ are distinct, we can find for $n\in\N$ an increasing homeomorphism $\phi^n : [0,\mu(H)] \rta [0,\mu^n(\Map^n)]  $ which satisfies $\phi^n(s_\bk  ) = s_\bk^n$ and $\phi^n(\wh t_\bk ) = \wh t_\bk^n$ for $\bk \in \bigcup_{m=0}^M [1,K]_{\BB Z}^m$ and which converges uniformly to the identity map $[0,\mu(H)] \rta [0,\mu(H)]$ as $ n \rta\infty$.  
	By Lemma~\ref{lem-time-change-conv}, we have $\eta_\bk^{\mas,n} \circ\phi^n \rta \eta_\bk^\mas$ uniformly on $[s_\bk , \wh t_\bk]$ for each $\bk  \in \bigcup_{m=0}^M [1,K]_{\BB Z}^m$. From this and the definitions of $\eta_{M,K}'$ and $\acute\eta_{M,K}^n$, we see that $\acute\eta_{M,K}^n\circ\phi^n \rta \acute\eta_{M,K}^n$ uniformly on the set 
	\eqbn
	A := \bigcup_{m=0}^M \bigcup_{\bk \in [1,K]_{\BB Z}^m} [s_\bk , \wh t_\bk] . 
	\eqen
	For each $t\in [0,\mu(H)]\setminus A$, we have $\eta_{M,K}'(t) = \eta_{M,K}'(\ol t)$, where $\ol t$ is the rightmost point of $A\cap [0,t]$.  By Lemma~\ref{lem-bead-diam}, 
	$D(\eta'_{M,K}(t), \eta'_{M,K}(\ol t )) \le \ep$.
	Similar considerations hold for $\acute\eta_{M,K}^n\circ\phi^n$ off of $A$. We thus obtain
	\(\limsup_{n \rta\infty} \BB d^{\op{Sk}}_D   \left( \eta'_{M,K} , \acute\eta^n_{M,K} \right) \leq 2 \ep\).
\end{proof}
\begin{remark}
	At a first glance, it might appear possible to prove Theorem~\ref{thm-metric-peano} more directly in the following manner. 
	If one can establish tightness of $\acute{\frk M}^n$ with respect to the GHPU topology, then by the Prokhorov theorem $(\acute{\frk M}^n, \acute Z^n)$ admits subsequential limits. 
	Each such subsequential limit $(\wt{\frk H}' , \wt Z')$ is a Brownian disk decorated by a space-filling curve together with a correlated two-dimensional Brownian motion which describes (in some sense) the evolution of the left and right boundary lengths of the curve. 
	One can then try to argue that this space-filling curve has to be space-filling SLE$_6$. 
	
	This does \emph{not} follow from~\cite[Theorem 1.11]{wedges}, which implies that the space-filling SLE$_6$-decorated Brownian disk $\frk H'$ is a.s.\ determined by its left/right boundary length process $Z'$. 
	The reason is that~\cite[Theorem 1.11]{wedges} only shows that $\frk H'$ is given by a \emph{non-explicit} measurable function of $Z'$, not that any curve-decorated Brownian disk whose left/right boundary length process agrees with $Z'$ has to be equal to $\frk H'$. 
	
	The paper~\cite{gwynne-miller-char} gives conditions under which a curve on a Brownian surface with known left/right boundary length process is in fact a form of SLE$_6$.
	One of the results of~\cite{gwynne-miller-char} is used to identify a subsequential limit of random planar maps decorated by a single percolation interface in~\cite{gwynne-miller-perc}. It is likely possible to prove Theorem~\ref{thm-metric-peano} by first establishing tightness of $\acute{\frk M}^n$, then checking the hypotheses of the space-filling SLE$_6$ version of the result of~\cite{gwynne-miller-char} (see~\cite[Theorem 7.1]{gwynne-miller-char}). 
	However, it appears that deducing Theorem~\ref{thm-metric-peano} from the case of a single interface, as we do here, is easier.
\end{remark}

\section{Consequences of the main result}
\label{sec-CLE}

In this section we prove some consequences of Theorem~\ref{thm:main-precise} mentioned in Section~\ref{sec:app}.
In Sections~\ref{subsec:ghpul} and~\ref{subsec:CLE-conv} we  provide the precise statement and the proof of  Theorem~\ref{thm:loop-vague}.  
In Section~\ref{subsec:pivotal} we prove results on  the scaling limit of pivotal points.

\subsection{The GHPUL topology}
\label{subsec:ghpul}
Given a metric space $(X, d)$,  an \emph{unrooted oriented loop} on $X$  is a continuous map from the circle to $X$ identified up to reparametrization by orientation-preserving homeomorphisms of the circle.
Define the pseudo-distance  between two continuous maps from the circle $\BB R/\Z$ to $X$  by
\begin{equation}\label{emb-eq:dist-unrooted}
\BB d^{\op u}_{d} (\ell,\ell')= \inf_{\psi} \sup_{t\in \R /\Z} d(\ell(t), \ell'(\psi(t)),
\end{equation}
where the infimum is taken over all orientation-preserving homeomorphisms $\psi : \BB R/\BB Z\rta\BB R/\BB Z$.

A closed set of unrooted oriented loops on $X$ with respect to the $\BB d^{\op u}_d$-metric is called a \emph{loop ensemble} on $X$.  
We let  $\cL(X)$ be the space of loop ensembles on $X$ and  consider the function
\[
\wh {\BB d}^{\op L}_{d}  (c,c')=\inf\{\ep>0: \forall \ell \in c, \exists \ell' \in c'\,\textrm{such that} \;\BB d^{\op u}_{d}(\ell,\ell')\le \ep\}\qquad \forall\;  c,c'\in \cL(X).
\] 
Then $\BB d^{\op L}_{d} $ defined by 
\begin{equation}\label{emb-eq:dist-ensemble}
\BB d^{\op L}_{d} (c,c')=\max\{ \wh {\BB d}^{\op L}_{d}  (c,c'),  \wh {\BB d}^{\op L}_{d}  (c',c) \}, \qquad \forall\; c,c'\in \cL(X)
\end{equation} 
is a metric on $\cL(X)$. Let $\BM^\GHPUL$ be the set of $5$-tuples $\frk X  = (X , d,  \mu , \eta, c)$, where $(X,d)$ is a compact metric space, $\mu$ is a finite Borel measure on $X$, $\eta \in C_0(\BB R,X)$, and $c\in \cL(X)$.   If we are given elements $\frk X^1  = (X^1 , d^1,  \mu^1 , \eta^1,c^1)$ and $\frk X^2  = (X^2 , d^2,  \mu^2 , \eta^2,c^2)$  of $ \BM^\GHPUL$ and isometric embeddings 
$\iota^1 : (X^1 , d^1) \rta (W,D)$ and $\iota^2 : (X^2 , d^2) \rta  (W,D)$ for some metric space $(W,D)$, we define the \emph{GHPU-Loop (GHPUL) distortion} of $(\iota^1,\iota^2)$ by
\[
\op{Dis}_{\frk X^1,\frk X^2}^\GHPUL\left(W,D , \iota^1, \iota^2 \right)   
:= \op{Dis}_{\frk X^1,\frk X^2}^\GHPU\left(W,D , \iota^1, \iota^2 \right)    +  \BB d^{\op{L}}_D \left(\iota^1(c^1) , \iota^2(c^2) \right) ,
\]
where $\op{Dis}_{\frk X^1,\frk X^2}^\GHPU(\cdot)$ is the GHPU distortion as defined in~\eqref{eqn-ghpu-var}.
The \emph{GHPU-Loop (GHPUL) distance} between $\frk X^1$ and $\frk X^2$ is given by
\[
\BB d^\GHPUL\left( \frk X^1 , \frk X^2 \right) 
= \inf_{(W, D) , \iota^1,\iota^2}  \op{Dis}_{\frk X^1,\frk X^2}^\GHPUL\left(W,D , \iota^1, \iota^2 \right),
\]
where the infimum is over all compact metric spaces $(W,D)$ and isometric embeddings $\iota^1 : X^1 \rta W$ and $\iota^2 : X^2\rta W$.	
{It can be proved following the argument of \cite[Proposition 1.3 and Section 2.2]{gwynne-miller-uihpq} that} $(\BM^\GHPUL,d^\GHPUL)$ is a complete separable metric space. 

Let $(H,d,\mu,\xi)$ be a Brownian disk as in Theorem~\ref{thm:main-precise} and let $\Gamma$ be a CLE$_6$ on $(H,d,\mu,\xi)$ as in Definition~\ref{def:CLE}, where each loop is viewed as an unrooted oriented loop by forgetting the parametrization. The closure of $\Gamma$ under the $\BB d^{\op u}_d$-metric is given by $\Gamma$ together with all points in $\D$ identified as trivial loops (see the the last paragraph in \cite[Section 3]{camia-newman-sle6}). In this sense, we view $\Gamma$ as an random variable in $\cL(H)$ throughout this section.

\subsection{The scaling limit of the loop ensemble}
\label{subsec:CLE-conv} 
Suppose $(\Map,\Be)$ is a triangulation with simple boundary and  $\omega$ is a site percolation on $(\Map,\Be)$ with monochromatic boundary condition. 
Let $\cC$ be  a non-boundary cluster of $\omega$, as defined in Section~\ref{sec:app}. Let $\neg\cC$ be  the connected component of $\cV(\Map)\setminus \cV(\cC)$ containing $\bdy\Map$.
The \emph{filled cluster}  $\ol\cC$ of $\cC$ is the largest subgraph of $\Map$ such that $v\in \cV(\cC)$ if and only if  $v\notin \neg \cC$. See Figure \ref{fig-tri}.
The \emph{outer boundary}  $\bdy\cC$ of $\cC$ is the largest subgraph of $\ol\cC$ such that  $v\in \cV(\bdy\cC)$ if and only if $v$ is adjacent to a vertex on $\neg \cC$.  
The \emph{loop} $\gamma=\gamma(\cC)$ \emph{surrounding} $\cC$  is defined\footnote{In \cite{bhs-site-perc} $\gamma$ was called the \emph{outside-cycle} of $\C$.} to be the collection of edges with one endpoint in $\cC$ and the other in $\neg \cC$. 
Order the edges in $\gamma$ in the order of visits by the space-filling exploration $\acute{\lambda}$ of $(\Map,\Be,\omega)$. As shown in \cite[Lemma~5.11]{bhs-site-perc}, 
$\gamma$ is an edge path in the sense of Section~\ref{sec-ghpu} under this order.
Let $\Gamma(\Map,\Be,\omega)$ be the collection of all the loops on $(\Map,\Be)$ defined as above. 

Recall the setting in Theorem~\ref{thm:loop-vague}. 
Let $\Gamma^n=\Gamma(\Map^n,\Be^n,\omega^n)$.
If we identify each  $\gamma^n \in \Gamma^n$ with an edge path on $\Map^n$ and extend it to a  continuous curve as in Section~\ref{sec-ghpu}, then $\gamma^n$ can be viewed as  an unrooted oriented loop on $(\Map^n,d^n)$.
Hence $\Gamma^n$ is a loop ensemble on $\Map^n$ in the sense of Section~\ref{subsec:ghpul}. 
Now $(\Map^n,d^n,\mu^n,\xi^n,\Gamma^n)$ can be identified as an element in $\BM^\GHPUL$. 
Now we are ready  to state the precise version of Theorem~\ref{thm:loop-vague}.

\begin{thm}\label{thm:loop}
	$(\Map^n,d^n,\mu^n,\xi^n,\Gamma^n)$ converges in law to  $(H,d,\mu,\xi,\Gamma)$ with respect to the GHPUL topology. Moreover, this convergence occurs jointly with the convergence of Theorem~\ref{thm:main-precise}.
\end{thm}

In the rest of this section, we will work in the setting of Proposition~\ref{prop-peano-conv}, so that $\{\acute{\frk M}^n \}_{n\in \N}$ is a sequence of percolated triangulations each equipped with its space-filling exploration path and $\frk H'$ is a Brownian disk decorated by a space-filling SLE$_6$. By the Skorokhod representation theorem, we can couple $\{\acute{\frk M}^n \}_{n\in \N}$ and $\frk H'$ such that the convergence in Theorem~\ref{thm:main-precise}  and Proposition~\ref{prop-component-conv} holds almost surely. Moreover,  we use Proposition~\ref{prop-ghpu-embed} to isometrically embed $\{\acute{\frk M}^n\}_{n\in \N}$ and $\frk H'$ into a metric space $(W,D)$ so that $\acute{\frk M}^n \rta \frk H'$ in the $D$-HPU sense.

\begin{proof}[Proof of Theorem~\ref{thm:loop}]  
	For $\ep>0$, let $\Gamma_\ep$ be the set of loops in $\Gamma$ with $D$-diameter larger than $\ep$.  
	Define $\Gamma^n_\ep$ for $n\in\BB N$ similarly. 
	Note that $ \BB D^{\op L}_{d^n}(\Gamma^n\setminus \Gamma^n_\ep,\Gamma\setminus \Gamma_\ep)\le 2\ep+o_n(1)$ for each $n\in \N$.
	It suffices to show that  in our coupling $\lim_{n\to\infty} \BB d^{\op L}_D(\Gamma^n_\ep,\Gamma_\ep)=0$   almost surely for each fixed $\ep > 0$.  
	Recall the bubbles $H_{\bk} \subset H$ for $\bk \in \bigcup_{m=0}^\infty \BB N^m$ from Section~\ref{sec-mating}. For $\bk \in \bigcup_{m=0}^\infty \BB N^m$, define
	\begin{equation}\label{def:Gammak}
	\Gamma(\bk)\defeq\{\gamma\in\Gamma:   \textrm{ $H_\bk$ is the smallest bubble containing $\gamma$}\} \qquad\textrm{and}\qquad \Gamma_\ep(\bk)\defeq\Gamma(\bk)\cap \Gamma_\ep.
	\end{equation}
	Here we emphasize that $H_\bk$ is closed, and that a loop $\gamma \in \Gamma(\bk)$ may intersect $\bdy H_\bk$.
	Define $\Gamma^n_\ep(\bk)\subset \Gamma^n(\bk)\subset \Gamma^n$ similarly with $\Gamma$ replaced by $\Gamma^n$ and $H_\bk$ replaced by the triangulation $\Map^n_\bk$ of Section~\ref{sec-sle6-def-discrete}. We claim that $\Gamma(\bk)\neq \emptyset$ if and only if $H_\bk$ is monochromatic and, furthermore, that $\gamma\in \Gamma(\bk)$ if and only if $\gamma=\gamma_{\bk'}$ for some $\bk'\in \cC_\bk$ where $H_{\bk'}$ is dichromatic (recall that $\gamma_{\bk'}$ is defined in Section~\ref{sec-cle-def}). Indeed, if all the $\gamma_{\bk'}$'s are removed from $H_\bk$, we are left with monochromatic bubbles smaller than $H_\bk$. Other loops in $H_\bk$ must be inside these smaller bubbles.

	Let $\BB K_\ep=\{\bk\in \bigcup_{m = 0}^\infty \N^m: \diam(H_\bk;D)\ge  \ep\}$. 
	Then by Lemma~\ref{lem-component-diam}, {we have that} $\#\BB K_\ep<\infty$ almost surely. Moreover, for $n$ large enough, {we have that} $\diam(M^n_\bk;D)\ge \ep$ if and only if $\bk\in \BB K_\ep$. 
	Since $\Gamma_\ep(\bk)\subset \bigcup_{\bk'\in\BB K_\ep} \Gamma_\ep(\bk')$ and $\Gamma^n_\ep(\bk)\subset \bigcup_{\bk'\in\BB K_\ep} \Gamma_\ep^n(\bk')$ almost surely for large enough $n$,  
	Theorem~\ref{thm:loop} will follow from the following assertion:   for each $\bk\in\BB K_\ep$, almost surely 
	\begin{align}
	&\lim_{n\to\infty} \#\Gamma^n_\ep(\bk)=\#\Gamma_\ep(\bk); \label{eq:loop-conv} \\
	&\lim_{n\to\infty} \BB d^{\op L}_D(\Gamma_\ep^n(\bk),\Gamma_\ep(\bk))=0 \textrm{ if }\Gamma(\bk)\neq \emptyset,\textrm{ namely, $H_\bk$ is monochromatic}. \label{eq:loop-conv1}
	\end{align}
	We first prove \eqref{eq:loop-conv}-\eqref{eq:loop-conv1} for $\bk=\emptyset$, which constitutes the main body of our proof.
	
	In the continuum, let $m\defeq\inf\{t\in \R: \eta_\emptyset(t)=\wh x_\emptyset\}$.
	For each $K\in \N$, let 
	\begin{align*}
	\Gamma^K(\emptyset)&\defeq \{\gamma_k:  k\in [1,K]_\Z,\; H_k\textrm{ is dichromatic}\}.\\
	\cI^K&\defeq\{t\in [0,m]: \eta_\emptyset(t)\notin \bigcup_{\gamma \in \Gamma^K(\emptyset)} \gamma\}
	\end{align*}
	By the definition of $\Gamma(\emptyset)$ in \eqref{def:Gammak}, {we see that} $\Gamma(\emptyset)$ is the set of loops that have a nontrivial segment traced by $\eta_{\emptyset}$. By the construction in Section~\ref{sec-cle-def}, {we have that} $\Gamma(\emptyset)=\cup_{K\in\N}\Gamma^K(\emptyset)$ and 
	$\bigcap_{K\in\N}\cI^K$ is the set of times where $\gamma_\emptyset$ visits the right boundary of $(H,x_\emptyset,\wh x_\emptyset)$. 
	If $u$ is uniform in $(0,1)$   independent of everything else, then almost surely $\gamma_\emptyset(um)$ is not on the boundary of $H$. 
	Therefore the Lebesgue measure of $\bigcap_{K\in\N}\cI^K$ is zero a.s., hence
	the Lebesgue measure of
	$\cI^K$ tends to $0$ as $K\to\infty$. By the continuity of  $\eta_\emptyset$ and since $H\setminus \eta_\emptyset$ has only finitely many connected components with diameter $>\ep$ for each $\ep> 0$ (this follows from the continuity of $\eta'$ --- see the proof of Lemma~\ref{lem-component-diam}), we may choose $K$ large enough such that  
	\begin{enumerate}[label=(\arabic{enumi})]
		\item\label{item:interface}
		for each connected component of $\cI^K$, the $D$-diameter of $\eta(I)$ is smaller than $\ep/2$; 
		\item\label{item:bubble}
		if $k>K$ and $H_k$ dichromatic, then $\diam(H_k;D)<\ep/2$.
	\end{enumerate}
	For this choice of $K$, each loop in $\Gamma(\emptyset)\setminus \Gamma^K(\emptyset)$ has $D$-diameter less than $\ep$, hence $\Gamma_\ep(\emptyset)\subset \Gamma^K(\emptyset)$.

	Now we turn our attention to the discrete. In the rest of the proof, the colors on $\cV(\Map^n)$ are with respect to $\omega^n$ unless otherwise stated. We  identify a loop in $\Gamma^n$ with the cluster it surrounds. We assume $n$ to be so large that for each  $k$ with $\gamma_k\in \Gamma^K(\emptyset)$,  
	we have $\type_k=\op{di}$ and $\omega^n$ is dichromatic on $(\Map^n_k,\Be^n_k)$.  For each such $k$,
	let $\cC^n_k$ be the cluster containing the blue vertices on $\bdy\Map^n_k$. Let $\Gamma^{n,K}(\emptyset)$ be the collection of such clusters, which is viewed as a subset of $\Gamma^n(\emptyset)$.

	We will now argue that loops in $\Gamma^{n,K}(\emptyset)$  admit an analogous description to loops in $\Gamma^K(\emptyset)$, and use this to show that they converge to the loops in $\Gamma^K(\emptyset)$. We recall the setting of Section~\ref{sec-sle6-def-discrete}. For each $k$ such that $\cC^n_k \in \Gamma^{n,K}(\emptyset)$,  let $\lambda^{n,2}$ be the percolation interface of $(\Map^n_k, \Be^n,\omega^n|_{\cV(\Map^n_k)})$ and $\wh \Be'_k$ be the target of $\lambda^{n,2}$.
	For $e=\Be_k$ or $\wh\Be'_k$, let $\tri(e)$ be the unique triangle containing $e$ that is visited by
	$\lambda^n_\emptyset$.  There are exactly two edges on $\tri(e)$ visited by $\lambda^n_\emptyset$.
	Let $\wh e_k$  and $e_k$ be the second edge on $\tri(\wh \Be'_k)$  and the first edge on $\tri(\Be_k)$  visited by  $\lambda^n_\emptyset$, respectively. Let $\lambda^{n,1}_k$ be the segment of $\lambda^n_\emptyset$ from $\wh e_k$ to $e_k$. Let $\gamma^n_k$ be the loop surrounding $\cC^n_k$.
	Then $\gamma^n_k$ is the concatenation of  $\lambda^{n,1}_\emptyset$ and $\lambda^{n,2}_\emptyset$. See Figure \ref{fig-tri}.
	\begin{figure}[ht!]
	\begin{center}
	\includegraphics[scale=1]{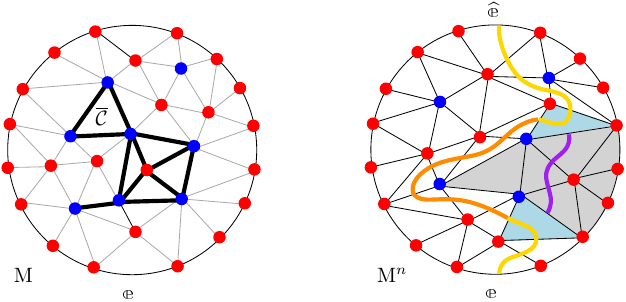} 
	\caption{\label{fig-tri} 
		{\bf Left}: Let $\cC$ be the cluster consisting of seven blue vertices. Then the filled cluster $\ol\cC$ is the map with bold edges.
		{\bf Right}: Illustration of objects in the proof of Theorem \ref{thm:loop}. The map $\Map^n_k$ is shown in gray, while tri$(\Be_k)$ (upper) and tri$(\wh\Be'_k)$ (lower) are shown in blue. The percolation cycle $\gamma_k^n$ is the concatenation of $\lambda_\emptyset^{n,1}$ (orange) and $\lambda_\emptyset^{n,2}$ (purple), while the percolation interface $\lambda^n_\emptyset$ is the concatenation of the yellow paths and the orange path.
	}
	\end{center}
	\end{figure}
	
	By Lemma~\ref{lem-interface-markov} and Proposition \ref{prop-component-conv}, {we have that} $\lambda^{n,2}$ converges to $\eta_k$ in the $\BB d^{\op u}_D$-metric, as defined in~\eqref{emb-eq:dist-unrooted}. Note that the target edge of $\lambda^{n,2}$ and $\lambda^n_k$ as defined in Section~\ref{sec-sle6-def-discrete} are different, but that the two target edges share a vertex. Therefore, the two interfaces have the same scaling limit. Recall $\bar \sigma_k=\inf\{s: \eta_\emptyset(s) \in \bdy H_k \}$ and $\bar\tau_k=\sup\{s:\eta_\emptyset (s) \in \bdy H_k \}$ defined in Section~\ref{sec-cle-def}.
	Let $\bar\sigma^n_k$ and $\bar\tau^n_k$ be such that $\lambda^n_\emptyset(\BB sn^{3/4} \bar\sigma^n_k)=\wh e_k$ and $\lambda^n_\emptyset(\BB s n^{3/4}\bar\tau^n_k)=e_k$.
	Since $\eta^n_\emptyset(\bar\sigma^n_k)\rta \eta_\emptyset(\bar\sigma_k)$,
	$\eta^n(\bar\tau^n_k)\rta \eta_\emptyset(\bar\tau_k)$, and both $\eta_\emptyset(\bar\sigma_k)$ and $\eta_\emptyset(\bar\tau_k)$ are visited by $\eta_\emptyset$ once, 
	we have   $(\bar\sigma^n_k,\bar\tau^n_k) \rta (\bar\sigma_k,\bar\tau_k)$ almost surely. 
	Therefore $\lambda^{n,1}$ converges to $\eta_\emptyset|_{[\bar\sigma_k,\bar\tau_k]}$ in the $\BB d^{\op u}_D$-metric. By Definition~\ref{def:CLE}, we have 
	\begin{equation}\label{eq:loop-limit}
	\textrm{$\BB d^{\op u}_D(\gamma^n_k,\gamma_k)\rta 0$ almost surely.}
	\end{equation}

	Let $m^n=\inf\{i\in\N_0: \lambda^n(i) =\wh \Be^n\}$ and 
	\[
	\cI^{n,K}\defeq\{i\in [0,m^n]_\Z: \lambda^n_\emptyset(i)\notin \bigcup_{\lambda \in \Gamma^{n,K}(\emptyset)} \lambda\}.
	\]
	Then by the previous paragraph the endpoints of the connected components of $\cI^{n,K}$ rescaled by $\BB sn^{3/4}$ converge to the endpoints of the connected components of $\cI^{K}$. 
	Choose $\delta$ such that  Lemma~\ref{lem-pinch} holds with  $\ep$ replaced by $\ep/2$.  
	Using the equicontinuity of $\eta^n_\emptyset$ for large enough $n$, by possibly increasing $K$, we can require that
	for each connected component $I$ of $\cI^{n,K}$,  the following hold: 
	\begin{enumerate}[label=(\arabic{enumi})]
		\item\label{item:interface1}
		The $D$-diameter of $\lambda_\emptyset^n(I)$ is less than $\delta/2$.
		\item\label{item:bubble1}
		Let $a,b$ be the endpoints of $I$. Let $v_a$  and $v_b$ 
		be the endpoints of $\lambda^n_\emptyset(a)$ and $\lambda^n_\emptyset(b)$  on $\bdy\Map^n$  respectively. 
		Then the $D$-diameter of the arc on $\bdy\Map^n$ between $v_a$ and $v_b$  is less than $\delta/2$.
	\end{enumerate}
	We claim that $\Gamma^n_\ep(\emptyset)\subset\Gamma^{n,K}(\emptyset)$ for such choice of $K$. For each $\cC\subset \Gamma^n(\emptyset)$,
	by definition the set $J(\cC):=\{j\in [0,m^n]_\Z: \lambda^n_\emptyset(j)\textrm{ has an endpoint on }\cC\}$ is nonempty.   Moreover, for  each $j\in J(\cC)$ the endpoint of $\lambda^n_\emptyset(j)$ on $\cC$  must be blue since the red endpoint is on the boundary cluster. If $\cC$ is surrounded by $\gamma_k$ for some $\gamma_k\in \Gamma^{n,K}(\emptyset)$, 
	then $J(\cC)=\{j\in [0,m^n]_\Z: \lambda^n_\emptyset(j) \in \gamma^n_k \}$.  Therefore, if $\cC\notin \Gamma^{n,K}(\emptyset)$, then $J(\cC)\subset \cI^{n,K}$. Consider a $j\in J(\cC)$ and 
	the connected component $I$ of $\cI^{n,K}$ containing $j$.  Let $v_a,v_b$ be defined as in \ref{item:bubble1} above. Since edges in $\{\lambda^n_\emptyset(i) \}_{i\in I}$ and boundary edges between $v_a$ and $v_b$ are not in $\cC$, it must be the case that 
	$\cC$ is in the region bounded by these edges. By our choice of $K$, we have $\diam(\cC;D)\le \ep/2$. Hence
	$\diam(\gamma^n;D)\le \ep$ if  $\dcon^{-1} n^{-1/4}<\ep/2$.  This gives  $\Gamma^n_\ep(\emptyset)\subset\Gamma^{n,K}(\emptyset)$. By \eqref{eq:loop-limit}, we obtain \eqref{eq:loop-conv}-\eqref{eq:loop-conv1} for $\bk=\emptyset$.

	As a byproduct of the argument above, for $n$ large enough, each blue cluster on $\Gamma^n(\emptyset)$ with a 
	vertex adjacent to the right endpoint of $\Be$ has diameter smaller than $\ep$. 
	Now given $k\in \N$ such that $H_k$ is dichromatic, choose $n$ so large that $\type_k=\op{di}$. Recall that $\omega^n$ and $\omega^n_k$ only differ at a single vertex $v$, which is an endpoint of $\wh \Be^n_k$. Therefore each cluster in $\Gamma^n(\bk)$ must have a vertex adjacent to $v$. 
	Therefore the same argument as for the aforementioned byproduct implies that the diameters of these clusters are smaller than $\ep$ for large enough $n$. Since $\Gamma(\bk)=\emptyset$, this proves \eqref{eq:loop-conv} for $\bk=k$. 
	
	By iterating  the argument above for each multi-index in $\BB K_\ep\cap (\N\cup \N^2)$ that corresponds to a monochromatic bubble,  we obtain \eqref{eq:loop-conv}-\eqref{eq:loop-conv1} for all $\bk\in \BB K_\ep$.  
\end{proof}

\subsection{The embedding from the mating-of-trees bijection}\label{subsec:relation}
Suppose $\{\acute{\frk M}^n \}_{n\in \N}$ and  $\frk H'$  are coupled such that $\acute Z^n\rta Z'$ almost surely, which is satisfied in our coupling.
In \cite[Section~7.2]{bhs-site-perc}, a sequence of mappings $\phi^n$ from $\cV(\Map^n)\cup \cE(\Map^n)$ to the unit disk $\D$ are considered. Identifying $\D$ with the Brownian disk $H$ (via~\eqref{eq:disk}), we obtain an embedding of
$\cV(\Map^n)\cup \cE(\Map^n)$ into $H$, which we still denote by $\phi^n$. 
Under $\{\phi^n \}$, several important percolation observables are proved to converge to their continuous counterpart, including the loop ensemble, the exploration tree, and the counting measure on  pivotal points. 
The precise definition of $\phi^n$ relies on more detailed information of the bijection in Section~\ref{sec-discrete}, which we will not review here. Here we list two properties of $\{\phi^n\}$ which follow from the definition and \cite[Lemma~9.20]{bhs-site-perc}, which specify $\phi^n$ up to an $o_n(1)$ error:
\begin{enumerate}[label=(\arabic{enumi})]
	\item \label{item:edge} For each $e\in \cE(\Map^n)$, let $t_e$ be such that $\acute\lambda^n(\mcon n t_e)=e$. Then almost surely
	\[
	\lim_{n\to \infty}\sup_{e\in \cE(\Map^n)}d(\phi^n(e), \eta'(t_e))=0.
	\]
	\item \label{item:vertex} \(\lim_{n\to \infty}\sup_{e,v}d(\phi^n(v), \phi^n(e))=0\) where $e,v$ in the $\sup$ range over edges and vertices on $\Map^n$ such that $v$ is an endpoint of $e$.  
\end{enumerate}

The following lemma is an immediate consequence of these two properties of $\phi^n$ and 
the almost sure convergence of the space-filling exploration $\acute{\eta}^n$ to $\eta'$ in $D$-uniform metric.

\begin{lem}\label{lem:compare}
	In the coupling above, almost surely
	$$
	\lim_{n\to \infty}\sup_{e\in \cE(\Map^n)} D(\phi^n(e), e)+\sup_{v\in \cV(\Map^n)}D(\phi^n(v),v)=0.
	$$
\end{lem}

Lemma~\ref{lem:compare} allows us to transfer all the convergence results in \cite[Theorem~7.2]{bhs-site-perc} to convergence in the $D$-metric in our coupling. Let us illustrate this in the setting of loop ensembles. For each $\gamma\in\Gamma^n$, let  $\reg(\gamma)=\ol\cC$ where $\cC$ is the cluster of $\omega^n$ surround by $\gamma$ and $\ol\cC$ is the filled cluster defined in Section~\ref{subsec:CLE-conv}.
Let $\area(\gamma)=\mu^n(\reg(\gamma))$.
For each $\gamma\in \Gamma$, let $\neg\gamma$ be the connected component of $H\setminus \gamma$ containing $\bdy H$. Let 
$\reg(\gamma)$ be the closure of the union of all connected components of $H\setminus\gamma$ other than $\neg\gamma$ where $\gamma$ visits the boundary in the same orientation as visiting $\bdy(\neg\gamma)$. Let $\area(\gamma)=\mu(\reg(\gamma))$. 
For $j\in\N$, let $\gamma^n_j$  (resp.\ $\gamma_j$)  be the loop in $\Gamma^n$ (reps. $\Gamma$) with the $j$-th largest area.  Then by  \cite[Theorem~7.2]{bhs-site-perc},  
$\lim_{n\to\infty}\BB d^{\op u}_D(\phi^n\circ \gamma^n_j, \gamma_j)=0$, with $\BB d_D^{\op u}$ as in~\eqref{emb-eq:dist-unrooted}.  
By Lemma~\ref{lem:compare}, {we have that} $\lim_{n\to\infty}\BB d^{\op u}_D(\phi^n\circ \gamma^n_j, \gamma^n_j)=0$ hence $\lim_{n\to\infty}\BB d^{\op u}_D(\gamma^n_j, \gamma_j)=0$. In fact, \cite[Theorem~7.2]{bhs-site-perc} gives the convergence $\gamma_j^n \rta \gamma_j$ in the $D$-uniform metric under the following parametrizations:  $\gamma^n_j$ is parametrized starting from the first edge visited by $\acute{\lambda}^n$ and rescaled as $\eta^n$, while $\gamma_j$ is as in Definition~\ref{def:CLE}.
Although this convergence is stronger than the $\BB d^{\op u}_D$-convergence for  individual loops, \cite[Theorem~7.2]{bhs-site-perc} does not imply Theorem~\ref{thm:loop} because it does not rule out the existence of sequence $\gamma^n\in\Gamma^n$ such that $\gamma^n$ encloses $o_n(1)$ units of $\mu^n$-area but $\diam(\gamma^n;D)$ is uniformly larger than a positive constant.  Our proof of Theorem~\ref{thm:loop} rules this out by the equicontinuity of $\acute{\eta}^n$.

\subsection{Pivotal measure and the color flipping at a pivotal point}
\label{subsec:pivotal}
In this subsection, we will explain why our results imply that the so-called \emph{pivotal measures} associated with $(\Map^n , \BB e^n ,\omega^n)$ converge to their continuum analogues under the coupling described just after Theorem~\ref{thm:loop}. This result will play an important role in the proof of the convergence of the Cardy embedding in~\cite{hs-cardy-embedding}; see Remark~\ref{rmk:dynamical}.

For $v\in \cV(\Map^n)\setminus\cV(\bdy \Map^n)$,  let $\Gamma^n_v$ be the loop ensemble associated with the percolation obtained from $\omega^n$  by flipping the color of $v$, and let $\cL^n_v$ be the symmetric difference of $\Gamma^n$ and $\Gamma^n_v$.  
For $\ep>0$, we say that $v$ is an \emph{$\ep$-pivotal point} of $(\Map^n,\Be^n, \omega^n)$ 
if  there are at least three loops in $\cL^n_v$ with  area (see the definition at the end of Section~\ref{subsec:CLE-conv}) at least $\ep$.
Morally speaking, $v$ is an $\ep$-pivotal point if flipping the color of $v$ results in some splitting or merging of clusters  of ``size'' at least $\ep$.

A point $v\in \D$ is called a pivotal point of $\Gamma$ if it is a  point of  intersection of at least two loops of $\Gamma$ or if it is visited at least twice by a loop in $\Gamma$.
As shown in \cite[Theorem 2]{camia-newman-sle6}, with probability 1, {we have that} $v\in\D$ is a pivotal point of $\Gamma$ if and only if one of the following two occurs:
\begin{enumerate}[label=(\arabic{enumi})]
	\item there exist exactly two loops $\gamma,\gamma'\in\Gamma$ such that $v\in\gamma\cap\gamma'$ and both $\gamma$ and $\gamma'$ visit $v$ once;
	\label{item:int}
	\item there exists a unique loop $\gamma\in\Gamma$ that visits $v$ and $\gamma$ visits $v$ exactly twice.  
	\label{item:double}
\end{enumerate}
By \emph{flipping the color} at $v$, we mean merging $\gamma,\gamma'$ into a single loop in Case~\ref{item:int} and splitting ${\gamma}$ into two loops in Case~\ref{item:double}. See Figure \ref{fig-flip}. If a loop does not visit $v$, flipping the color at $v$ keeps the loop unchanged.  Let $\Gamma_v$ be the set of loops obtained after flipping the color at $v$. Then it is easy to see that  the orientation of $\Gamma$ induces an orientation on each loop in $\Gamma_v$, making it an ensemble of unrooted oriented loops. Moreover,  the symmetric difference  $\cL_v$ of $\Gamma$ and $\Gamma_v$ always contains exactly  three loops.  We say that $v$ is an $\ep$-pivotal point of $\Gamma$ if each loop in $\cL_v$ has area at least $\ep$.

\begin{figure}[ht!]
	\begin{center}
		\includegraphics[scale=1]{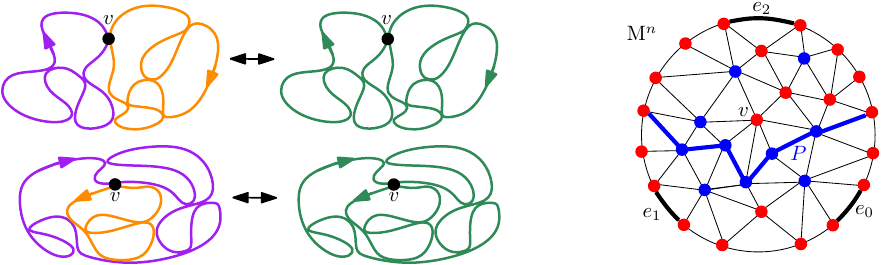} 
		\caption{ 
			{\bf Left}: Illustration of the operation of flipping the color at $v\in\D$. Each loop is drawn with a different color. In Case \ref{item:int} we go from left to right when we merge the two loops, while in Case \ref{item:double} we go from right to left when we split the single loop. The top (resp.\ bottom) row illustrates the case of non-nested (resp.\ nested) loops.
			{\bf Right}: Illustration of the event $E^n(e_0,e_1;e_2,v)$ considered in Section \ref{subsec:crossing}.
		}
		\label{fig-flip}
	\end{center}
\end{figure}

Let $\nu^n$ be the measure on $\cV(\Map^n)$ where each vertex is assigned mass $n^{-1/4}$.
For each $\ep>0$,  let $\nu^n_\ep$ be the restriction of $\nu^n$ to the $\ep$-pivotal points of $\Gamma^n$.  By \cite[Theorem~7.2]{bhs-site-perc}, there exists a random finite Borel measure $\nu_\ep$ supported on the set of $\ep$-pivotal points of $\Gamma$ such that  $(\phi^n)_* \nu^n$, which is the pushforward of $\nu^n_\ep$ under $\phi^n$,  converges to $\nu_\ep$ in probability. 

\begin{prop}\label{prop:pivot1}
	In our coupling $\nu^n_\ep$ converges to $\nu_\ep$ in probability with respect to the $D$-Prokhorov metric. Moreover,  $\mathbf{1}_{\nu^n_\ep=0}$ converges to  $\mathbf 1_{\nu_\ep=0}$ in probability.
\end{prop}
\begin{proof}
	Conditional on everything else, on the event that $\nu^n_\ep\neq 0$, let $\bpv^n$ be a uniformly sampled $\ep$-pivotal point of $\Gamma^n$.
	Moreover, on the event that $\nu_\ep\neq 0$, let $\bpv$ be a point sampled according to $\nu_\ep(\cdot)/ \nu_\ep(H)$.
	By \cite[Theorem~7.2]{bhs-site-perc}, we  may extend our coupling such that  almost surely $\mathbf{1}_{\nu^n_\ep=0}$ converges to  $\mathbf 1_{\nu_\ep=0}$, and, moreover,
	\[ 
	\lim_{n\to\infty} d(\phi^n( \bpv^n), \bpv)=0\qquad \textrm{and}\qquad \lim_{n\to\infty} \nu^n_\ep(\cV(\Map^n)) = \nu_\ep(H).
	\]  
	By Lemma~\ref{lem:compare}, we have \(\lim_{n\to\infty} D( \bpv^n, \bpv)=0\) almost surely.  This concludes the proof.
\end{proof}

We will now prove the convergence of the loop ensemble obtained by flipping the color of a single pivotal point. 
Let $\bpv^n$ and $\bpv$ be defined as in the proof of Proposition~\ref{prop:pivot1}. Suppose we are under the extended coupling considered there, where $D(\bpv^n,\bpv)\rta 0$.
Let $ \Gamma^n_{\bpv^n}$  (resp., $ \Gamma_{\bpv}$) be the loop ensemble obtained by flipping the color of $\bpv^n$ (resp., $\bpv$).  Let $\wh \nu^n_\ep$  be the restriction of $\nu$ to the $\ep$-pivotal points of $ \Gamma^n_{\bpv^n}$.
By \cite[Proposition~7.9]{bhs-site-perc}, there exists a random finite Borel measure $\wh \nu_\ep$ supported on the set of $\ep$-pivotal points of $\wh \Gamma$ such that   $(\phi^n)_*\wh \nu^n_\ep$  converges to $\wh \nu_\ep$ in probability.  By Lemma~\ref{lem:compare}, we have the following.

\begin{prop}\label{prop:pivot2}
	In the coupling above, {we have that} $\wh \nu^n_\ep \rta \wh\nu_\ep$ and $\Gamma^n_{\bpv^n}\rta \Gamma_{\bpv}$ in probability with respect to the $D$-Prokhorov topology and the topology on loop ensembles on $(W,D)$, respectively.
\end{prop}
\begin{proof}
	The proof of the first result is identical to Proposition~\ref{prop:pivot1}. The second convergence follows from Theorem~\ref{thm:loop}, Lemma~\ref{lem:compare}, and \cite[Proposition 7.11 ]{bhs-site-perc}, where it is shown that for each $\delta>0$, loops in $\Gamma^n_{\bpv^n}$ with area larger than $\delta$  under the pushforward of $\phi^n$ converge to loops in 
	$\Gamma_{\bpv}$ with area larger than $\delta$  in a topology stronger than $\BB d^{\op L}_D$.
\end{proof}

\begin{remark}\label{rmk:dynamical}
	The dynamical percolation on $\Map^n$ is the following Markov process: starting with  a sample of $\omega^n$, flip the colors on $\cV(\Map^n)$ according to the Poisson point process with intensity $\nu^n\otimes dt$, where $dt$ is the Lebesgue measure on $(0,\infty)$.
	As a consequence of Propositions~\ref{prop:pivot1} and~\ref{prop:pivot2}, it will be shown in \cite{hs-cardy-embedding} that the variant of  dynamical percolation on $\Map^n$ where only $\ep$-pivotal points are allowed to flip converges to its continuum analog.  
	As $\ep\rta 0$, this $\ep$-dependent limiting dynamic further converges to an ergodic Markov process called the \emph{Liouville dynamical percolation} introduced in \cite{ghss-ldp}. 
	It will also be shown that $\nu_\ep$ and $\wh{\nu}_\ep$ can be defined equivalently as the ``$\sqrt{8/3}$-LQG Minkowski content" of the set of pivotal points of $\Gamma$ and $\Gamma_{\bpv}$, respectively,.
\end{remark}

\subsection{Re-rooting invariance and crossing events}
\label{subsec:crossing}
If $\wt{\BB e}^n$ is another choice of root edge for $\Map^n$, then the order in which points are traced by the space-filling explorations of $(\Map^n , \BB e^n , \omega^n)$ and $(\Map^n , \wt{\BB e}^n , \omega^n)$ are different, and hence these processes give rise to different left/right boundary length processes. Here we will describe the joint scaling limit of these processes for any finite number of different root vertices. This allows us to show that certain ``crossing events" for site percolation on $\Map^n$ converge to their continuum analogs,
which in turn will be a key input in~\cite{hs-cardy-embedding} (see Remark~\ref{rmk:Cardy}).

Suppose $\omega$ is a site percolation on $(\Map,\Be)$ with monochromatic boundary condition and let $\Gamma(\Map,\Be,\omega)$ be  its  loop ensemble as defined in the first paragraph of Section~\ref{subsec:CLE-conv}.
Our definition of $\Gamma(\Map,\Be,\omega)$ there depends on $\Be$ because it uses the space-filling exploration $\acute{\lambda}$ of $(\Map,\Be,\omega)$. However, it is well known that for each $\gamma\in\Gamma$  the collection of dual edges on $\gamma$  forms a simple cycle on the dual map of $\Map$. Moreover, if each dual edge is oriented such that the red (resp.\ blue) vertex is on the left (resp.\ right), then the simple cycle is itself oriented.  
This gives an ordering of edges on $\gamma$ up to cyclic permutations, thus making $\gamma$ an unrooted oriented loop. Since
this definition of $\gamma$ agrees with the one in Section~\ref{subsec:CLE-conv} based on $\acute{\lambda}$,  we see that $\Gamma(\Map,\Be,\omega)$ only depends on $(\Map,\omega)$ but not on $\Be$.

\begin{lem}\label{re-rooting}
	In the setting of Theorem~\ref{thm:loop}, we further assume that  we have chosen our coupling so that  $(\Map^n,d^n, \mu^n, \xi^n, \Gamma^n)$ converges to $(H,d,\mu,\xi,\Gamma)$ almost surely  with respect to the GHPUL topology.
	Let $k\in \N$ and let $0<u_1<\cdots<u_k<1$.
	For $i\in [1,k]_\Z$, let $\Be^n_i=\beta^n(\lceil u_i\ell^n\rceil)$ where $\beta^n$ and $\ell^n$ are the boundary curve and boundary length of $(\Map^n,\Be^n)$, respectively. 
	For $i\in[1,k]_\Z$, let $\acute \lambda^n_i$ be the space-filling exploration of $(\Map^n, \Be^n_i, \omega^n)$ and let $\acute{\eta}^n$ and $\acute Z^n_i$ be the corresponding rescaled space-filling exploration and random walk, respectively.  Let $\eta'_i$ be such that $(H,\eta')\overset{d}{=}(H,\eta'_i)$, $\eta'_i(0)= \xi(u_i)$ and  the $\CLE_6$ corresponding to $\eta'_i$ agrees with $\Gamma$ as an element in $\cL(H)$. Let $Z'_i$ be the boundary length process of $\eta'_i$. Then  $(\Map^n,d^n, \mu^n, \xi^n, \acute \eta^n, \acute\eta^n_1,\cdots, \acute \eta^n_k)$ and $(\acute Z^n_i)_{i\in[1,k]_\Z}$ 
	jointly converge in probability  to $(H,d,\mu,\xi,\eta',\eta'_1,\cdots, \eta'_k)$ and $(Z'_i)_{i\in[1,k]_\Z}$ in the product topology of $\BM^{\GHPU}_{k+2}\times C_0(\R; \R^k)$.
\end{lem}
\begin{proof} 
	In Definition~\ref{def:CLE}, $\Gamma$ as an element in $\cL(H)$ and  $\eta'$ modulo monotone reparametrization determine each  other.  From our construction of $\Gamma$, it is clear to $\eta'$ determines $\Gamma$ a.s. On the other hand,  $\Gamma$ determines $\eta'$ a.s.\ since it determines the nested exploration.  Therefore $(H,d,\mu,\xi, \Gamma)$ determines $(\eta',Z')$  and $(\eta'_i,Z'_i)_{i\in[1,k]_\Z}$ a.s. By Theorem~\ref{thm:main-precise} and Footnote~\ref{fn:in-prob} we conclude the proof.
\end{proof}

Although the definition of the $k$ walks $(\acute Z^n_i)_{i\in[1,k]_\Z}$  is quite elementary, their coupling is not straightforward to understand without using the loop ensemble. 
This was the obstruction in \cite{bhs-site-perc} to proving the joint convergence of certain \emph{crossing event}s which we can prove here.  In the setting of Theorem~\ref{thm:loop}, suppose $e_0,e_1,e_2$ are three distinct edges on $\bdy\Map^n$ ordered clockwise.  For $i,j\in[0,2]_{\Z}$, we denote by $(e_i,e_{j})$ the set of boundary  vertices of $\Map^n$ situated between $e_i$ and $e_j$ in clockwise order  (including one endpoint of $e_i$ and one endpoint of $e_j$). For a vertex $v\in \cV(\Map^n)$, we denote by $E^n(e_0,e_1;e_2,v)$ the event that there exists a simple path (i.e., a sequence of distinct vertices on $\Map^n$ where any two  consecutive vertices are adjacent) $P$ on $\Map^n$ such that 
\begin{itemize}
	\item[(a)]  $P$ contains one endpoint in $(e_1,e_2)$ and one endpoint in $(e_2,e_0)$, while all other vertices of $P$ are inner blue vertices;  
	\item[(b)] either $v\in P$ or $v$ is on the same side of $P$ as the edge $e_2$.
\end{itemize}
See Figure \ref{fig-flip}.

In the setting of Lemma~\ref{re-rooting}, let $k=3$. Write $\Be^n$ as $e^n_0$ and set  $u_0=0$. Let $x_i=\xi(u_i)$ for $i\in [0,3]_\Z$. 
Let $u^n$ be a uniform integer from $[1,\#\cE(\Map^n)]_\Z$.  We further assume that $\{u^n \}_n\in \N$ is independent of everything else and $u^n/\#\cE(\Map^n)$ converge almost surely to a uniform variable  $u$ on $(0,1)$. Let $v=\eta(u)$.
Conditioning on $u^n$, let $v^n$ be a uniformly chosen endpoints of $\acute{\lambda}^n(u^n)$.  
Then  $v^n\in\cV(\Map^n)$ is sampled from $\mu^n(\cdot)/\mu^n(\Map^n)$ and $v\in\D$ is sampled from $\mu(\cdot)/\mu(H)$.
It is proved in \cite[Theorem~7.6]{bhs-site-perc}\footnote{In \cite{bhs-site-perc} the measure $\mu^n$ is defined such that each vertex on $\Map^n$ has mass $n^{-1}$ instead. However, it is easy to see that the 
	$d^n$-Prokhorov distance between these two measures tends to $0$ in probability, say, by Lemma~\ref{lem:compare} and the proof of~\cite[Lemma~9.22]{bhs-site-perc}.} that in such a coupling, where $\acute Z^n\rta Z'$ and $v^n\rta v$ almost surely, the events $E^n(e^n_0,e^n_1; e^n_2,e^n_3)$ and $E^n(e^n_0,e^n_1; e^n_2,v)$ converge in probability to some events $E(x_1,x_2;x_3,v)$ and $E(x_1,x_2;x_3,x_4)$.  These two events are defined in terms of $(Z',u)$ in \cite[Section~6.9]{bhs-site-perc}. By rotating the role of these boundary edges,  we have the following.
\begin{prop}\label{prop:cro}
	Under  the convention that $4=1$ and $5=2$, {the convergence in Lemma \ref{re-rooting} is joint with convergence of }
	the  indicators of  events $E^n(e^n_0,e^n_1; e^n_2,e^n_3)$, $E^n(e^n_1,e^n_2; e^n_3,e^n_0)$, and
	$\{E^n(e^n_i,e^n_{i+1}; e^n_{i+2},v)\}_{i\in[1,3]_\Z}$. 
\end{prop}
The \emph{joint} convergence {of the crossing events} in Proposition~\ref{prop:cro} is far from obvious using the techniques in \cite{bhs-site-perc}.
\begin{remark}[Cardy embedding]\label{rmk:Cardy}
	Under  the convention that $4=1$ and $5=2$, let
	\begin{equation}\label{eq:cdy}
	p^n_i(v)=\P\left[E^n(e^n_i,e^n_{i+1}; e^n_{i+2},v) \big|(\Map^n,\Be^n) \right]\qquad   \textrm{for }i=1,2,3 \textrm{ and } v\in \cV(\Map^n).
	\end{equation}
	Then  the function $\Cardy^n:\cV(\Map^n)\to \Delta:=\{ (x_1,x_2,x_3)\in  (0,\infty)^3\,:\,x_1+x_2+x_3=1 \}$ defined by $\Cardy^n=(p^n_1+p^n_2+p^n_3)^{-1}(p^n_1,p^n_2,p^n_3)$ is the so called \emph{Cardy embedding}. 
	When $\Map^n$ is replaced by the the intersection of the regular triangular lattice with a Jordan domain $\Omega$ with three boundary marked points, the result of Smirnov \cite{smirnov-cardy} is equivalent to the statement that the Cardy embedding converges to the  Riemann mapping from $\Omega$ to  $\Delta$, where the three marked points are mapped to the three vertices of $\Delta$. 
	In \cite{hs-cardy-embedding}, based on results mentioned in Remark~\ref{rmk:dynamical}, the second and third author will show that $\Cardy^n$ converges to the conformal embedding from $(H,x_0,x_1,x_3)$ to $\Delta$.
\end{remark}

\newcommand{\etalchar}[1]{$^{#1}$}
\def\cprime{$'$}

\end{document}